\documentclass[11pt,reqno]{article}%{amsart}
\usepackage{dsfont, amssymb,amsmath,amscd,latexsym, amsthm, extarrows, amsxtra,amsfonts}
\usepackage[all]{xy}
\usepackage[active]{srcltx}
\usepackage[pdfstartview=FitH, bookmarksnumbered=true,bookmarksopen=true, colorlinks=true, pdfborder=001, citecolor=blue, linkcolor=blue,urlcolor=blue]{hyperref}
\usepackage[round,authoryear]{natbib}
\usepackage{graphicx}  %插入图片的宏包
\usepackage{float}  %设置图片浮动位置的宏包
\usepackage{subfigure}

\usepackage{graphicx}
\usepackage[title]{appendix}
\usepackage{bm}

\usepackage[top=2.8cm, bottom=2.8cm, left=2.8cm, right=2.8cm]{
	geometry}
\newtheorem{theorem}{Theorem}[section]
\newtheorem{assumption}[theorem]{Assumption}

\newtheorem{definition}[theorem]{Definition}

\newtheorem{lemma}[theorem]{Lemma}

\newtheorem{proposition}[theorem]{Proposition}
\newtheorem{remark}[theorem]{Remark}
\numberwithin{equation}{section}

\newcommand{\F}{\mathcal{F}}
\newcommand{\Rbf}{\mathbf{R}}
\newcommand{\Ebf}{\mathbf{E}}
\newcommand{\Ac}{\mathcal{A}}
\newcommand{\rd}{\mathrm{d}}
\newcommand{\bbeta}{\pmb{\beta}}
\newcommand{\bpi}{\pmb{\pi}}

\def\trieq{\triangleq}
\DeclareMathOperator*{\range}{Range}
\DeclareMathOperator*{\var}{Var}
\DeclareMathOperator*{\cov}{Cov}

\begin{document}
\makeatletter
\def\@setauthors{%
\begingroup
\def\thanks{\protect\thanks@warning}%
\trivlist \centering\footnotesize \@topsep30\p@\relax
\advance\@topsep by -\baselineskip
\item\relax
\author@andify\authors
\def\\{\protect\linebreak}%
{\authors}%
\ifx\@empty\contribs \else ,\penalty-3 \space \@setcontribs
\@closetoccontribs \fi
\endtrivlist
\endgroup } \makeatother
 \baselineskip 18pt

\title{Equilibrium Portfolio Selection for Smooth Ambiguity Preferences}
\author{
Guohui Guan\thanks{School of Statistics, Renmin University of China, Beijing 100872, China; Email: $<$guangh@ruc.edu.cn$>$.}
\and
Zongxia Liang\thanks{Department of Mathematical Sciences, Tsinghua
University, Beijing 100084, China; Email: $<$liangzongxia@mail.tsinghua.edu.cn$>$.}
\and
Jianming Xia\thanks{RCSDS, NCMIS, Academy of Mathematics and Systems Science,
Chinese Academy of Sciences, Beijing 100190, China; Email:
$<$xia@amss.ac.cn$>$.}}

\maketitle

\begin{abstract}
This paper investigates the  equilibrium portfolio selection for smooth ambiguity preferences in a continuous-time market.
The investor is uncertain about the risky asset's drift term and updates the subjective belief according to the Bayesian rule. Two versions of the verification theorem are established and an equilibrium strategy can be
 decomposed into a myopic demand and two hedging demands. When the prior is Gaussian, the closed-form equilibrium solution is obtained. A puzzle in the numerical results is interpreted via an alternative representation of the smooth ambiguity preferences.
\end{abstract}
\vskip15pt

\section{Introduction}

An important topic in modern financial theory is portfolio selection. The expected utility maximization problem with complete information is first studied by \cite{Merton69,Merton1971} for the continuous-time model and is then extensively developed by a large number of literatures; see \citet{KS1999} for a review.
% \citet{Pliska1986}, \citet{Kara1987}, and \citet{Cox1989, Cox1991} for complete markets. In the much more complicated case of
%incomplete markets, the problem was studied by \citet{HePears1991a,HePears1991b} and \citet{KLSX1991} for a specified model
%and by \citet{KS1999, KS2003} for a general
%incomplete semimartingale model. All of the above literatures studied the expected utility maximization problem for complete information.
An important feature in practical investments is partial information, in particular, the unobservable expected stock price returns. The study on expected utility maximization with partial information can go back to \citet{detemple1986}, \citet{Gennotte1986} and \citet{KaraXue1991} and is then extensively developed by, among others, \citet{Lakner1995, Lakner1998}, \cite{karatzas2001bayesian}, %\cite{rogers2001relaxed},
\cite{honda2003optimal}, \cite{sass2004optimizing}, \cite{Rieder2005}, %\cite{putschogl2008optimal,putschogl2011optimal},
\citet{bjork2010optimal}, \citet{Hata2018}, and \cite{Bis2019}.

This paper aims to investigate the continuous-time portfolio selection problem with partial information for ambiguity smooth preferences, which is proposed and axiomatized by \cite{Klibanoff05}. For notational simplicity, we consider a financial market consisting of two assets: a risk-free bond and a stock.\footnote{The extension to the case of multiple risky assets is easy but notational tedious.}  We assume that only the asset prices are observable.
The objective functional of a smooth ambiguity preference at time $t$ is
\begin{equation}\label{equ1}
	\mathbf{E}\left[\left.\phi\left(\mathbf{E}[U(X_T) \vert \mathcal{F}^S_t,Z]\right)\right| \mathcal{F}^S_t\right],
\end{equation}
where $\mathcal{F}^S_t$ represents the available information at time $t$ generated by the stock price, $Z$ is the unobservable drift of the stock price,
$U$ is the utility function for terminal wealth $X_T$, and the shape of $\phi$ captures the attitude toward the uncertainty generated by $Z$.

When $\phi$ is linear, the two conditional expectations of \eqref{equ1} can be reduced, and hence \eqref{equ1} reduces to the standard conditional expected utility $\mathbf{E}[U(X_T) \vert \mathcal{F}^S_t]$.
When $\phi$ is nonlinear, the two conditional expectations can not be reduced anymore and the irreducibility allows for ambiguity-sensitive behavior.
In particular,  the concavity (convexity) of $\phi$ captures ambiguity aversion (seeking),  according to \cite{Klibanoff05}.
%in the sense that it places a larger (smaller) weight on bad expected ``U-utility" realizations.
%The smooth ambiguity model consists of two expectations and two functions $(\phi, U)$, representing the ambiguity attitude and risk attitude, respectively,
Moreover, the irreducibility also leads to the time-inconsistency of the preference, which further leads to the time-inconsistency of optimality: a trading strategy that is optimal today may no longer be optimal tomorrow and hence may be deviated. To find a time-consistent strategy, we consider the so-called \textit{equilibrium solutions}, which have been used in \citet{Strotz1955} for discrete-time setting and \citet{Ekeland2010}, \citet{Ekeland2012} and \cite{Bjork17} for continuous-time setting with objective functionals being significantly different to ours. We refer to \citet{HeJiang2021} for some new progress and \citet{HeZhou2022} for a review on equilibrium solutions for time-inconsistent preferences.
 %and \cite{Hu12}, \cite{BMZ14} for time-consistent mean-variance criterion.
%In \cite{Bjork17}, the authors present a general time-inconsistent problem with two non-linear functions and a state-dependent valuation, which however does not include the form under smooth ambiguity. Inspired by \cite{Bjork17}, we solve the portfolio selection problem under smooth ambiguity by considering the equilibrium strategy and value function. We adopt a similar definition of the equilibrium strategy as in \cite{Bjork17}. Nevertheless, the verification theorem and extended Hamilton-Jacobi-Bellman (HJB) equations are very different, which are presented in Theorem \ref{thm:y}.

The main contributions of this paper are as follows.
\begin{itemize}

\item As far as we know, the existing literature contains no report on the equilibrium strategies for smooth ambiguity preferences with Bayesian learning and time-state-dependent admissible strategies in a continuous-time framework. Our work fulfills such a gap.\footnote{ \cite{balter2021time} also investigate the equilibrium strategies for smooth ambiguity preferences in a continuous-time setting. But the admissible strategies in \cite{balter2021time} are very restrictive: they are time-dependent and state-independent.}

\item We establish two different versions of the verification theorem based on two different \textit{revealing processes},
where both ambiguity averse and ambiguity seeking cases are included. We find that the equilibrium strategy can be decomposed into three parts: a myopic demand and two hedging demands. The two hedging demands are related to the uncertainty of the drift and the risk of the revealing process, respectively. In particular, the optimal strategy for the expected utility preference can also be decomposed in this way as a special case.

\item For the Gaussian prior and the exponential-power specification for $(U,\phi)$, we provide the equilibrium strategies in closed form %in terms of the solution to a system of ODEs. We prove a global existence result for the ODEs
and give a rigorous verification of the solution under mild conditions. % over the ambiguity attitude. %Moreover, the attitudes towards risk and ambiguity are also separated in the closed-form solution.
We find that the equilibrium strategy is a linear function of the revealing process.
%, which is composed of myopic demand and two hedging demands related to the uncertainty of the drift and the risk of the revealing process.

\item In the numerical example, we find that ambiguity aversion lowers the total hedging demand, which is consistent with many previous works. An interesting case is that $\phi(u)=-{1\over\alpha}(-u)^\alpha$ with $\alpha\in(0,1)$. In this case, $\phi$ is convex and hence the investor is ambiguity seeking. But the hedging demand $h^Z$ related to the uncertainty of $Z$ is decreasing w.r.t. the prior variance of $Z$, which represents the uncertainty of $Z$. Such a phenomenon seems like a puzzle. We give an interpretation to this ``puzzle" via the alternative representation of \eqref{equ1}:
$$\Ebf[V(C_{X_T})\vert\mathcal{F}_t]\text{ with } C_{X_T}\trieq U^{-1}(\Ebf[U(X_T)|\mathcal{F}_t,Z]),$$
where $V=\phi\circ U$ and $C_{X_T}$ is the conditional certainty equivalent of $X_T$. $C_{X_T}$ is a monetary payoff.
The shape of $V$ represents the investor's attitude toward the risk measured in \textit{monetary} scale.
When $\alpha\in(0,1)$, $V$ is concave and the investor dislikes the risk of $C_{X_T}$.
Therefore, in this case, it is reasonable that the hedging demand $h^Z$ is decreasing w.r.t. the prior variance of $Z$.
While the shape of $\phi$ represents the investor's attitude toward the risk measured in \textit{utility} scale.
What the convexity of $\phi$ represents is that the investor is seeking for the risk of the conditional expected utility $\Ebf[U(X_T)|\mathcal{F}_t,Z]$, which is not monetary.

\end{itemize}

The optimal investment for smooth ambiguity preference is investigated by \citet{Taboga2005} and \cite{Gollier11} for a single period setting. The discrete-time recursive smooth ambiguity preference is developed by \cite{Klibanoff09} and \cite{hayashi2011intertemporal} and is then applied in portfolio selection and asset pricing by \cite{JM12} and \citet{chen2014dynamic}.  The continuous time recursive smooth ambiguity preference is developed and applied by \citet{HansenMiao2018} and \citet{HansenMiao2022}. In contrast to the non-recursive form \eqref{equ1}, the recursive smooth ambiguity preference is time consistent and hence so is the optimal solution.

The remainder of this paper is as follows. Section 2 formulates the portfolio selection problem for smooth ambiguity preferences. Sections 3 and 4 establish two versions of the verification theorem on the equilibrium strategy. Section 5 provides the closed-form results when the prior is Gaussian.  Section 6 presents and discusses some numerical examples.

\section{Problem Formulation}\label{sec:pro:formulation}

\subsection{Financial Market}

Let  $(\Omega, \mathcal{F}, \{\mathcal{F}_t\}_{t\in[0,T]}, \mathbb{P} )$ be a filtered complete probability space, where constant $T>0$ is  the time horizon and filtration $\{\mathcal{F}_t\}_{t\in[0,T]}$ satisfies the usual conditions and represents the whole information of the financial market.
The financial market consists of two assets: one bond (risk-free asset) and one stock (risky asset). The risk-free interest rate of the bond is a constant $r$.
The stock price process  $S=\{S_t\}_{t\in[0,T]}$ is a geometric Brownian motion which satisfies
\begin{equation*}
	\begin{split}
		\frac{\mathrm{d}S_t}{S_t}=Z\mathrm{d}t+\sigma\mathrm{d}W_t,
	\end{split}
\end{equation*}
where $W=\{W_t\}_{t\in[0,T]}$ is a standard Brownian motion with respect to (w.r.t.)  filtration $\{\mathcal{F}_t\}_{t\in[0,T]}$, the volatility $\sigma>0$ is a constant, and the drift  $ Z$  is an unknown constant.

Assume that $ Z$ is, a prior, independent of the Brownian motion $W$ and $F_0$ is the  prior probability distribution function of $ Z$ satisfying the following condition
\begin{equation}\label{assup1}
	\exists\, \eta>0, \quad \mathbf{E}[e^{\eta Z^{2}}]=\int_{\mathbf{R}} e^{\eta z^2} F_0(\rd z)<\infty.
\end{equation}

The investor (she) does not have access to the whole information but observes the evolutions of the asset prices.  Therefore the filtration of her admissible information is the natural filtration $\{\F^S_t\}$ generated by $S$.
Let process $\bbeta=\{\bbeta_t\}_{t\in[0,T]}$ be given by
$$\bbeta_t\trieq\mathbf{E}[ Z\mid\mathcal{F}^S_t],\quad t\in[0,T].$$
Let process $W^S=\{W^S_t\}_{t\in[0,T]}$ be given by
\begin{equation}\label{equ:ws}
	W^S_t\trieq\int_0^t{ Z-\bbeta_s\over\sigma}\mathrm{d}s+ W_t.
\end{equation}
It is well known that $W^S$ is a standard Brownian motion w.r.t. filtration $\{\mathcal{F}^S_t\}$ under probability $\mathbb{P}$; see, e.g.,  \citet[Proposition 2]{Bis2019}. Obviously, in terms of $W^S$, the stock price evolves according to the following stochastic differential equation (SDE)
\begin{equation}\label{eq:dS:WS}
	{\rd S_t\over S_t}=\bbeta_t\rd t+\sigma\rd W^S_t.
\end{equation}

Let $Y=\{Y_t\}_{t\in[0,T]}$ be the logarithmic stock price process, i.e.,
$$Y_t\trieq\log S_t,%=\log S(0)+\sigma W^\mathbb{Q}_t+\left(r-{\sigma^2\over2}\right)t,
\quad t\in[0,T].$$
Then by \citet[Theorem 1]{Bis2019} (see also Proposition \ref{prop:mu}),
$$\bbeta_t
=\varphi(t,Y_t),\quad t\in[0,T],$$
where $$\varphi(t,y)=\frac{\displaystyle\int_{\mathbf{R}} z \exp \left(\frac{z}{\sigma^{2}}\left[y-Y_0+\frac{1}{2} \sigma^2 t\right]- \frac{z^2}{2\sigma^2}t\right) F_0(\rd z)}{\displaystyle\int_{\mathbf{R}} \exp \left(\frac{z}{\sigma^{2}}\left[y-Y_0+\frac{1}{2} \sigma^2 t\right]- \frac{z^2}{2\sigma^2}t\right) F_0(\rd z)}$$
and $\varphi\in C^\infty([0,T]\times\mathbf{R})$.
Therefore, by \eqref{eq:dS:WS}, we have
\begin{equation}\label{eq:dS:gtY}
	{\rd S_t\over S_t}=\varphi(t,Y_t)\rd t+\sigma\rd W^S_t
\end{equation}
and hence
%$$Y_t=\log S(0)+\sigma W_t+\left(r-{\sigma^2\over2}\right)t$$
%and hence
\begin{equation}\label{eq:dY}
	\rd Y_t%=\left(\beta(t)-{\sigma^2\over2}\right)\rd t+\sigma \rd W^S_t
	=\left(\varphi(t,Y_t)-{\sigma^2\over2}\right)\rd t+\sigma\rd W^S_t.
\end{equation}

%Using Proposition \ref{prop:mu},  the posterior distribution of $ Z$ is
%\begin{equation}
%	\begin{split}
%		P( Z\le v|\mathcal{F}^S_t)&=\mathbf{E}[1_{\{ Z\le v\}}|\mathcal{F}^S_t]
%		\\&=\frac{\int_{-\infty}^v\exp \left(\frac{z}{\sigma^{2}}\left[Y_t-Y_0+\frac{1}{2} \sigma^2 t\right]- \frac{z^2}{2\sigma^2}\right) F_0(\rd z)}{\int_{\mathbf{R}} \exp \left(\frac{z}{\sigma^{2}}\left[Y_t-Y_0+\frac{1}{2} \sigma^2 t\right]- \frac{z^2}{2\sigma^2}\right) F_0(\rd z)}.
%	\end{split}
%\end{equation}		
%%	$$P( Z\le v|\mathcal{F}^S_t)=P( Z\le v|\mathcal{F}^Y_t).$$
%The posterior distribution of $ Z$ only depends on the states at time $t$. We use the notation $F(t,Y_t,\cdot)=P( Z\le v|\mathcal{F}^S_t)$ to represent the posterior distribution of $ Z$ at time $t$.
% As $Y_t$ can be expressed by $\beta(t)$, for simplicity, we use the notation $F(t,{\beta(t)},\cdot)$ instead of $F(t,Y_t,\cdot)$ to represent the posterior distribution of $ Z$ at time $t$ later.

\subsection{Trading Strategies}

The investor does not have access to the whole information but observes the evolutions of the stock price, or equivalently, of the logarithmic stock price $Y$.
The evolution of $Y$ reveals information to the investor about the true value of $Z$ and therefore is called the \textit{revealing process}.
The investor trades the bond and the stock continuously within the time horizon $[0,T]$.
Based on \eqref{eq:dS:gtY}--\eqref{eq:dY}, her trading strategy can be represented by a feedback function
\begin{equation}\label{eq:pi}
	\tilde{\pi}: (t,x,y)\mapsto \tilde{\pi}(t,x,y),
\end{equation}
according to which, the self-financing wealth process $X^{\tilde{\pi}}=\{X^{\tilde{\pi}}_t\}_{t\in[0,T]}$ evolves as
\begin{equation}\label{equ:x}
	\mathrm{d}X^{\tilde{\pi}}_t=rX^{\tilde{\pi}}_t\mathrm{d}t+\tilde{\pi}(t,X^{\tilde{\pi}}_t,Y_t)[\varphi(t,Y_t)-r]\mathrm{d}t
	+\tilde{\pi}(t,X^{\tilde{\pi}}_t,Y_t)\sigma\mathrm{d}W^S_t%\qquad t\in[0,T].
\end{equation}
% $Y=\{Y_t\}_{t\in[0,T]}$ satisfies the following SDE
%\begin{equation}\label{eq:Yty}
%\rd Y_t%=\left(\beta(t)-{\sigma^2\over2}\right)\rd t+\sigma \rd W^S_t
%=\left(\varphi(t,Y_t)-{\sigma^2\over2}\right)\rd t+\sigma\rd W^S_t,%\quad t\in[0,T],\\
%\end{equation}
and $\tilde{\pi}\left(t,X^{\tilde{\pi}}_t,Y_t\right)$ is the dollar amount invested in the stock at time $t\in[0,T]$. Let
\begin{equation}
	\tilde\Pi_0\trieq\left\{\tilde{\pi}:[0,T]\times\mathbf{R}^2\to\mathbf{R}\,\left|\,
	%\begin{array}{l}
	\text{SDE \eqref{equ:x} has a unique strong}  \text{ solution } X^{\tilde{\pi}}% \\
	%\text{for any given states } X^{\tilde{\pi}}_t=x \text{ and } Y_t=y.
	% \end{array}
	\right.\right\}.
\end{equation}
Let $\tilde\Pi\subseteq\tilde\Pi_0$ denote the set of admissible trading strategies $\tilde{\pi}$. The definition of admissible strategies is not given here for general models. But it will be specified in Section \ref{sec:special} for some special cases.

\subsection{Smooth Ambiguity Preference and Equilibrium Strategies}

Following \cite{Klibanoff05}, we assume that the investor has a smooth ambiguity preference and her objective function at time $t$ is
\begin{equation}\label{eq:sm:t:J}
	J(t,x,y,\tilde{\pi})\trieq\mathbf{E}\Big[\left.\phi\Big(\mathbf{E} \left[ U\left(X^{\tilde{\pi}}_T\right) \mid Z,X^{\tilde{\pi}}_t=x,Y_t=y\right]\Big)\,\right|\,X^{\tilde{\pi}}_t=x,Y_t=y\Big],
\end{equation}
where
\begin{itemize}
	
	\item $U$ is an increasing concave function capturing risk aversion;
	\item $\phi$ is an increasing function capturing her attitude toward ambiguity. In particular, she is ambiguity averse (resp. seeking, neutral) if $\phi$ is concave (resp. convex, linear), according to  \cite{Klibanoff05}.
\end{itemize}
%{\color{red}
%\begin{remark}
%Another version of the smooth ambiguity model is formulated by functions $V$ and $U$, i.e.,
%\begin{equation}\label{eq:sm:t:J1}
%	V^{-1}\left(\mathbf{E}\Big[\left.V\circ U^{-1}\left(\mathbf{E} \left[ U\left(X^{\tilde{\pi}}_T\right) \mid Z,X^{\tilde{\pi}}_t=x,Y_t=y\right]\right)\,\right|\,X^{\tilde{\pi}}_t=x,Y_t=y\Big]\right),
%\end{equation}
%\eqref{eq:sm:t:J} and \eqref{eq:sm:t:J1} are equivalent by setting $V=\phi\circ U^{-1}$. $\phi$ and $V$ reflect the investor's attitude toward the expected ``U-utility" realizations and certainty equivalent, respectively.
%\end{remark}
%}

A realization of $Z$ is usually denoted by $z\in\range (Z)$. We will use the following notation:
\begin{align}
	%&J(t,x,y,\tilde{\pi})\trieq\mathbf{E}\left[\left.\phi\left(\mathbf{E} \left[\left.U\left(X^{\tilde{\pi}}_t\right)\,\right |\, Z,X^{\tilde{\pi}}_t=x,Y_t=y\right]\right)\,\right|\,X^{\tilde{\pi}}_t=x,Y_t=y\right],\\
	{\tilde g}^{z,\tilde{\pi}}(t,x,y)\trieq \mathbf{E} \left[\left.U\left(X^{\tilde{\pi}}_T\right)\,\right |\, Z=z,X^{\tilde{\pi}}_t=x,Y_t=y\right],\label{eq:gzpi}
\end{align}
for $(t,x,y)\in [0,T]\times\mathbf{R}^2$, $\tilde{\pi}\in\tilde\Pi$ and $z\in\range (Z)$. %, which is justified as follows.
For every $\tilde{\pi}\in\tilde\Pi$, ${\tilde g}^{z,\tilde{\pi}}(t,x,y)$ is the ``inside" conditional  expected utility, given the state $(x,y)$ at time $t$ and the inside information $Z=z$. Here ``inside" refers to the situation as if the investor observes the realization $Z=z$.

Let $\{\F^W_t\}$ be the natural filtration generated by $W$. Then the $\sigma$-field generated by $ Z$ and $\F^S_t$ coincides that by $ Z$ and $\F^W_t$, that is,
$\sigma( Z)\vee \F^S_t=\sigma( Z)\vee \F^W_t$.
Given a realization $z\in\range(Z)$, the corresponding ``conditional realization" $X^{z,\tilde{\pi}}$ of $X^{\tilde{\pi}}$ satisfies the following SDE
\begin{equation}\label{equ:x:mu}
	\mathrm{d}X^{z,\tilde{\pi}}_t=rX^{z,\tilde{\pi}}_t\mathrm{d}t+\tilde{\pi}(t,X^{z,\tilde{\pi}}_t,Y^{z}_t)(z-r)\mathrm{d}t
	+\tilde{\pi}(t,X^{z,\tilde{\pi}}_t,Y^{z}_t)\sigma\mathrm{d}W_t. %,\qquad t\in[0,T].
\end{equation}
Here $Y^{z}$ denotes the corresponding ``conditional realization" of $Y$, which satisfies the following SDE
\begin{equation}\label{equ:y}
	\rd Y^{z}_t=\left(z-{\sigma^2\over2}\right)\rd t+\sigma \rd W_t. %,\quad t\in[0,T].
\end{equation}
By \eqref{equ:x:mu}--\eqref{equ:y}, we know that $(X^{z,\tilde{\pi}},Y^{z})$ is a Markovian process w.r.t. the filtration $\{\F^W_t\}$.
%Then
%\begin{equation}\label{eq:g:zpi}
%\begin{split}
%	&\mathbf{E} \left[U\left(X^{\tilde{\pi}}_t\right)\mid Z=z,{\mathcal{F}^S_t}\right]\\
%	=&\mathbf{E} \left[U\left(X^{ Z,\tilde{\pi}}(T)\right)\mid Z=z,{\mathcal{F}^S_t}\right]\\
%	=&\mathbf{E} \left[U\left(X^{z,\tilde{\pi}}_t\right)\mid Z=z,{\mathcal{F}^W_t}\right]\\
%	=&\mathbf{E} \left[U\left(X^{z,\tilde{\pi}}_t\right)\mid {\mathcal{F}^W_t}\right]\quad(\text{$Z$ is independent of $\F^W_t$)}\\
%	=&\mathbf{E} \left[U\left(X^{z,\tilde{\pi}}_t\right)\mid X^{z,\tilde{\pi}}_t=x,Y^z(t)=y\right]\\
%	=&\mathbf{E} \left[U\left(X^{z,\tilde{\pi}}_t\right)\mid Z=z, X^{z,\tilde{\pi}}_t=x,Y^Z_t=y\right]\\
%	=&{\tilde g}^{z,\tilde{\pi}}(t,x,y)
%\end{split}
%\end{equation}
%for some Borel function ${\tilde g}^{z,\tilde{\pi}}$, which justifies \eqref{eq:gzpi}. Moreover,
%$$\mathbf{E} \left[U\left(X^{\tilde{\pi}}_t\right)\mid Z,{\mathcal{F}^S_t}\right]
%={\tilde g}^ {Z,\tilde{\pi}}(t,x,y),$$
%which, by Proposition \ref{prop:mu}, justifies \eqref{eq:sm:t:J}.
%
%
%Then \eqref{eq:sm:t}-\eqref{eq:gzpi} can also be rewritten as
%\begin{align*}
%	&J(t,x,y,\tilde{\pi})\trieq\mathbf{E}\left[\left.\phi\left(\mathbf{E} \left[\left.U\left(X^{\tilde{\pi}}_t\right)\,\right |\, Z,{\mathcal{F}^S_t}\right]\right)\,\right|\,{\mathcal{F}^S_t}\right],\\
%	&{\tilde g}^{z,\tilde{\pi}}(t,x,y)\trieq \mathbf{E} \left[\left.U\left(X^{\tilde{\pi}}_t\right)\,\right |\, Z=z,{\mathcal{F}^S_t}\right].
%\end{align*}
Obviously,
$$X^{\tilde{\pi}}=X^{ Z,\tilde{\pi}}\text{ and }Y=Y^{ Z}.$$
Then
\begin{equation}
	{\tilde g}^{z,\tilde{\pi}}(t,x,y)= \mathbf{E} \left[\left.U\left(X^{z,\tilde{\pi}}_T\right)\,\right |\,X^{z,\tilde{\pi}}_t=x,Y^z_t=y\right].\label{eq:gzpi:z}
\end{equation}
Because the  posterior distribution  of $Z$ at time $t$ depends only on $Y_t$ (see, e.g., Proposition \ref{prop:mu}), we also have
\[
\begin{split}
	J(t,x,y,\tilde{\pi})&=\mathbf{E}\left[\left.\phi\left({\tilde g}^{Z,\tilde{\pi}}(t,x,y)\right)\,\right|\,X^{\tilde{\pi}}_t=x,Y_t=y\right]
	\\&=\mathbf{E}\left[\left.\phi\left({\tilde g}^{Z,\tilde{\pi}}(t,x,y)\right)\,\right|\,Y_t=y\right].
\end{split}
\]

We say $\tilde{\pi}^*\in\tilde\Pi$ is optimal for $(t,x,y)$ if $\tilde{\pi}^*$ maximizes \eqref{eq:sm:t:J} over $\tilde{\pi}\in\tilde\Pi$.
In the case of linear $\phi$, the investor is neutral towards  ambiguity and the objective function \eqref{eq:sm:t:J} reduces to the classical conditional expected utility $\mathbf{E} \left[U\left(X^{\tilde{\pi}}_T\right)\mid X^{\tilde{\pi}}_t=x,Y_t=y\right]$ with partial observations. In general, $\phi$ is non-linear, and therefore the corresponding optimization problem is time-inconsistent:  the optimal strategy $\tilde{\pi}^*$ for $(t,x,y)$ is not necessarily optimal for $\left(s, X^{\tilde{\pi}^*}_s, Y_s\right)$, $s>t$.
Therefore, we consider the equilibrium solutions, instead of optimal solutions, following \cite{Bjork17}.

\begin{definition}[Equilibrium Strategy]{\label{eq}}
Consider an admissible strategy ${\tilde{\pi}^*}\in\tilde\Pi$.
	We say $\tilde{\pi}^*$ is an equilibrium strategy if, for every $(t,x,y)$ and $\tilde{\pi}\in\tilde\Pi$,  whenever $\tilde{\pi}_{t,h}\in\tilde\Pi$ for all sufficiently small $h>0$, we have
	\begin{eqnarray}{\label{mv1}}
		\mathop {\limsup}\limits_{h\rightarrow 0^+} \frac{J(t,x,y, \tilde{\pi}_{t,h})-J(t,x,y,\tilde{\pi}^*)}{h}\le 0,
	\end{eqnarray}
	where $ \tilde{\pi}_{t,h}$ is given by
	\begin{eqnarray*}
		\tilde{\pi}_{t,h}=
		\begin{cases}
			\tilde{\pi}&  \text{on }  [t,t+h),\\
			\tilde{\pi}^*& \text{otherwise}.
		\end{cases}
	\end{eqnarray*}
	In this case, we say ${\tilde g}^z={\tilde g}^{z,\tilde{\pi}^*}$ is an equilibrium value function for the inside expected utility, where ${\tilde g}^{z,\tilde{\pi}}$ is given by \eqref{eq:gzpi}.
\end{definition}

The equilibrium value function for the inside expected utility plays a central role in studying the equilibrium strategies, as we will see.

\section{Verification Theorem}\label{sec:verification}

For every $z\in\range( Z)$ and $\tilde{\pi}\in\tilde\tilde\Pi$,  the infinitesimal generator $\mathcal{\tilde{A}}^{z,\tilde{\pi}}$ for the inside system \eqref{equ:x:mu}--\eqref{equ:y} is given by, for $f\in C^{1,2,2}([0,T)\times \mathbf{R}^2)\cap C([0,T]\times \mathbf{R}^2)$,
\begin{equation*}
	\mathcal{\tilde{A}}^{z,\tilde{\pi}}f\trieq f_t+[rx+\tilde{\pi}  (z-r)]f_x+\left(z-\frac{\sigma^2}{2}\right)f_y+\frac{1}{2}\tilde{\pi}^2\sigma^2f_{xx}
	+\frac{1}{2}\sigma^2 f_{yy}+\tilde{\pi} \sigma^2f_{xy}.
\end{equation*}

The following assumptions on functions $\{\tilde{g}^z, z\in{\text{Range}(Z)}\}$ will be used in the formulation of the verification theorem.

\begin{assumption}\label{assu0:y}
	For any $\tilde{\pi}\in\tilde\Pi$ and $(t,x,y)\in[0,T)\times\Rbf^2$,   there exists $\tilde{t}\in(t,T)$ such that, $\forall z\in\range(Z)$,  under the conditional probability $\mathbb{P}[\,\cdot\mid X^{z,\tilde{\pi}}_t=x,Y^z_t=y]$, the following two families are supermartingales w.r.t. filtration $\{\F^W_{s}\}$,
	\begin{align*}
		&\left\{\int_t^s\tilde{g}^z_x(u,X^{z,\tilde{\pi}}_u,Y^{z}_u) \tilde{\pi}(u,X^{z,\tilde{\pi}}_u,Y^{z}_u) \mathrm{d}W_u\right\}_{s\in[t,\tilde{t}]},\\
		&\left\{\int_t^s\tilde{g}^z_y(u,X^{z,\tilde{\pi}}_u,Y^{z}_u)\mathrm{d}W_u\right\}_{s\in[t,\tilde{t}]}.
	\end{align*}
\end{assumption}

\begin{assumption}\label{assu1:y}
	For any $z\in\range(Z)$, $\tilde{\pi}\in\tilde\Pi$  and $(t,x,y)\in[0,T)\times\Rbf^2$,  there exists $\tilde{t}\in(t,T)$ such that,  under the conditional probability $\mathbb{P}[\,\cdot\mid X^{z,\tilde{\pi}}_t=x,Y^z_t=y]$, the family
	\begin{equation*}\label{equ:assu1:y}
		\left\{\frac{1}{h}\int_t^{t+h}\mathcal{\tilde{A}}^{z,\tilde{\pi}} \tilde{g}^z(u,X^{z,\tilde{\pi}}_u,Y^{z}_u)\mathrm{d}u\right\}_{0<h\le \tilde{t}-t}
	\end{equation*}
	is uniformly integrable.
\end{assumption}

\begin{assumption}\label{assu2:y}
	For any $\tilde{\pi}\in\tilde\Pi$ and $(t,x,y)\in[0,T)\times\Rbf^2$,  there exists $\tilde{t}\in(t,T)$ such that,
	under the conditional probability $\mathbb{P}[\,\cdot\mid Y^z_t=y]$,
	the family
	\begin{equation*}\label{equ:assu2:y}
		\left\{\phi^\prime\left(\tilde{g}^Z(t,x,y)\right)\mathbf{E}\left[\left.\frac{1}{h}\int_t^{t+h}\mathcal{\tilde{A}}^{Z,\tilde{\pi}} \tilde{g}^Z(u,X^{Z,\tilde{\pi}}_u,Y^{Z}_u)\mathrm{d}u\,\right|Z,X^{Z,\tilde{\pi}}_t=x,Y^Z_t=y\right]\right\}_{0<h\le\tilde{t}-t}
	\end{equation*}
	is uniformly integrable  in the case of concave $\phi$, or, the family
	\begin{equation*}\label{equ:assu22:y}
		\left\{\phi^\prime\left(\tilde{g}^{Z, \tilde{\pi}_{t,h}}(t,x,y)\right)\mathbf{E}\left[\left.\frac{1}{h}\int_t^{t+h}\mathcal{\tilde{A}}^{Z,\tilde{\pi}} \tilde{g}^Z(u,X^{Z,\tilde{\pi}}_u,Y^{Z}_u)\mathrm{d}u\,\right|Z,X^{Z,\tilde{\pi}}_t=x,Y^Z_t=y\right]\right\}_{0<h\le\tilde{t}-t}
	\end{equation*}
	is uniformly integrable in the case of convex $\phi$.
\end{assumption}

\begin{assumption}\label{assu3:y}
	In the case of convex $\phi$, for any $z\in\range(Z)$, $\tilde{\pi}\in\tilde\Pi$  and $(t,x,y)\in[0,T)\times\Rbf^2$,
	$$\lim_{h\to 0^+}\mathbf{E} \left[
	\tilde{g}^z(t+h,X^{z,\tilde{\pi}}_{t+h},Y^{z}_{t+h})
	\mid X^{z,\tilde{\pi}}_t=x, Y^z_t=y\right]=\tilde{g}^z(t,x,y).$$
\end{assumption}

\begin{theorem}[Verification Theorem]\label{thm:y}
	Assume that $U$ is an increasing concave $C^2$ function and $\phi$ is an increasing concave/convex $C^1$ function.
	Consider a family of functions
	$$\left\{\tilde{g}^z, z\in\range( Z)\right\}\subset C^{1,2,2}([0,T)\times \mathbf{R}^2)\cap C([0,T]\times \mathbf{R}^2)$$
	and an admissible strategy $\tilde{\pi}^*\in\tilde\Pi$.  Assume that they satisfy the following three conditions:
	\begin{description}
		\item[(a)] $\forall (t,x,y)\in[0,T]\times \mathbf{R}^2 $,
		\begin{align}
			&\sup_{\tilde{\pi}\in\tilde\Pi}\mathbf{E}\left[\phi^\prime\left(\tilde{g}^Z(t,x,y)\right)\mathcal{\tilde{A}}^{ Z,\tilde{\pi}} \tilde{g}^Z(t,x,y)\mid Y_t=y\right]=0, \label{equ:hjb:y}
			\\ &\mathcal{\tilde{A}}^{z,\tilde{\pi}^*}\tilde{g}^z(t,x,y)=0, \label{equ:hjbg:y}
			\\ &\tilde{g}^z(T,x,y)=U(x);\label{equ:hjbg2:y}
		\end{align}
		\item[(b)] %$\forall (t,x,y)\in[0,T]\times \mathbf{R}^2 $ and $z\in\range(Z)$,
		$\forall z\in\range(Z)$, $\left\{\tilde{g}^z\left(t,X^{z,\tilde{\pi}^*}_t,Y^{z}_t\right)\right\}_{t\in[0,T]}$ is a martingale w.r.t. filtration $\{\mathcal{F}^W_t\}$;
		\item[(c)] $\left\{\tilde{g}^z, z\in\range(Z)\right\}$ satisfies Assumptions \ref{assu0:y}--\ref{assu3:y}.
	\end{description}
	Then $\tilde{\pi}^*$ is an equilibrium strategy and, $\forall z\in\range(Z)$ and $(t,x,y)\in[0,T]\times \mathbf{R}^2$,
	$$\mathbf{E}\left[U(X^{\tilde{\pi}^*}_T)\mid Z=z,X^{\tilde{\pi}^*}_T=x,Y_t=y\right]=\tilde{g}^z(t,x,y).$$
\end{theorem}

\begin{proof} See Appendix \ref{app:thm}. \end{proof}

Now we derive the equilibrium strategy $\tilde{\pi}^*$ from condition (a) of Theorem \ref{thm:y}. We can see from \eqref{equ:hjb:y}--\eqref{equ:hjbg:y} that
$$\tilde{\pi}^*\in\arg \sup_{\tilde{\pi}\in\tilde\Pi}\mathbf{E}\left[\phi^\prime\left(\tilde{g}^Z(t,x,y)\right)\mathcal{\tilde{A}}^{ Z,\tilde{\pi}} \tilde{g}^Z(t,x,y)\mid Y_t=y\right].$$
The terms on the right-hand side involving $\tilde\pi$ are
\begin{equation}\label{termspi}
\begin{split}
&\Ebf\left[\phi^\prime\left(\tilde{g}^Z(t,x,y)\right)\tilde{g}^Z_x(t,x,y)(Z-r)\mid Y_t=y\right]\tilde\pi(t,x,y)\\
+&\sigma^2\Ebf\left[\phi^\prime\left(\tilde{g}^Z(t,x,y)\right)\tilde{g}^Z_{xy}(t,x,y)\mid Y_t=y\right]\tilde\pi(t,x,y)\\+&{1\over2}\sigma^2\Ebf\left[\phi^\prime\left(\tilde{g}^Z(t,x,y)\right)\tilde{g}^Z_{xx}(t,x,y)\mid Y_t=y\right]\tilde\pi^2(t,x,y).
\end{split}
\end{equation}
Then by maximizing \eqref{termspi} in $\tilde\pi$, the equilibrium strategy $\tilde{\pi}^*$ satisfies
\begin{equation}\label{ofc:y}
	\begin{split}
		\tilde{\pi}^*(t,x,y)=&
		-\frac{\mathbf{E}\left[\phi^\prime(\tilde{g}^Z(t,x,y)){\tilde g}_{x}^ Z(t,x,y)( Z-r) \mid  Y_t=y\right]}
		{\sigma^2 \mathbf{E}\left[\phi^\prime(\tilde{g}^Z(t,x,y))\tilde{g}^Z_{xx}(t,x,y) \mid Y_t=y\right]}\\
		&-\frac{\mathbf{E}\left[\phi^\prime(\tilde{g}^Z(t,x,y))\tilde{g}^Z_{xy}(t,x,y) \mid Y_t=y\right]}
		{\mathbf{E}\left[\phi^\prime(\tilde{g}^Z(t,x,y))\tilde{g}^Z_{xx}(t,x,y) \mid Y_t=y\right]}
	\end{split}
\end{equation}
if such a $\tilde{\pi}^*$ is admissible.

Let $\cov(\,\cdot \mid Y_t=y)$ and $\var(\,\cdot \mid Y_t=y)$  denote the covariance and the variance under the conditional probability $\mathbb{P}(\,\cdot\mid Y_t=y)$, respectively. We have  decomposition
\begin{align*}
&\mathbf{E}\left[\phi^\prime(\tilde{g}^Z(t,x,y)){\tilde g}_{x}^ Z(t,x,y)( Z-r) \mid  Y_t=y\right]\\
=&\mathbf{E}\left[\phi^\prime(\tilde{g}^Z(t,x,y)){\tilde g}_{x}^ Z(t,x,y)\mid  Y_t=y\right] \cdot \mathbf{E}\left[ Z-r \mid  Y_t=y\right]\\
&+\cov(\phi^\prime(\tilde{g}^Z(t,x,y)){\tilde g}_{x}^ Z(t,x,y), Z-r \mid  Y_t=y)
\end{align*}
Then we can write \eqref{ofc:y} as
\begin{equation}\label{ofc2:y}
	\begin{split}
		\tilde{\pi}^*(t,x,y)=&
		\underbrace{-\frac{\mathbf{E}\left[\phi^\prime(\tilde{g}^Z(t,x,y)){\tilde g}_{x}^ Z(t,x,y) \mid  Y_t=y\right]}
		{\sigma^2\mathbf{E}\left[\phi^\prime(\tilde{g}^Z(t,x,y))\tilde{g}^Z_{xx}(t,x,y)  \mid Y_t=y\right]}\left(\Ebf[Z\mid Y_t=y]-r\right)}_{\tilde\pi^M}
		\\& \underbrace{-\frac{\cov\left(\phi^\prime(\tilde{g}^Z(t,x,y)){\tilde g}_{x}^ Z(t,x,y), Z-r \mid  Y_t=y\right)}
		{\sigma^2\mathbf{E}\left[\phi^\prime(\tilde{g}^Z(t,x,y))\tilde{g}^Z_{xx}(t,x,y) \mid Y_t=y\right]}}_{\tilde\pi^Z}\\
		&
		\underbrace{-\frac{\mathbf{E}\left[\phi^\prime(\tilde{g}^Z(t,x,y))\tilde{g}^Z_{xy}(t,x,y) \mid Y_t=y\right]}
			{\mathbf{E}\left[\phi^\prime(\tilde{g}^Z(t,x,y))\tilde{g}^Z_{xx}(t,x,y) \mid Y_t=y\right]}}_{\tilde\pi^Y}.
		%\\
		%		&=-\frac{\int_{\mathbf{R} }\phi^\prime(\tilde{g}^z){\tilde g}_{x}^z(z-r) p\rd z}{\int_{\mathbf{R} }\phi^\prime(\tilde{g}^z)\tilde{g}^z_{xx}\sigma^2 p\rd z}
		%		-\frac{\int_{\mathbf{R} }\phi^\prime(\tilde{g}^z){\tilde g}_{x\beta}^z\zeta(t) p\rd z}{\int_{\mathbf{R} }\phi^\prime(\tilde{g}^z)\tilde{g}^z_{xx}\sigma^2 p\rd z}.
	\end{split}
\end{equation}
%Now we are going to give a brief financial interpretation to the above decomposition.
%Recalling \eqref{equ:x}, the instantaneous excess return of portfolio $\tilde\pi$ is
%$$\mathrm{d}X^{\tilde\pi}_t-rX^{\tilde\pi}_t\mathrm{d}t=\tilde{\pi}(t,X^{\tilde{\pi}}_t,Y_t)[\varphi(t,Y_t)-r]\mathrm{d}t
%	+\tilde{\pi}(t,X^{\tilde{\pi}}_t,Y_t)\sigma\mathrm{d}W^S_t.$$
%We also have
%$$\mathrm{d}X^{\tilde\pi}_t-rX^{\tilde\pi}_t\mathrm{d}t=\tilde{\pi}(t,X^{\tilde{\pi}}_t,Y_t)(Z-r)\mathrm{d}t
%	+\tilde{\pi}(t,X^{\tilde{\pi}}_t,Y_t)\sigma\mathrm{d}W_t.$$
%Moreover, $\varphi(t,y)=\mathbf{E}\left[Z \mid Y_t=y\right]$.
%Then \eqref{termspi} roughly reads
%{\color{red}
%\begin{align*}
%&\Ebf\left[\phi^\prime\left(\tilde{g}^Z(t,x,y)\right)\tilde{g}^Z_x(t,x,y)\mid Y_t=y\right] {\Ebf[\mathrm{d}X^{\tilde\pi}_t-rX^{\tilde\pi}_t\mathrm{d}t\mid Y_t=y]\over \mathrm{d}t}
%\\
%+&{\cov(\phi^\prime\left(\tilde{g}^Z(t,x,y)\right)\tilde{g}^Z_x(t,x,y), \mathrm{d}X^{\tilde\pi}_t-rX^{\tilde\pi}_t\mathrm{d}t\mid Y_t=y]\over \mathrm{d}t}\\
%+&\Ebf\left[\phi^\prime\left(\tilde{g}^Z(t,x,y)\right)\tilde{g}^Z_{xy}(t,x,y)\mid Y_t=y\right]
%{\cov(\mathrm{d}X^{\tilde\pi}_t-rX^{\tilde\pi}_t\mathrm{d}t, \mathrm{d}Y_t \mid Y_t=y)\over \mathrm{d}t}\\
%+&{1\over2}\Ebf\left[\phi^\prime\left(\tilde{g}^Z(t,x,y)\right)\tilde{g}^Z_{xx}(t,x,y)\mid Y_t=y\right]
%{\var(\mathrm{d}X^{\tilde\pi}_t-rX^{\tilde\pi}_t\mathrm{d}t\mid Y_t=y)\over \mathrm{d}t},
%\end{align*}
%}
The first term $\tilde\pi^M$ in \eqref{ofc2:y} is called myopic demand, which is optimal for an investor that behaves myopically:  taking into account the portfolio's expected return $\Ebf[Z\mid Y_t=y]\tilde\pi(t,x,y)$ (and the volatility $\sigma\tilde\pi(t,x,y)$), but disregarding the effect from the risk of $Y_t$ and the ambiguity of $Z$.
The second term $\tilde\pi^Z$ is the hedging demand associated with the ambiguity of $Z$, which takes into account the correlation between $Z$ and the marginal utility $\phi^\prime(\tilde{g}^Z(t,x,y)){\tilde g}_{x}^ Z(t,x,y)$. The third term $\tilde\pi^Y$ is the hedging demand associated with the risk of the revealing process $Y$, where the effect of the volatility of $Y$ is implicit; see Remark \ref{rmk:hgd} below for a more obvious interpretation, where the effect of volatility of the revealing process $\bbeta$ is explicit.

Now we discuss the case $\phi(u)=u$. 	
%The financial interpretations for hedging demands 1 and 2 are clear in the case $\phi(x)=x$.
Actually, in this case, the smooth ambiguity preference reduces to the expected utility preference. Consider the conditional expected utility of $\tilde{\pi}^*$, which is given by
	$$
	\tilde{g}^{EU}(t,x,y)=\mathbf{E}\left[\tilde{g}^Z(t,x,y) \mid Y_t=y\right],\quad (t,x,y)\in[0,T)\times\Rbf^2.
	$$
	By Proposition \ref{prop:mu} and $\varphi(t,y)=\mathbf{E}\left[Z \mid Y_t=y\right]$, we obtain the following partial derivatives:
	\begin{equation}\label{eq:partialderi}
	\begin{split}
		\tilde{g}^{EU}_t(t,x,y)=&\mathbf{E}\left[\tilde{g}^Z_t(t,x,y) \mid Y_t=y\right]
		+{1\over 2}\cov\left(Z,\tilde{g}^Z(t,x,y) \mid Y_t=y\right)\\
		&-{1\over 2\sigma^2}\cov\left(Z^2,\tilde{g}^Z(t,x,y) \mid Y_t=y\right),\\
		\tilde{g}^{EU}_x(t,x,y)=&\mathbf{E}\left[\tilde{g}^Z_x(t,x,y) \mid Y_t=y\right],\\
		\tilde{g}^{EU}_y(t,x,y)=&\mathbf{E}\left[\tilde{g}^Z_y(t,x,y) \mid Y_t=y\right]+{1\over\sigma^2}\cov(Z,\tilde{g}^Z(t,x,y) \mid Y_t=y)\\
		=&\mathbf{E}\left[\tilde{g}^Z_y(t,x,y) \mid Y_t=y\right]+{1\over\sigma^2}\mathbf{E}\left[\left(Z-\varphi(t,y)\right)\tilde{g}^Z(t,x,y) \mid Y_t=y\right],\\
		\tilde{g}^{EU}_{xx}(t,x,y)=&\mathbf{E}\left[\tilde{g}^Z_{xx}(t,x,y) \mid Y_t=y\right],\\
		\tilde{g}^{EU}_{xy}(t,x,y)=&\mathbf{E}\left[\tilde{g}^Z_{xy}(t,x,y) \mid Y_t=y\right]+{1\over\sigma^2}\mathbf{E}\left[\left(Z-\varphi(t,y)\right){\tilde g}_x^Z(t,x,y) \mid Y_t=y\right],\\
		\tilde{g}^{EU}_{yy}(t,x,y)=&\mathbf{E}\left[\tilde{g}^Z_{yy}(t,x,y) \mid Y_t=y\right]+{1\over\sigma^4}\mathbf{E}\left[\left(Z-\varphi(t,y)\right)^2\tilde{g}^Z(t,x,y) \mid Y_t=y\right]\\&+{2\over\sigma^2}\mathbf{E}\left[\left(Z-\varphi(t,y)\right)\tilde{g}^Z_y(t,x,y) \mid Y_t=y\right]\\
		&-{1\over\sigma^2}\mathbf{E}\left[\varphi_y(t,y)\tilde{g}^Z(t,x,y) \mid Y_t=y\right].
		\end{split}
	\end{equation}
Moreover,
	\begin{equation}\label{eq:varZ:y}
		\varphi_y(t,y)={1\over \sigma^2}\left[\mathbf{E}\left[Z^2 \mid Y_t=y\right]
		-\varphi^2(t,y)\right]={1\over\sigma^2}\var(Z\mid Y_t=y).
	\end{equation}
	Plugging \eqref{eq:partialderi}--\eqref{eq:varZ:y} into $\mathbf{E}\left[\left.\mathcal{\tilde{A}}^{ Z,\tilde{\pi}} \tilde{g}^Z(t,x,y)\right|Y_t=y\right]$, we get
	\begin{align*}
		\mathbf{E}\left[\left.\mathcal{\tilde{A}}^{ Z,\tilde{\pi}} \tilde{g}^Z(t,x,y)\right|Y_t=y\right]=\mathcal{\tilde{A}}^{\tilde{\pi}}_0\tilde{g}^{EU}(t,x,y),
	\end{align*}
	where $\mathcal{\tilde{A}}^{\tilde{\pi}}_0$ is the infinitesimal operator for \eqref{eq:dY} and \eqref{equ:x}:
	\begin{align*}
		\mathcal{\tilde{A}}^{\tilde{\pi}}_0f\trieq f_t+[rx+\tilde{\pi}  (\varphi-r)]f_x+\left(\varphi-\frac{\sigma^2}{2}\right)f_y+\frac{1}{2}\tilde{\pi}^2\sigma^2f_{xx}
		+\frac{1}{2}\sigma^2 f_{yy}+\tilde{\pi} \sigma^2 f_{xy}.
	\end{align*}
	Then condition (a) of Theorem \ref{thm:y} implies that
	\begin{align*}
		\begin{cases}
			\sup_{\tilde{\pi}\in\tilde\Pi}\mathcal{\tilde{A}}^{\tilde{\pi}}_0 \tilde{g}^{EU}(t,x,y)=0,\quad (t,x,y)\in[0,T)\times\Rbf^2,\\
			\tilde{g}^{EU}(T,x,y)=U(x),\quad (x,y)\in\Rbf^2,
		\end{cases}
	\end{align*}
	which is the HJB equation for the value function of the expected utility maximization. Moreover,   condition (a) of Theorem \ref{thm:y} also implies that
	$$\mathcal{\tilde{A}}^{\tilde{\pi}^*}_0\tilde{g}^{EU}(t,x,y)=\sup_{\tilde{\pi}\in\tilde\Pi}\mathcal{\tilde{A}}^{\tilde{\pi}}_0\tilde{g}^{EU}(t,x,y)$$
	and hence
	\begin{equation}\label{eq:pi*:eu:y}
		\tilde{\pi}^*(t,x,y)=-{{\tilde g}_{x}^{EU}(t,x,y) \over \tilde{g}^{EU}_{xx}(t,x,y)}{\varphi(t,y)-r\over \sigma^2}
		-\frac{\tilde{g}^{EU}_{xy}(t,x,y)}{\tilde{g}^{EU}_{xx}(t,x,y)}.
	\end{equation}
	Therefore, in the case $\phi(u)=u$, the equilibrium solution $\tilde{\pi}^*$ is also the optimal solution for the expected utility maximization.
	
	The first term of \eqref{ofc2:y}  coincides with the first term of \eqref{eq:pi*:eu:y}, which is called myopic demand in the expected utility maximization literature. %Therefore, we call the first term of \eqref{ofc2:y} myopic demand.
	
	The second term of \eqref{eq:pi*:eu:y} is called hedging demand in the expected utility maximization literature and, by \eqref{eq:partialderi}, it can be decomposed as
	\begin{equation}\label{eq:hd:eu:y}
		-\frac{\tilde{g}^{EU}_{xy}(t,x,y)}{\tilde{g}^{EU}_{xx}(t,x,y)}=-\frac{\cov({\tilde g}_{x}^ Z(t,x,y), Z-r \mid  Y_t=y)}{\sigma^2\mathbf{E}\left[\tilde{g}^Z_{xx}(t,x,y)  \mid Y_t=y\right]}-\frac{\mathbf{E}\left[\tilde{g}^Z_{xy}(t,x,y) \mid Y_t=y\right]}{\mathbf{E}\left[\tilde{g}^Z_{xx}(t,x,y) \mid Y_t=y\right]}.
	\end{equation}
The two terms on the right-hand side of \eqref{eq:hd:eu:y} coincide with the second and the third terms of \eqref{ofc2:y}.
%Therefore, the hedging demand in the expected utility maximization literature has two parts,
%Moreover,
%	$$\text{hedging demand 1}=-{{\partial\Ebf[{\tilde g}_{x}^ Z(t,x,y)\mid Y_t=w]\over\partial w}|_{w=y} \over \mathbf{E}\left[\tilde{g}^Z_{xx}(t,x,y) \mid Y_t=y\right]}.$$
	
	%In the case $\alpha=1$, $\phi(x)=x$. \eqref{ofc} can be rewritten as
	%\begin{equation}\label{pi:eu:y}
	%	\begin{split}
	%		\tilde{\pi}^*(t,x,y)=&
	%	-\frac{\mathbf{E}\left[{\tilde g}_{x}^ Z(t,x,y)( Z-r) \mid  Y_t=y\right]}{\mathbf{E}\left[\tilde{g}^Z_{xx}(t,x,y)\sigma^2  \mid Y_t=y\right]}-\frac{\mathbf{E}\left[\tilde{g}^Z_{xy}(t,x,y) \mid Y_t=y\right]}{\mathbf{E}\left[\tilde{g}^Z_{xx}(t,x,y) \mid Y_t=y\right]}\\
	%	=&-\frac{\mathbf{E}\left[{\tilde g}_{x}^ Z(t,x,y)( Z-r) \mid  Y_t=y\right]}{\sigma^2\tilde{g}^{EU}_{xx}(t,x,y)}-\frac{\mathbf{E}\left[\tilde{g}^Z_{xy}(t,x,y) \mid Y_t=y\right]}{\tilde{g}^{EU}_{xx}(t,x,y)}\\
	%	=&-\frac{\mathbf{E}\left[{\tilde g}_{x}^ Z(t,x,y)( Z-r) \mid  Y_t=y\right]}{\sigma^2\tilde{g}^{EU}_{xx}(t,x,y)}\\
	%	&-\frac{\tilde{g}^{EU}_{xy}(t,x,y)-{1\over\sigma^2}\mathbf{E}\left[\left(Z-\varphi(t,y)\right){\tilde g}_x^Z(t,x,y) \mid Y_t=y\right]}{\tilde{g}^{EU}_{xx}(t,x,y)}\\
	%	=&-{{\tilde g}_{x}^{EU}(t,x,y) \over \sigma^2\tilde{g}^{EU}_{xx}(t,x,y)}( \varphi(t,y)-r)-\frac{\tilde{g}^{EU}_{xy}(t,x,y)}{\tilde{g}^{EU}_{xx}(t,x,y)}.
	%	\end{split}
	%\end{equation}
	
	%which is the HJB equation associated with \eqref{eq:dY} and \eqref{equ:x} in the expected utility maximization problem. \eqref{pi:eu} is the related optimal investment policy in the expected utility maximization problem.

\section{$\bbeta$-version Verification Theorem}\label{sec:beta-v}

For every $t\in[0,T]$ and $z\in\range( Z)$, $\bbeta^z_t=\varphi(t,Y^z_t)$.
By \eqref{eq:varZ:y}, for every $t\in[0,T)$, $\varphi(t,\cdot)$ is increasing on $\mathbf{R}$.  Moreover, it is strictly increasing if and only if $ Z$ is not degenerated.
%{\color{blue} We have the following proposition for function $\varphi(\cdot,\cdot)$.
%\begin{proposition}\label{prop1:y}
%	%$\varphi(\cdot,\cdot)$ is a well-defined finite-valued	$C^\infty([0,T]\times\mathbf{R})$ function. Besides,
%For every $t\in[0,T)$, $\varphi(t,\cdot)$ is increasing on $\mathbf{R}$.  Moreover, it is strictly increasing if and only if $ Z$ is not degenerate.
%\end{proposition}
%\begin{proof}
%	%From Theorem 1 in \cite{Bis2019}, the first part obviously holds. Next, we show that $\varphi(t,y)$ is increasing in $y$. We directly have
%	\[
%	\begin{split}
%		\frac{\partial \varphi(t,y)}{\partial y}
%		=\frac{\mathbf{E}[{ Z^2\exp(\frac{ Z}{\sigma^2}y-\frac{  Z^2}{2\sigma^2}t)}]\mathbf{E}[\exp(\frac{ Z}{\sigma^2}y-\frac{  Z^2}{2\sigma^2}t)]-\{\mathbf{E}[ Z\exp(\frac{ Z}{\sigma^2}y-\frac{  Z^2}{2\sigma^2}t)]\}^2}{\sigma^2\{\mathbf{E}[\exp(\frac{ Z}{\sigma^2}y-\frac{  Z^2}{2\sigma^2}t)]\}^2}.
%	\end{split}
%	\]
%	Using Cauchy-Schwarz inequality, we have that $\frac{\partial \varphi(t,y)}{\partial y}\ge 0$. The equality holds if and only if there exists some real number $c$ such that $ Z\sqrt{\exp(\frac{ Z}{\sigma^2}y-\frac{  Z^2}{2\sigma^2}t)}=c \sqrt{\exp(\frac{ Z}{\sigma^2}y-\frac{  Z^2}{2\sigma^2}t)}$ a.s., i.e., $ Z$ is degenerate.
%\end{proof}
Hereafter, we always assume that $Z$ is not degenerated. Then $Y$ and $\bbeta$ generate the same filtration. It is much more natural to choose $\bbeta$ as the revealing process than $Y$. In this section, we translate the $Y$-version results in Section \ref{sec:verification} into the $\bbeta$-version.

Consider the following change of variables:
$$(t,x,y)\to (t,x,\beta)\text{ with } \beta=\varphi(t,y)$$
%Denote the inverse function of $\varphi(t,\cdot)$ as $\varphi^{-1}(t,\cdot)$.
%To distinguish the two versions, we can use notations $\tilde g$, $\tilde\pi$, ${\mathcal{\tilde{A}}}$ for $Y$-version, and $g$, $\pi$, $\mathcal{A}$ for $\bbeta$-version.
and, for every $\tilde g\in C^{1,2,2}([0,T)\times \mathbf{R}^2)\cap C([0,T]\times \mathbf{R}^2)$, let
$$g(t,x,\beta)=\tilde g(t,x,y).$$
Then
\begin{align*}
&\tilde g_t(t,x,y)=g_t(t,x,\beta)+g_\beta(t,x,\beta)\varphi_t(t,y),\\
&\tilde g_x(t,x,y)=g_x(t,x,\beta),\\
&\tilde g_y(t,x,y)=g_\beta(t,x,\beta)\varphi_y(t,y),\\
&\tilde g_{xx}(t,x,y)=g_{xx}(t,x,\beta),\\
&\tilde g_{xy}(t,x,y)=g_{x\beta}(t,x,\beta)\varphi_y(t,y),\\
&\tilde g_{yy}(t,x,y)=g_{\beta\beta}(t,x,\beta)\varphi^2_y(t,y)+g_\beta(t,x,\beta)\varphi_{yy}(t,y).
\end{align*}
%and
%\[\color{red}
%\begin{split}
%\varphi_{yy}(t,x,y)&={1\over\sigma^4}\mathbf{E}
%\left[\left(Z-\varphi(t,y)\right)^2Z^2 \mid Y_t=y\right]
%-\varphi_y(t,y){1\over\sigma^2}\mathbf{E}\left[Z^2\mid Y_t=y\right]
%\\&={1\over \sigma^4}\cov (\left(Z-\varphi(t,y)\right)^2,Z^2\mid Y_t=y).
%\end{split}
%\]

An admissible trading strategy can now be represented by a feedback function
\begin{equation}\label{eq:pibeta}
\pi: (t,x,\beta)\mapsto \pi(t,x,\beta)\trieq\tilde\pi(t,x,y)
\end{equation}
for some $\tilde\pi\in\tilde\Pi$.
Let $$\Pi=\{\pi\mid \pi\text{ satisfies \eqref{eq:pibeta} for some }\tilde\pi\in\tilde\Pi\}.$$

For every $z\in\range(Z)$ and $\pi\in\Pi$,  the infinitesimal generator $\mathcal{A}^{z,\pi}$ for the state processes $(X,\bbeta)$ is given by, for $f\in C^{1,2,2}([0,T)\times \mathbf{R}^2)\cap C([0,T]\times \mathbf{R}^2)$,
\begin{equation*}
	\begin{split}
	&\mathcal{A}^{z,\pi}f(t,x,\beta)\trieq \mathcal{\tilde A}^{z,\tilde\pi}\tilde f(t,x,y)\\
	=&f_t(t,x,\beta)+[rx+\pi(t,x,\beta)  (z-r)]f_x(t,x,\beta)\\
	&+\left[\varphi_t(t,y)+\left(z-\frac{\sigma^2}{2}\right)\varphi_y(t,y)+{1\over 2}\sigma^2\varphi_{yy}(t,y)\right]f_\beta(t,x,\beta)
	\\&+\frac{1}{2}\pi^2\sigma^2f_{xx}(t,x,\beta)
	+\frac{1}{2}\sigma^2\varphi^2_y(t,y) f_{\beta\beta}(t,x,\beta)+\pi \sigma^2\varphi_y(t,y)f_{x\beta}(t,x,\beta)\\
	=&f_t(t,x,\beta)+[rx+\pi(t,x,\beta)  (z-r)]f_x(t,x,\beta)\\
	&+\left[\varphi_t(t,\varphi^{-1}(t,\beta))+\left(z-\frac{\sigma^2}{2}\right)\varphi_y(t,\varphi^{-1}(t,\beta))+{1\over 2}\sigma^2\varphi_{yy}(t,\varphi^{-1}(t,\beta))\right]f_\beta(t,x,\beta)
	\\&+\frac{1}{2}\pi^2\sigma^2f_{xx}(t,x,\beta)
	+\frac{1}{2}\sigma^2\varphi^2_y(t,\varphi^{-1}(t,\beta)) f_{\beta\beta}(t,x,\beta)+\pi \sigma^2\varphi_y(t,\varphi^{-1}(t,\beta))f_{x\beta}(t,x,\beta).
	\end{split}
\end{equation*}

%The following assumptions on functions $\{g^z, z\in{\text{Range}(Z)}\}$ will be used in the formulation of the verification theorem.
Assumptions \ref{assu0:y}--\ref{assu3:y} are equivalent to the followings assumptions.
\begin{assumption}\label{assu0}
	For any $\pi\in\Pi$ and $(t,x,\beta)\in[0,T)\times\Rbf^2$,   there exists $\tilde{t}\in(t,T)$ such that, $\forall z\in\range(Z)$,  under the conditional probability $\mathbb{P}[\,\cdot\mid X^{z,\pi}_t=x,\bbeta^z_t=\beta]$, the following two families are supermartingales w.r.t. filtration $\{\F^W_{s}\}$,
	\begin{align*}
		&\left\{\int_t^sg^z_x(u,X^{z,\pi}_u,\bbeta^{z}_u) \pi(u,X^{z,\pi}_u,\bbeta^{z}_u) \mathrm{d}W_u\right\}_{s\in[t,\tilde{t}]},\\
		&\left\{\int_t^sg^z_\beta(u,X^{z,\pi}_u,\bbeta^{z}_u)
		\varphi_y(u,\varphi^{-1}(u,\bbeta^z_u))\mathrm{d}W_u\right\}_{s\in[t,\tilde{t}]}.
	\end{align*}
\end{assumption}

\begin{assumption}\label{assu1}
	For any $z\in\range(Z)$, $\pi\in\Pi$  and $(t,x,\beta)\in[0,T)\times\Rbf^2$,  there exists $\tilde{t}\in(t,T)$ such that,  under the conditional probability $\mathbb{P}[\,\cdot\mid X^{z,\pi}_t=x,\bbeta^z_t=\beta]$, the family
	\begin{equation*}\label{equ:assu1}
		\left\{\frac{1}{h}\int_t^{t+h}\mathcal{A}^{z,\pi} g^z(u,X^{z,\pi}_u,\bbeta^{z}_u)\mathrm{d}u\right\}_{0<h\le \tilde{t}-t}
	\end{equation*}
	is uniformly integrable.
\end{assumption}

\begin{assumption}\label{assu2}
	For any $\pi\in\Pi$ and $(t,x,\beta)\in[0,T)\times\Rbf^2$,  there exists $\tilde{t}\in(t,T)$ such that,
	under the conditional probability $\mathbb{P}[\,\cdot\mid \bbeta^z_t=\beta]$,
	the family
	\begin{equation*}\label{equ:assu2}
		\left\{\phi^\prime\left(g^Z(t,x,\beta)\right)\mathbf{E}\left[\left.\frac{1}{h}\int_t^{t+h}\mathcal{A}^{Z,\pi} g^Z(u,X^{Z,\pi}_u,\bbeta^{Z}_u)\mathrm{d}u\,\right|Z,X^{Z,\pi}_t=x,\bbeta^Z_t=\beta\right]\right\}_{0<h\le\tilde{t}-t}
	\end{equation*}
	is uniformly integrable  in the case of concave $\phi$, or, the family
	\begin{equation*}\label{equ:assu22}
		\left\{\phi^\prime\left(g^{Z, \pi_{t,h}}(t,x,\beta)\right)\mathbf{E}\left[\left.\frac{1}{h}\int_t^{t+h}\mathcal{A}^{Z,\pi} g^Z(u,X^{Z,\pi}_u,\bbeta^{Z}_u)\mathrm{d}u\,\right|Z,X^{Z,\pi}_t=x,\bbeta^Z_t=\beta\right]\right\}_{0<h\le\tilde{t}-t}
	\end{equation*}
	is uniformly integrable in the case of convex $\phi$.
\end{assumption}

\begin{assumption}\label{assu3}
	In the case of convex $\phi$, for any $z\in\range(Z)$, $\pi\in\Pi$  and $(t,x,\beta)\in[0,T)\times\Rbf^2$,
	$$\lim_{h\to 0^+}\mathbf{E} \left[
	g^z(t+h,X^{z,\pi}_{t+h},\bbeta^{z}_{t+h})
	\mid X^{z,\pi}_t=x, \bbeta^z_t=\beta\right]=g^z(t,x,\beta).$$
\end{assumption}

 Theorem \ref{thm:y} can be reformulated as the following $\bbeta$-version verification theorem.
\begin{theorem}[$\bbeta$-Version Verification Theorem]\label{thm}
	Assume that $U$ is an increasing concave $C^2$ function and $\phi$ is an increasing concave/convex $C^1$ function.
	Consider a family of functions
	$$\left\{g^z, z\in\range( Z)\right\}\subset C^{1,2,2}([0,T)\times \mathbf{R}^2)\cap C([0,T]\times \mathbf{R}^2)$$
	and an admissible strategy $\pi^*\in\Pi$.  Assume that they satisfy the following three conditions:
	\begin{description}
		\item[(a)] $\forall (t,x,\beta)\in[0,T]\times \mathbf{R}^2 $,
		\begin{align}
			&\sup_{\pi\in \Pi}\mathbf{E}\left[\phi^\prime\left(g^ Z(t,x,\beta)\right)\mathcal{A}^{ Z,\pi} g^ Z(t,x,\beta)\mid \bbeta_t=\beta\right]=0, \label{equ:hjb}
			\\ &\mathcal{A}^{z,\pi^*}g^z(t,x,\beta)=0, \label{equ:hjbg}
			\\ &g^z(T,x,\beta)=U(x);\label{equ:hjbg2}
		\end{align}
		\item[(b)] %$\forall (t,x,\beta)\in[0,T]\times \mathbf{R}^2 $ and $z\in\range(Z)$,
		$\forall z\in\range(Z)$, $\left\{g^z\left(t,X^{z,\pi^*}_t,\bbeta^{z}_t\right)\right\}_{t\in[0,T]}$ is a martingale w.r.t. filtration $\{\mathcal{F}^W_t\}$;
		\item[(c)] $\left\{g^z, z\in\range(Z)\right\}$ satisfies Assumptions \ref{assu0}--\ref{assu3}.
	\end{description}
	Then $\pi^*$ is an equilibrium strategy and, $\forall z\in\range(Z)$ and $(t,x,\beta)\in[0,T]\times \mathbf{R}^2$,
	$$\mathbf{E}\left[U(X^{\pi^*}_T)\mid Z=z,X^{\pi^*}_T=x,\bbeta_t=\beta\right]=g^z(t,x,\beta).$$
\end{theorem}

%\begin{proof} See Appendix \ref{app:thm}. \end{proof}

%Now we derive the equilibrium strategy $\pi^*$ from condition (a) of Theorem \ref{thm}. We can see from \eqref{equ:hjb}-- \eqref{equ:hjbg} that
%$$\pi^*\in\arg \sup_{\pi\in \Pi}\mathbf{E}\left[\phi^\prime\left(g^ Z(t,x,\beta)\right)\mathcal{A}^{ Z,\pi} g^ Z(t,x,\beta)\mid \bbeta_t=\beta\right].$$
%Then by the first-order condition of optimality, the equilibrium strategy $\pi^*$ satisfies
By \eqref{ofc:y}, the $\bbeta$-version equilibrium strategy $\pi^*$ satisfies
\begin{equation}\label{ofc}
	\begin{split}
		\pi^*(t,x,\beta)=&
		-\frac{\mathbf{E}\left[\phi^\prime(g^ Z(t,x,\beta))g_{x}^ Z(t,x,\beta)( Z-r) \mid  \bbeta_t=\beta\right]}
		{\sigma^2 \mathbf{E}\left[\phi^\prime(g^ Z(t,x,\beta))g^ Z_{xx}(t,x,\beta)  \mid \bbeta_t=\beta\right]}\\
		&-\frac{\varphi_y(t,y) \mathbf{E}\left[\phi^\prime(g^ Z(t,x,\beta))g^ Z_{x\beta}(t,x,\beta)\mid \bbeta_t=\beta\right]}
		{\mathbf{E}\left[\phi^\prime(g^ Z(t,x,\beta))g^ Z_{xx}(t,x,\beta) \mid \bbeta_t=\beta\right]}
	\end{split}
\end{equation}
if such a $\pi^*$ is admissible.

%Let $\cov(\,\cdot \mid \bbeta_t=\beta)$ and $\var(\,\cdot \mid \bbeta_t=\beta)$  denote the covariance and the variance under the conditional probability $\mathbb{P}(\,\cdot\mid \bbeta_t=\beta)$, respectively.
%Recalling $\varphi(t,y)=\mathbf{E}\left[Z \mid \bbeta_t=\beta\right]$, we can write \eqref{ofc} as
Similarly to \eqref{ofc2:y}, we can write \eqref{ofc} as
\begin{equation}\label{ofc2}
	\begin{split}
		\pi^*(t,x,\beta)=&
		\underbrace{-\frac{\mathbf{E}\left[\phi^\prime(g^ Z(t,x,\beta))g_{x}^ Z(t,x,\beta) \mid  \bbeta_t=\beta\right]}
		{\sigma^2 \mathbf{E}\left[\phi^\prime(g^ Z(t,x,\beta))g^ Z_{xx}(t,x,\beta)  \mid \bbeta_t=\beta\right]}\left(\beta-r\right)}_{\pi^M}
		\\& \underbrace{-\frac{\cov\left(\phi^\prime(g^ Z(t,x,\beta))g_{x}^ Z(t,x,\beta), Z-r \mid  \bbeta_t=\beta\right)}
		{\sigma^2 \mathbf{E}\left[\phi^\prime(g^ Z(t,x,\beta))g^ Z_{xx}(t,x,\beta)  \mid \bbeta_t=\beta\right]}}_{\pi^Z}\\
		&
		\underbrace{-\frac{\varphi_y(t,\varphi^{-1}(t,\beta)) \mathbf{E}\left[\phi^\prime(g^ Z(t,x,\beta))g^ Z_{x\beta}(t,x,\beta) \mid \bbeta_t=\beta\right]}
			{\mathbf{E}\left[\phi^\prime(g^ Z(t,x,\beta))g^ Z_{xx}(t,x,\beta) \mid \bbeta_t=\beta\right]}}_{\pi^{\bbeta}}.
	\end{split}
\end{equation}

\begin{remark}\label{rmk:hgd}
Note that $\varphi_y$ is the ratio of the volatility of $\bbeta_t$ to the volatility of $S_t$. Therefore,
 $\pi^{\bbeta}$ is formally linear w.r.t. the volatility of the revealing process $\bbeta$.
The volatility of $Y_t$ coincides with the volatility of $S_t$ and hence the ratio of the two volatilities is $1$.
So the effect to $\tilde\pi^Y$ from the volatility of $Y$ is implicit.
This observation confirms that the third term is related to the risk of the revealing process.
\end{remark}

\section{Closed-Form Solution: A Special Case}\label{sec:special}

\subsection{The Special Model}

Now we consider the case of Gaussian prior.
Assume that $Z$ is Gaussian:
\[ Z\sim N\left(\beta_0,\sigma_0^2\right),\]
where $\beta_0\in\mathbf{R} $ and $\sigma_0\in\mathbf{R}$ are known constants.  By \citet[Propositions 10--11 and Remark 5]{Bis2019}, the posterior distribution of $ Z$ given $\mathcal{F}_t^S$ is also Gaussian:
\[ Z|\mathcal{F}_t^S\sim N\left(\bbeta_t,\zeta(t)\right),\]
where
\begin{eqnarray*}
	\zeta(t)={\sigma^2\sigma_0^2\over\sigma^{2}+t\sigma^{2}_0}\quad\text{and}\quad
	\bbeta_t=\zeta(t)\left\{{Y_t-Y_0\over\sigma^2}+\frac{t}{2}+{\beta_{0}\over\sigma_0^{2} }\right\}.
	\end{eqnarray*}
{In this case, $\varphi(t,y)=\zeta(t)\left\{{y-Y_0\over\sigma^2}+\frac{t}{2}+{\beta_{0}\over\sigma_0^{2} }\right\}$.} Moreover, $\bbeta$ satisfies the following SDE:
$$\rd \bbeta_t={\zeta(t)\over\sigma}\rd W^S_t={\zeta(t)\over\sigma^2}[(Z-\bbeta_t)\rd t+\sigma \rd W_t].$$	
	
Obviously, $\bbeta$ and $Y$ generate the same filtration.  Therefore, the systems \eqref{eq:dS:gtY}--\eqref{eq:dY} and \eqref{equ:x:mu}--\eqref{equ:y} can be reformulated with $\bbeta$ instead of $Y$. Actually,
the system \eqref{eq:dS:gtY}--\eqref{eq:dY} is equivalent to
\begin{align*}%\label{eq:dS:beta}
&{\rd S_t\over S_t}=\bbeta_t\rd t+\sigma\rd W^S_t,\\
                       %\label{eq:dbeta}
&\rd \bbeta_t%=\left(\beta(t)-{\sigma^2\over2}\right)\rd t+\sigma \rd W^S_t
	={\zeta(t)\over\sigma} \rd W^S_t.
\end{align*}
Then equation \eqref{equ:x} can be reformulated as
\begin{align}\label{equ:x:beta}
		\mathrm{d}X^{\pi}_t=rX^{\pi}_t\mathrm{d}t+\pi(t,X^{\pi}_t,\bbeta_t)[\bbeta_t-r]\mathrm{d}t
		+\pi(t,X^{\pi}_t,\bbeta_t)\sigma\mathrm{d}W^S_t. %,\qquad t\in[0,T],\\
%\label{equ:beta1}
%		&\rd \bbeta_t={\zeta(t)\over\sigma} \rd W^S_t,\quad t\in[0,T].
\end{align}
Then the inside system \eqref{equ:x:mu}--\eqref{equ:y} can be reformulated as the following $\bbeta$-version one:
\begin{align}\label{equ:xbeta:mu}
		&\mathrm{d}X^{z,\pi}_t=rX^{z,\pi}\mathrm{d}t+\pi(t,X^{z,\pi}_t,\bbeta^{z}_t)(z-r)\mathrm{d}t+\pi(s,X^{z,\pi}_t,\bbeta^{z}_t)\sigma\mathrm{d}W_t,%\qquad t\in[0,T],
\\
\label{equ:beta:mu}
	&	\rd \bbeta^{z}_t={\zeta(t)\over\sigma^2} \left[(z-\bbeta^{z}_t)\rd t+\sigma \rd W_t\right]. %,\quad t\in[0,T].
\end{align}

%{\color{red}Replacing $(\mathcal{A},Y)$ with $(\mathcal{A},\bbeta)$, we obtain a reformulation of the verification theorem, which will be called \textit{$\bbeta$-version verification theorem}.}

Assume that utility function $U$ is CARA, i.e.,
$$U(x)=-{1\over k}e^{-kx},\quad x\in\mathbb{R},$$
where $k>0$.  Function $\phi$ is given by
$$\phi(u)=
\begin{cases}
-{1\over\alpha}(-u)^\alpha & \text{if }\alpha\ne0,\\
-\log(-u) &\text{if }\alpha=0,
\end{cases}\qquad
u\in(-\infty,0),
$$where $\alpha\in\mathbf{R}$.
Obviously, $\phi^\prime(u)=(-u)^{\alpha-1}$ and $\phi^{\prime\prime}(u)=(1-\alpha)(-u)^{\alpha-2}$. Therefore,
$\phi$ is strictly concave if $\alpha>1$, linear if $\alpha=1$, and strictly convex if $\alpha<1$.

\begin{definition}[Admissible Strategy]\label{def:ad}
Given the above Gaussian prior, for the above specific $\phi$ and $U$, we say \eqref{eq:pibeta} is an admissible strategy if it satisfies the following two conditions:
\begin{description}\label{def:admissible}
\item[(i)] SDE \eqref{equ:x:beta} has a unique strong solution $X^{\pi}$;
% with given states  $X^{\pi}_t=x$ and  $\bbeta_t=\beta$ for every $(t,x,\beta)$;
\item[(ii)]  For every $(t,x,\beta)\in[0,T)\times\Rbf^2$, $\epsilon>0$ and $\rho>2$, there exist constants $\delta>2$, $\tilde{t}\in(t,T)$  and $C>0$ such that, for every $z\in\range(Z)$,
\begin{align}
	&\sup_{s\in[t,\tilde{t}]}\mathbf{E}
	\left[\left.e^{-\delta k e^{r(T-s)}X^{z,\pi}_s}\,\right|\,X^{z,\pi}_t=x,\bbeta^z_t=\beta\right]\leq Ce^{\epsilon z^2}, \label{inequ01}
	\\ &\sup_{s\in[t,\tilde{t}]}\mathbf{E}
	\left[\left.|\pi(s,X^{z,\pi}_s,\bbeta^{z}_s)|^\rho\, \right|\, X^{z,\pi}_t=x,\bbeta^z_t=\beta \right]\leq Ce^{\epsilon z^2}.\label{inequ02}
\end{align}
\end{description}
\end{definition}

The following proposition shows the abundance of admissible strategies.

\begin{proposition}\label{prop:admissible}
Assume that \eqref{eq:pibeta} satisfies condition (i) of Definition \ref{def:admissible} and that there exists $C>0$  such that
\begin{equation}
	|\pi(t,x,\beta)|\leq C\left(1+|\beta|+\sqrt{\log(1+|x|)}\right),\quad (t,x,\beta)\in[0,T)\times\Rbf^2.\label{inequ022}
\end{equation}
Then $\pi$ is admissible.
\end{proposition}

\proof See Appendix \ref{app:admissible}. \qed

The set of admissible strategies is denoted by $\Pi$.
For every $z\in\range( Z)$ and $\pi\in\Pi$,  the infinitesimal generator $\mathcal{A}^{z,\pi}$ for the inside system \eqref{equ:xbeta:mu}--\eqref{equ:beta:mu} is given by,
for
$f\in C^{1,2,2}([0,T)\times \mathbf{R}^2)\cap C([0,T]\times \mathbf{R}^2)$,
\begin{equation}
	\mathcal{A}^{z,\pi} f=f_t+\left[rx+\pi(z-r)\right]f_x+
	{\zeta\over \sigma^2}(z-\beta)f_\beta +\frac{1}{2}\sigma^2 \pi^2f_{xx}+\frac{1}{2}{\zeta^2\over\sigma^2}f_{\beta\beta}
	+\pi\zeta f_{x\beta}.\label{theclosedform}	
\end{equation}

\subsection{ Ansatz}

For every $\pi\in\Pi$, its inside expected utility is
\begin{equation*}
g^{z,\pi}(t,x,\beta)= \mathbf{E} \left[\left.U\left(X^{z,\pi}_T\right)\,\right |\,X^{z,\pi}_t=x,\bbeta^z_t=\beta\right].
\end{equation*}
We make the following ansatz on the equilibrium value $g^z(t,x,\beta)$ for the inside expected utility:
\begin{itemize}
\item For every $z$, there exists some function $f^z:[0,T]\times\Rbf\to\Rbf$ such that
\begin{equation}\label{eq:ans:gz}
g^z(t,x,\beta)=-\frac{1}{k}\exp\left\{-ke^{r(T-t)}x+f^z(t,\beta)\right\},\quad (t,x,\beta)\in[0,T]\times\Rbf^2;
\end{equation}
\item For every $t$, $f^z(t,\beta)$ is a quadratic form of $(\beta,z)$, i.e., there exist $C^1$ functions $m_i$, $i=1,\dots,6$, such that
\begin{equation}\label{eq:ans:fz}
f^z(t,\beta)=\frac{1}{2} m_1(t)\beta^2+m_2(t)\beta z+\frac{1}{2} m_3(t)z^2+m_4(t)\beta
+m_5(t)z+m_6(t).
\end{equation}
\end{itemize}

Let
\[
\begin{split}
	&a_1(t,m_2,m_3)=\frac{(1+\zeta(t)m_2)(\zeta(t)^{-1}+\!\alpha m_2)}{\zeta(t)^{-1}-\alpha m_3},\\
	&a_2(t,m_2,m_3)=\frac{\alpha
		(1+\!\zeta(t)m_2)}{\zeta(t)^{-1}-\alpha m_3}.
\end{split}
\]
Consider the following system of ODEs:
\begin{eqnarray}\label{equ:ode}
\begin{cases}
		\sigma^2 m_1^\prime(t)=2\zeta(t)m_1(t)-a_1^2\left(t,m_2(t),m_3(t)\right),
		\\
		\sigma^2 m_2^\prime(t)=(1+\zeta(t)m_2(t))a_1\left(t,m_2(t),m_3(t)\right) +\zeta(t)m_2(t),
		\\
		\sigma^2 m_3^\prime(t)=-2\zeta(t)m_2(t)-\zeta^{2}(t)m_2^2(t),
		\\
		\sigma^2 m_4^\prime(t)=\zeta(t)m_4(t)-a_1\left(t,m_2(t),m_3(t)\right) a_2\left(t,m_2(t),m_3(t)\right)m_5(t) -r \zeta(t)m_1(t),
		\\
		\sigma^2 m_5^\prime(t)=(1+\zeta(t)m_2(t))a_2\left(t,m_2(t),m_3(t)\right)m_5(t)-r(1+\zeta(t)m_2(t)),
		\\
		\sigma^2 m_6^\prime(t)=-\frac{1}{2} a_2^2\left(t,m_2(t),m_3(t)\right)m_5^2(t)+\frac{1}{2}r^2
		 -\frac{1}{2}\zeta^{2}(t) m_1(t)-r \zeta(t)m_4(t),\\
		m_1(T)=m_2(T)=\dots=m_6(T)=0.
\end{cases}
\end{eqnarray}

%{\color{red}
%\begin{lemma}\label{lma:m4m5}
%Assume that $\{m_i, i=1,\dots,6\}$  solves \eqref{equ:ode} and $\zeta(t)^{-1}-\alpha m_3(t)\neq0$ for all $t\in[0,T]$. We have
%\begin{equation}
%	\begin{cases}\label{equ:m4m5}
%		m_4=-rm_1-rm_2,\\
%		m_5=-rm_2-rm_3.
%	\end{cases}
%\end{equation}
%\end{lemma}
%\proof See Appendix \ref{app:lma:m4m5}.\qed
%}

\begin{lemma}\label{lma:ansatz}
Assume that $\{m_i, i=1,\dots,6\}$  solves \eqref{equ:ode} and $\zeta(t)^{-1}-\alpha m_3(t)>0$ for all $t\in[0,T]$. Then
\begin{equation}
	\begin{cases}\label{equ:m4m5}
		m_4=-rm_1-rm_2,\\
		m_5=-rm_2-rm_3.
	\end{cases}
\end{equation}
Let
\begin{equation}\label{equ:u*}
\pi^*(t,x,\beta)=\frac{e^{-r(T-t)}}{k\sigma^2}\left\{a_1(t,m_2(t),m_3(t))+\zeta(t)m_1(t)\right\} (\beta-r).
\end{equation}
Let $g^z$ be given by \eqref{eq:ans:gz}--\eqref{eq:ans:fz}.
Then %$\{g^z,z\in\range(Z)\}$ and $\pi^*$ satisfies the following condition:
for all $ (t,x,\beta)\in[0,T]\times \mathbf{R}^2$,
\begin{align}
&\sup_{\pi\in \Pi}\mathbf{E}\left[\phi^\prime\left(g^ Z(t,x,\beta)\right)\Ac^{ Z,\pi} g^ Z(t,x,\beta)\mid \bbeta_t=\beta\right]=0, \label{equ:hjb:beta}
\\ &\Ac^{z,\pi^*}g^z(t,x,\beta)=0, \label{equ:hjbg:beta}
\\ &g^z(T,x,\beta)=-{1\over k}e^{-kx}.\label{equ:hjbg2:beta}
\end{align}
\end{lemma}

\proof See Appendix \ref{app:lma:ansatz}. \qed

%{
%\color{red}$f^z(t,\beta)$ satisfies
%$$f^z(t,\beta)=\frac{1}{2} m_1(t)(\beta-r)^2+m_2(t)(\beta-r) (z-r)+\frac{1}{2} m_3(t)(z-r)^2+\hat{m}_6(t).$$
%}

%And ODE \eqref{ODEb} has a unique explicit solution:
%\[b(t)=ke^{r(T-t)},\]
%which coincides with the coefficient of $x$ in many portfolio selection problems with CARA utility ({{see \cite{browne1995optimal}, \cite{vigna2014efficiency}}}).
%\begin{remark}
%Although  ODEs~\eqref{equ:ode} can not be solved explicitly,
%
%We observe that $r$ and $k$ do not influence ODEs~\eqref{equ:ode}. $\pi^*(t,\beta)$ depends on $r$ and $k$ only through function $b(t)^{-1}$.
%\end{remark}

\subsection{Closed-Form Solution of ODE System \eqref{equ:ode}}

We can first solve the following system of ODEs:
\begin{equation}\label{equ:ode:m2m3}
\begin{cases}
\sigma^2m_2^\prime(t)=(1+\zeta(t)m_2(t))a_1\left(t,m_2(t),m_3(t)\right) +\zeta(t)m_2(t),\\
\sigma^2m_3^\prime(t)=-2 \zeta(t)m_2(t)-\zeta^{2}(t)m_2^2(t),\\
m_2(T)=m_3(T)=0.
\end{cases}
\end{equation}
Then we solve the other equations in \eqref{equ:ode}  one by one.

The following proposition provides the solution of \eqref{equ:ode:m2m3} in closed form.
\begin{proposition}\label{prop:m2m3}
There exists a unique $\alpha^*<0$ such that
\begin{equation}\label{equ:alpha*}
\left(1+\sigma_0^{-2}\sigma^2T^{-1}\right)\int_0^\infty e^{-x}(1+(1-\alpha^*)x)^{\frac{\alpha^*}{1-\alpha^*}}\rd x=1.
%\left(1+\sigma_0^{-2}\sigma^2T^{-1}\right)\int^{1}_0e^{\frac{-s}{1-s}}\left(1-s\right)^{-\frac{\alpha^*-2}{\alpha^*-1}}\left(1-\!\alpha^*
%s\right)^{-\frac{\alpha^*}{\alpha^*-1}}
%\rd s=1.
\end{equation}
For every $\alpha>\alpha^*$, \eqref{equ:ode:m2m3} has a unique solution. More precisely, the closed-form solution is given as follows.
\begin{description}
\item[(a)] For $\alpha=0$,
\begin{equation*}
	\begin{cases}
		m_2(t)={\zeta(t)^{-1}\over 1+\log{\zeta(t)\over\zeta(T)}}-\zeta(t)^{-1},\\
		m_3(t)=\zeta(T)^{-1}\int^{\log{\zeta(t)\over\zeta(T)}}_0 e^{-x}(1+x)^{-2}\rd x-{T-t\over\sigma^2}.
	\end{cases}
\end{equation*}

\item[(b)] For $\alpha=1$,
\begin{equation*}
	\begin{cases}
		m_2(t)=\frac{\zeta(t)^{-1}}{1+\sigma^{-2}\zeta(T)(T-t)}-\zeta(t)^{-1},\\
		m_3(t)=\zeta(t)^{-1}-\frac{\zeta(T)^{-1}}{1+\sigma^{-2}\zeta(T)(T-t)}.
	\end{cases}
\end{equation*}

\item[(c)] For $\alpha\in(\alpha^*,\infty)\setminus\{0,1\}$,
\begin{equation*}
	\begin{cases}
		m_2(t)=-\zeta(t)^{-1}\Psi^{-1}\left((\sigma^{-2}\zeta(T)(T-t)\right),\\
		m_3(t)={1\over \alpha}\left
		(\zeta(t)^{-1}-\zeta(T)^{-1}e^{\frac{\zeta(t){m}_2(t)}{1+\zeta(t){m}_2(t)}}\left[1+\zeta(t){m}_2(t)\right]^{\frac{\alpha}{\alpha-1}}\left[1+\!\alpha
		\zeta(t){m}_2(t)\right]^{-\frac{1}{\alpha-1}}\right),
%		
%		
%		\frac{1}{\alpha}\zeta(t)\times
%		\\&\left(1-e^{\frac{-\Psi^{-1}\left(\sigma^{-2}\zeta(T)(T-t)\right)}{1-\Psi^{-1}\left(\sigma^{-2}\zeta(T)(T-t)\right)}}\left[1-\Psi^{-1}\left(\sigma^{-2}\zeta(T)(T-t)\right)\right]^{\frac{\alpha}{\alpha-1}}\left[1-\!\alpha
%		\Psi^{-1}\left(\sigma^{-2}\zeta(T)(T-t)\right)\right]^{-\frac{1}{\alpha-1}}\right),
	\end{cases}
\end{equation*}
where
\[\Psi(x)=\int^x_0\frac{1}{\psi(s)}\rd s, \quad x\in\left[0,{1\over\alpha\vee 1}\right]\]
and $$\psi(x)=e^{\frac{x}{1-x}}\left(1-x\right)^{\frac{\alpha-2}{\alpha-1}}\left(1-\alpha
x\right)^{\frac{\alpha}{\alpha-1}}, \quad x\in\left[0,{1\over\alpha\vee 1}\right).$$
\end{description}
Moreover, for every $t\in[0,T)$,
\begin{align*}
		&0>\zeta(t){m}_2(t)>-\frac{1}{\alpha\vee 1},\\
		&{m}_3(t)<0,\\
		&\zeta(t)^{-1}-\alpha m_3(t)>0.
	\end{align*}
\end{proposition}

\proof See Appendix \ref{app:prop:m2m3}. \qed

Now we solve the other equations in \eqref{equ:ode}  one by one. To this end, for $i\in\{1,2\}$, let
$$A_i(t)=a_i(t,m_2(t),m_3(t)),\quad t\in[0,T].$$
By Proposition \ref{prop:m2m3}, we know that
$$\inf_{t\in[0,T]}\{\zeta(t)^{-1}-\alpha m_3(t)\}>0$$
and hence both of $A_1$ and $A_2$ are bounded on $[0,T]$.
Then the closed-form solutions for $m_1$, $m_5$, $m_4$ and $m_6$  are given as follows.
\begin{eqnarray}
	\begin{cases}
		m_1(t)=\sigma^{-2}{\zeta(t)}^{-2}\int_t^T\left[\zeta^{2}(s) A_1^2(s)\right]\rd s,
		\\
		%m_5(t)=r\sigma^{-2}\int_t^T [1+\zeta(s)m_2(s)]\exp\left(-\int_t^s \sigma^{-2} {(1+\zeta(u)m_2(u))} A_2(u)\rd u \right)\rd s,
		m_4(t)=-rm_1(t)-rm_2(t),
		\\
		%m_4(t)=\sigma^{-2}\zeta(t)^{-1}\int_t^T\left[\zeta(s) A_1(s) A_2(s)m_5(s)+r \zeta^{2}(s)m_1(s)\right]\rd s,
		m_5(t)=-rm_2(t)-rm_3(t),
		\\
		m_6(t)=\sigma^{-2}\int_t^T\left[\frac{1}{2} A_2^2(s)m_5^2(s)-\frac{1}{2}r^2
		 +\frac{1}{2}\zeta^{2}(s) m_1(s)+r \zeta(s)m_4(s)\right]\rd s.
	\end{cases}
\end{eqnarray}
%We can explicitly obtain $m_1(t)$, $m_5(t)$, $m_4(t)$ and $m_6(t)$  successively.

In particular, for $\alpha=0$,
\begin{eqnarray*}%\color{red} \text{double check. it seems inconsistent with $f_\alpha$.}\\
	%\begin{cases}
		{m}_1(t)={\zeta(T)\over \zeta^2(t)} \int^{-1}_{-1-\log {\zeta(t) \over \zeta(T)}}e^{-x-1}x^{-2}\rd x;
	%	m_4(t)=-{r\zeta(T)\over \zeta^2(t)} \int^{-1}_{-1-\log {\zeta(t) \over \zeta(T)}}e^{-x-1}x^{-2}\rd x-{r\zeta(t)^{-1}\over 1+\log{\zeta(t)\over\zeta(T)}}+r\zeta(t)^{-1},\\
	%	m_5(t)=-{r\zeta(t)^{-1}\over 1+\log{\zeta(t)\over\zeta(T)}}+r\zeta(t)^{-1}-r\zeta(T)^{-1}\int^{\log{\zeta(t)\over\zeta(T)}}_0 e^{-x}(1+x)^{-2}\rd x+r{T-t\over\sigma^2}.\\
	%\end{cases}
\end{eqnarray*}
for $\alpha=1$,
\begin{eqnarray*}%\color{red} \text{double check}\\
	%\begin{cases}
		{m}_1(t)=\frac{\sigma^{-2}{\zeta(t)}^{-2}\zeta^2(T)(T-t)}{1+\sigma^{-2}\zeta(T)(T-t)}.
		%{m}_4(t)=r\sigma^{-2}\zeta(t)^{-1}\zeta(T)(T-t)-\frac{r\sigma^{-2}\zeta(t)^{-1}\zeta(T)(T-t)}{1+\sigma^{-2}\zeta(T)(T-t)},\\
		%{m}_5(t)=\frac{r\sigma^{-2}(T-t)}{1+\sigma^{-2}\zeta(T)(T-t)}.
	%\end{cases}
\end{eqnarray*}

\subsection{Equilibrium Solution}

The following theorem provides an equilibrium solution in closed form.
\begin{theorem}\label{thm:result}
Assume that $\alpha>\max\{\alpha^*, 1-0.5\sigma_0^{-2}\sigma^2T^{-1}\}$, where $\alpha^*$ is given by Proposition \ref{prop:m2m3}. Let $m_i$, $i=1,\dots,6$, be given by Proposition \ref{prop:m2m3}. Let $\pi^*$ be given by \eqref{equ:u*}.
%\[
%\begin{split}\pi^*(t,x,\beta)=
%\color{red} \text{ recover the formula that the appendix directly proves}
%	\end{split}
%\]
Then $\pi^*$ is an equilibrium solution and the equilibrium value function for the inside expected utility is given by  \eqref{eq:ans:gz}--\eqref{eq:ans:fz}.
%\[\begin{split}
%g^z(t,x,\beta)=&\mathbf{E}\left[\left.U(X^{z,\pi^*}_T)\,\right|\,Z=z, X^{z,\pi^*}_t=x,\bbeta^z_t=\beta\right]\\=&
%-\frac{1}{k}\!\exp\left\{\!-ke^{r(T-t)}x\!+\!\frac{1}{2} m_1(t)\beta^2\!+\! m_2(t)\beta z\right.
%\\&\left.\!+\!\frac{1}{2}m_3(t) z^2+\!m_4(t)\beta
%\!+  m_5(t)z\!+\!m_6(t)\right\}.
%\end{split}
%\]
\end{theorem}
\begin{proof}
	See Appendix \ref{app:thm:result}.
\end{proof}

By \eqref{ofc2}, we have the following decomposition.
\[
\begin{split}\pi^*(t,x,\beta)=
\left[1+\underbrace{\frac{\alpha m_2(t)+\alpha m_3(t)}{\zeta(t)^{-1}-\alpha m_3(t)}}_{h^Z}
		+\underbrace{\zeta(t)\left[\frac{ \zeta(t)^{-1}m_2(t)+\alpha m_2^2(t)}{\zeta(t)^{-1}-\alpha m_3(t)}+m_1(t)\right]}_{h^{\bbeta}}\right]\frac{e^{-r(T-t)}}{k\sigma^2}(\beta-r),	\\
%	=&\color{red}\frac{e^{-r(T-t)}}{k\sigma^2}\left\{\frac{(1\!+\!\zeta(t)m_2(t))(\zeta(t)^{-1}\!+\!\alpha
%			m_2(t))}{\zeta(t)^{-1}\!\!-\!\!\alpha m_3(t)}+\zeta(t)m_1(t)\right\} (\beta-r)
	 (t,x,\beta)\in[0,T)\times\Rbf^2,
	\end{split}
\]
{where $\frac{e^{-r(T-t)}}{k\sigma^2}(\beta-r)$ is the myopic demand $\pi^M$, $h^Z={\pi^Z\over\pi^M}$ and $h^{\bbeta}={\pi^{\bbeta}\over\pi^M}$ are the ratios of the hedging demands to the myopic demand.}

\section{Numerical Examples}

We estimate the parameters of the risky asset using  data of the S\&P 500 index daily
closing prices from Center for Research in Security Prices (CRSP)  from January 2017 to December 2021 by maximum likelihood estimation.   During the period, $\beta_0=17.2\%$, $\sigma=19.2\%$ and $\sigma_0=12.1\%$. The investment period is taken as 2 years. %Over the same period, the average annual Treasury bill return was $r=1.42\%$.  The risk aversion parameter is $k=1.5$.

%Figures \ref{fig:alpha}--\ref{fig:sigma02} display  hedging ratios at time 0, where $h=h^Z+h^{\bbeta}$.

Figure \ref{fig:alpha}  plots $h^Z$, $h^{\bbeta}$ and $h\trieq h^Z+h^{\bbeta}$ v.s. $\alpha$. We can see that the hedging ratio $h^Z$ is decreasing w.r.t. $\alpha$: the more ambiguity averse the investor becomes, the smaller $h^Z$ and $h$ are.  Such monotonicity is quite reasonable. In particular, when $\alpha$ goes to $\infty$, $h$ goes to $-1$, which means that the optimal holding of the ambiguous asset for an extremely ambiguity averse investor is zero.
%Furthermore, when $\alpha$ goes to $\alpha^*$, $h$ goes to infinity, which implies that the optimal holding of the ambiguous asset for an extremely ambiguity seeking investor is infinite.

Figures \ref{fig:sigma0}--\ref{fig:sigma02} plot the hedging ratios v.s. $\sigma_0$ for some fixed $\alpha$.
We can see that $h^Z$ is decreasing w.r.t. $\sigma_0$ in the case $\alpha>0$ and increasing w.r.t. $\sigma_0$ in the cases $\alpha<0$.  Such monotonicity is easy to understand for $\alpha>1$ and $\alpha<0$, since, in the case $\alpha>1$ ($\alpha<0$),  $\phi$ is concave (convex) and displays ambiguity aversion (seeking).  However, it seems a puzzle for $\alpha\in(0,1)$: in this case, $\phi$ is convex and displays ambiguity seeking so that $h^Z$ is expected to be \textit{increasing} w.r.t. $\sigma_0$ as in the case $\alpha<0$. But Figure \ref{fig:sigma01} shows that it is \textit{decreasing}!

To explain such a puzzle, we take a closer look at the representation of smooth ambiguity preference (at time $t=0$):
$$\Ebf[\phi(\Ebf[U(X_T)|Z])].$$
We can see that $\phi(\Ebf[U(X_T)|Z])$ is function of the conditional expected utility $\Ebf[U(X_T)|Z]$.
The shape of $\phi$ represents the investor's attitude toward the risk measured in \textit{utility} scale.
What the convexity of $\phi$ represents is that the investor is seeking for the risk of the conditional expected utility $\Ebf[U(X_T)|Z]$.
Let $$V=\phi\circ U.$$ Then the smooth ambiguity preference is represented by
$$\Ebf[V(C_{X_T})]\text{ with } C_{X_T}\trieq U^{-1}(\Ebf[U(X_T)|Z]),$$
where $C_{X_T}$ is the conditional certainty equivalent of $X_T$. $C_{X_T}$ is a monetary payoff.
The shape of $V$ represents the investor's attitude toward the risk measured in \textit{monetary} scale.
In our example,
\begin{eqnarray}
	V(x)=\phi\circ U(x)=
	\begin{cases}
    -{1\over \alpha k^\alpha}e^{-k\alpha x}  & \text{if }\alpha\ne0, \\
     kx+\log k & \text{if }\alpha=0.
\end{cases}
\end{eqnarray}
%Define $\phi\equiv V\circ U^{-1}$, the KMM model can also be formulated as the certainty equivalent at time $t$ to receiving $X^{Z,\pi}_T$ as follows
%\begin{equation}\label{equ:ced}
%	V^{-1}\left(\mathbf{E}\left[ V \left(U^{-1}\left(\mathbf{E}\left[ U(X^{Z,\pi}_T)| Z, \mathcal{F}_t^S\right]\right)\right )|\mathcal{F}_t^S\right] \right).
%\end{equation}
%Here
%\begin{eqnarray}
%	V(x)=\phi\circ U(x)=-{1\over \alpha k^\alpha}e^{-k\alpha x}.
%\end{eqnarray}
%In \eqref{eq:sm:t:J}, $\phi$ represents the attitude of the investor  measured on a utility scale. In \eqref{equ:ced}, the certainty equivalent $U^{-1}\left(\mathbf{E}\left[ U(X^{Z,\pi}_T)| Z, \mathcal{F}_t^S\right]\right)$ for the outcome of the utility function is monetary. In fact, $V$ captures the attitude  of the investor in monetary scale. The shape of $V$ characterizes the attitude towards the variation in the certainty equivalent.
When $\alpha>0$ ($\alpha<0$), $V$ is concave (convex) and the investor dislikes (likes) the risk of the conditional certainty equivalent.
Moreover, $\sigma_0$ is the risk of monetary return $Z$. Therefore, it is reasonable that $h^Z$ is decreasing (increasing) w.r.t. $\sigma_0$ for $\alpha>0$ ($\alpha<0$).%Note that, $h^Z=0$ for $\alpha=0$, which is natural since $V$ is risk neutral
\footnote{
The exponential-power specification for $(U,\phi)$ has also been used in \cite[p. 1333]{Gollier11}. The closed form of the conditional certainty equivalent in \cite{Gollier11} implies that its risk is represented by $\sigma_0$ (up to a multiplier).
%Hence, the hedging coefficient 1 decreases w.r.t. $\sigma_0$ for $\alpha\in(0,1)$.
}

Figures \ref{fig:alpha}--\ref{fig:sigma02} show that $h^{\bbeta}$ is positive in the case $\alpha>1$ and negative in the case $\alpha<1$.  The hedging ratios are zero when there is no uncertainty over $Z$ ($\sigma_0=0$). $h^{\bbeta}$ depends on the ratio of the conditional expectations in \eqref{ofc2}, which has a complicated relation with $Z$.  Thus, it is not monotone w.r.t. $\alpha$ or $\sigma_0$.

Figures \ref{fig:sigma0}--\ref{fig:sigma02} also show that $h$ goes to -1 (resp. infinity) when $\sigma_0$ goes to infinity in the case $\alpha>0$ (resp. $\alpha<0$). Furthermore, the results in  \citet[Lemma 1]{Gollier11} can also be observed in Figure \ref{fig:alpha}. As $h$ is always larger than $-1$, the sign of the optimal investment is the same as the sign of the excess equity premium $\beta-r$. Thus, the demand for the ambiguous asset is positive (zero/negative) if the equity premium is positive (zero/negative).

\begin{figure}[htbp]
	\begin{minipage}{0.5\textwidth}
		\centering
		\includegraphics[totalheight=6cm]{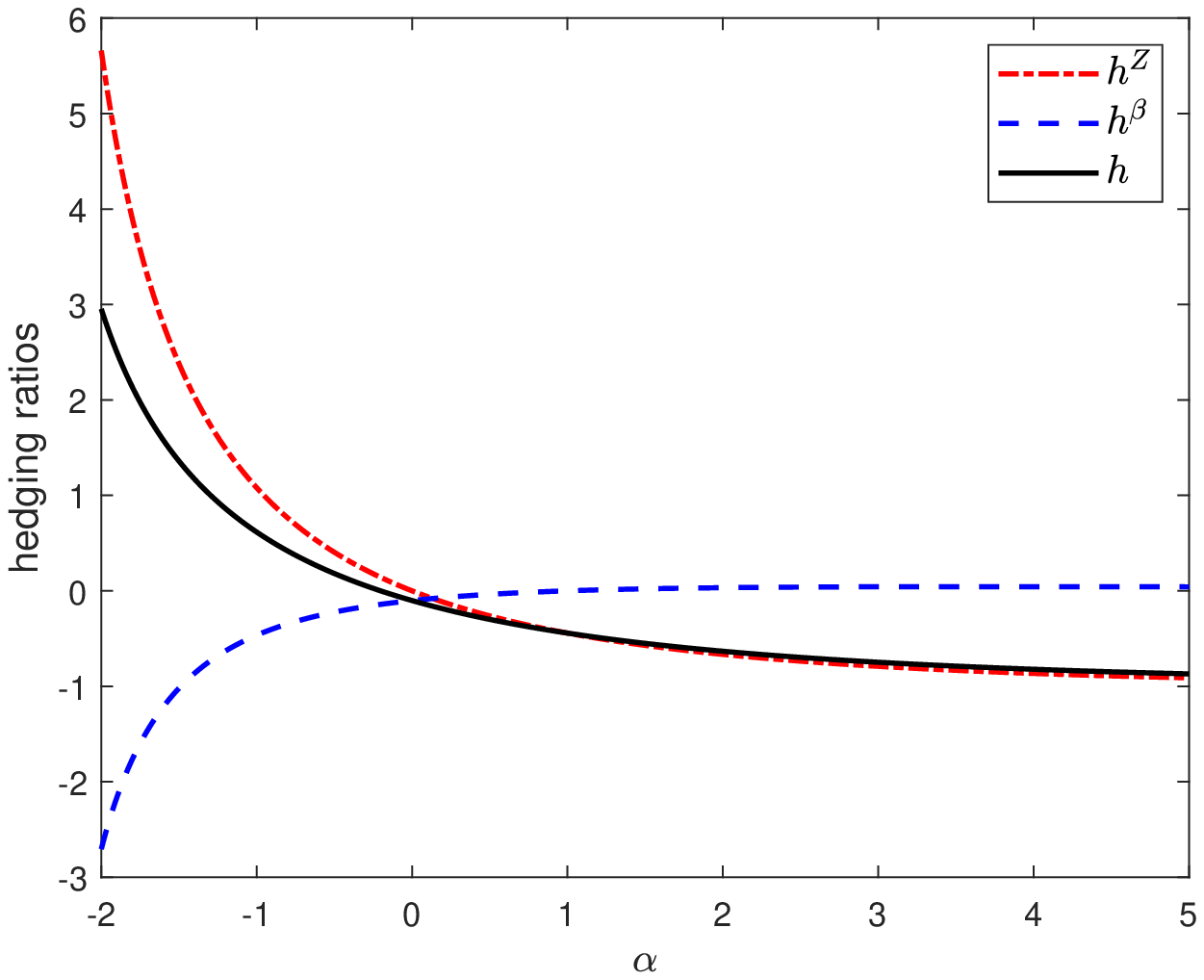}
		\caption{Effects of $\alpha$.}
		\label{fig:alpha}
	\end{minipage}\hfill
	\begin{minipage}{0.5\textwidth}
		\centering
		\includegraphics[totalheight=6cm]{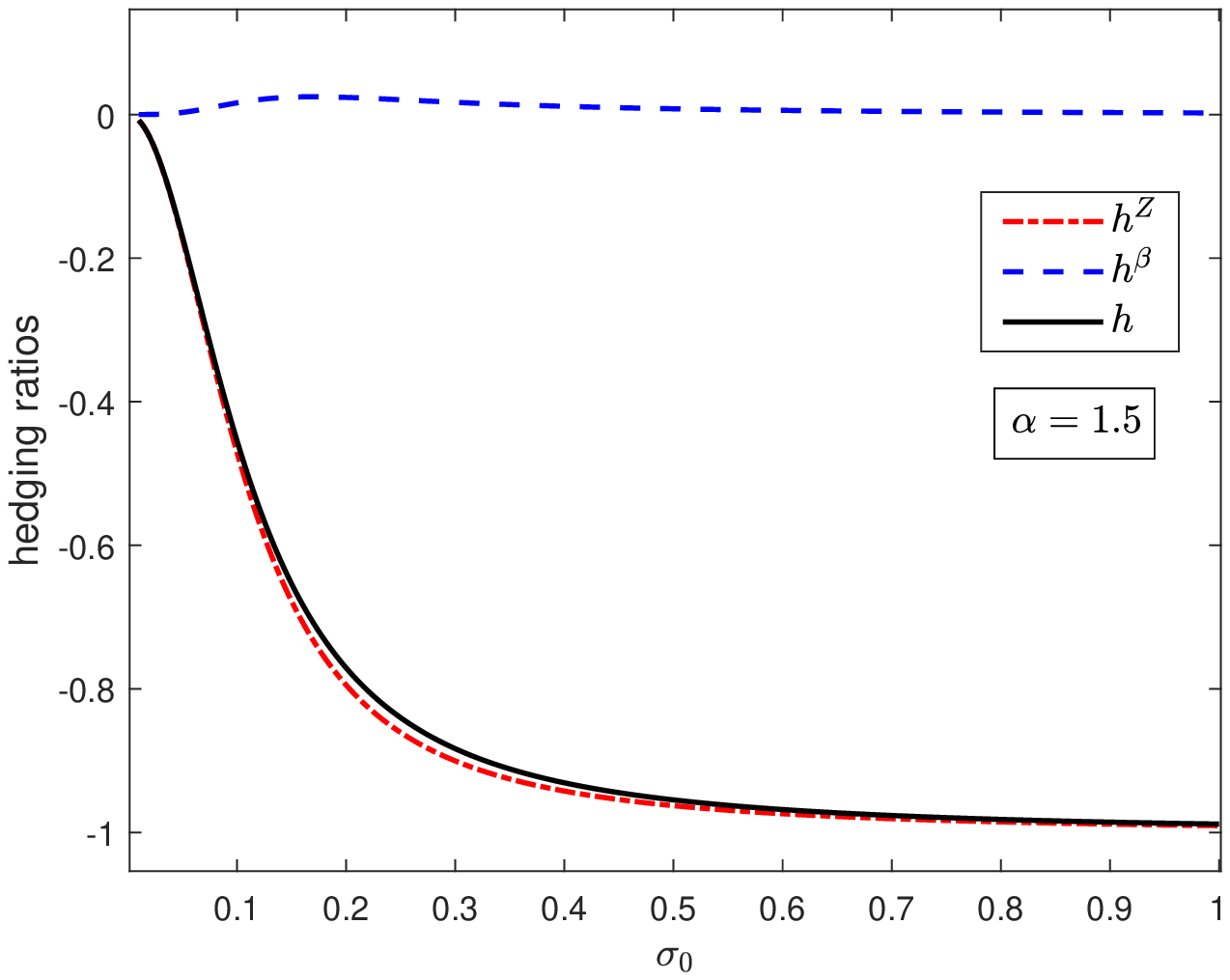}
		\caption{Effects of $\sigma_0$ ($\alpha=1.5$).}
		\label{fig:sigma0}
	\end{minipage}\hfill
	\begin{minipage}{0.5\textwidth}
		\centering
		\includegraphics[totalheight=6cm]{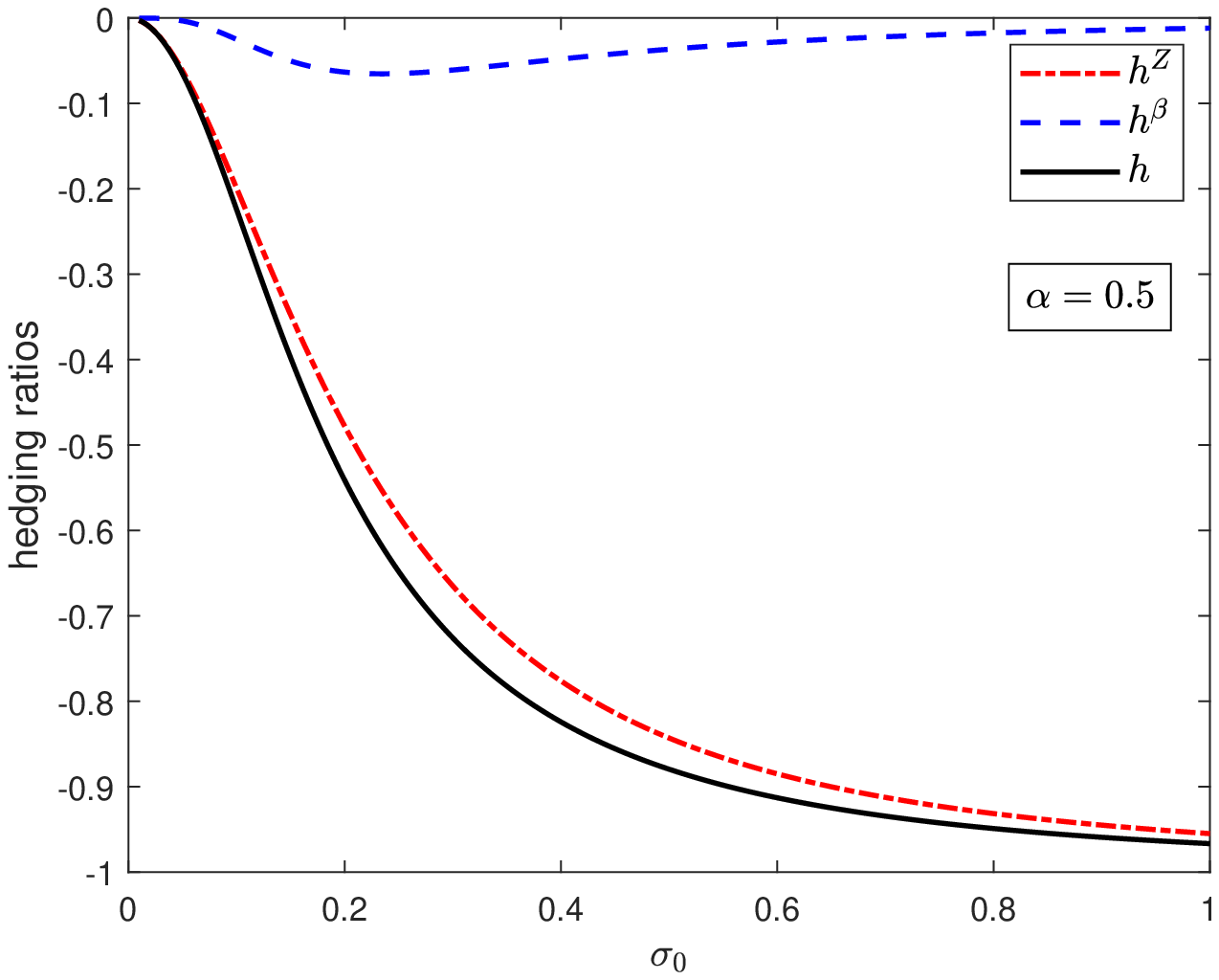}
		\caption{Effects of $\sigma_0$ ($\alpha=0.5$).}
		\label{fig:sigma01}
	\end{minipage}\hfill
	\begin{minipage}{0.5\textwidth}
		\centering
		\includegraphics[totalheight=6cm]{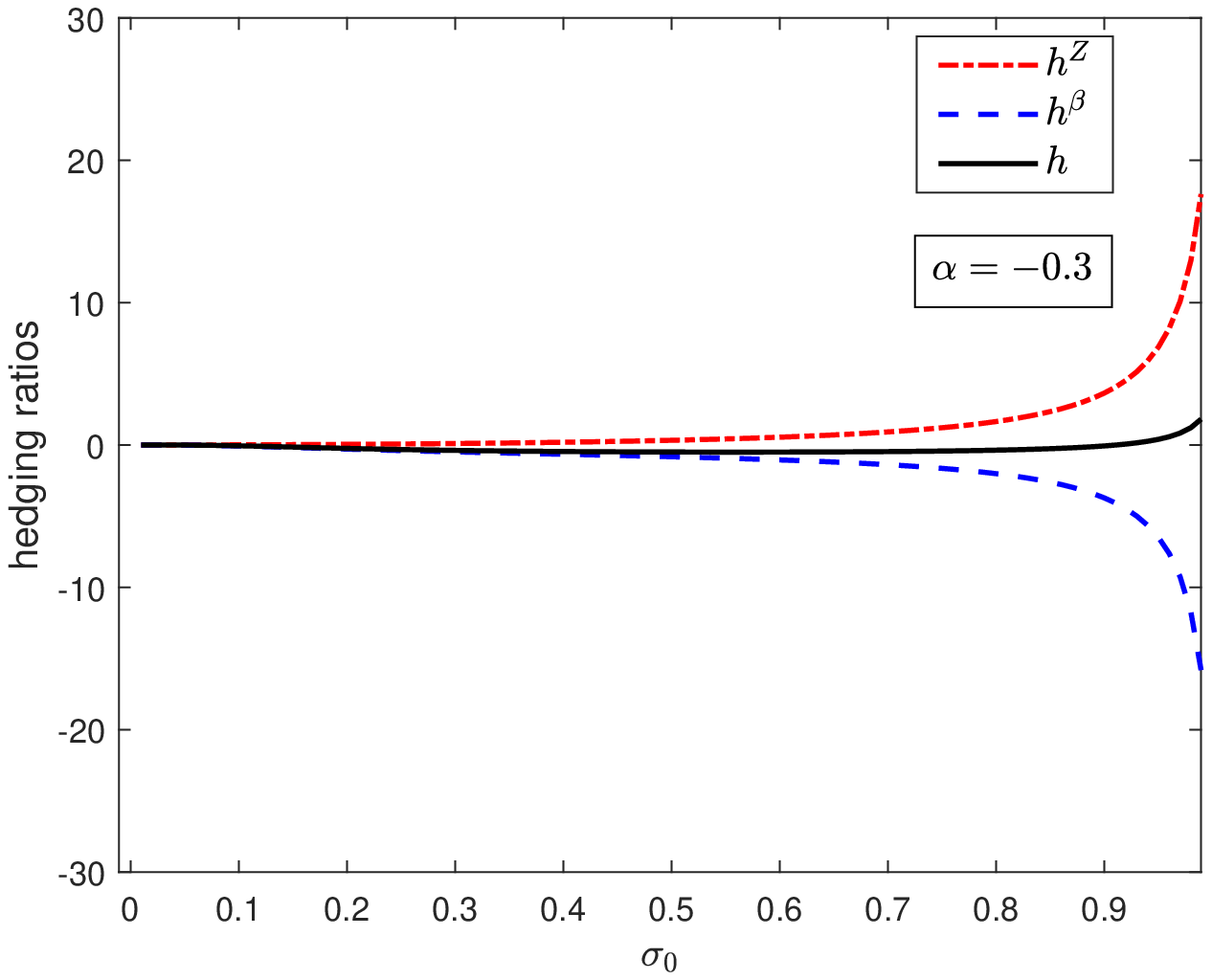}
		\caption{Effects of $\sigma_0$ ($\alpha=-0.3$).}
		\label{fig:sigma02}
	\end{minipage}\hfill
\end{figure}

\newpage
\begin{appendices}

\section{Posterior Distribution of $ Z$}\label{app:posterior}

Now we are going to investigate the posterior distribution of $ Z$. % and the dynamics of $\beta$.
%At initial time, the prior distribution function of $ Z$ is denoted by $F_0(\cdot)$. Suppose   that
%\begin{equation}\label{assup1}
%	\exists \eta>0, \quad \mathbf{E}[e^{\eta Z^{2}}]=\int_{\mathbf{R}} e^{\eta z^2} F_0(\rd z)<+\infty.
%\end{equation}
%Now we are going to derive the dynamics of $\beta$.
Following \cite{Bis2019}, let probability measure $\mathbb{Q}$ be defined by
\begin{align*}
\frac{\rd \mathbb{Q}}{\rd \mathbb{P}}\bigg|_{\mathcal{F}_T}\trieq\exp\left(-\frac{ Z-r}{\sigma}W_T-{1\over2}\left({ Z-r\over \sigma}\right)^2T\right),
\end{align*}
and  process $W^\mathbb{Q}=\{W^\mathbb{Q}_t:  t\in[0,T]\}$  by
\[
W^\mathbb{Q}_t\trieq W_t+\frac{ Z-r}{\sigma}t,\quad t\in[0,T].
\]
Then we have
\[
{\rd S_t\over S_t}=r\rd t+\sigma\rd W^\mathbb{Q}_t.
\]
Recalling that $Y_t=\log S_t$, we have
$$Y_t-Y_0=\left(r-{1\over2}\sigma^2\right)t+\sigma W^\mathbb{Q}_t.$$
The next proposition characterizes the posterior distribution of $ Z$, which extends some results of \cite{Bis2019}.
%To show the property of $ Z$ given information before time $t$, we present the following proposition.
\begin{proposition}\label{prop:mu}
	For any Borel function $f:\mathbf{R}\to\mathbf{R}$ such that $\mathbf{E}[|f( Z)|]<\infty$, we have
%\begin{equation}
%	\mathbf{E}[f( Z)\mid \mathcal{F}_t^S]=\psi(t,W^\mathbb{Q}_t),
%\end{equation}
%where 	
%\begin{equation}\label{eq:psi}
%\mathbf{E}[f( Z)\mid \mathcal{F}_t^S]\color{red}=\frac{\displaystyle\int_{\mathbf{R}} f(z) \exp\left(\frac{z-r}{\sigma^2}\left(Y_t-Y_0-(r-{1\over 2}\sigma^2)t\right)-\frac{ (z-r)^2}{2\sigma^2}t\right) F_0(\rd z)}{\displaystyle\int_{\mathbf{R}} \exp\left(\frac{z-r}{\sigma^2}\left(Y_t-Y_0-(r-{1\over 2}\sigma^2)t\right)-\frac{ (z-r)^2}{2\sigma^2}t\right) F_0(\rd z)}.
%\end{equation}
\begin{equation}\label{eq:psi}
	\mathbf{E}[f( Z)\mid \mathcal{F}_t^S]
	=\frac{\displaystyle\int_{\mathbf{R}} f(z) \exp\left(\frac{z}{\sigma^2}\left[Y_t-Y_0+{1\over 2}\sigma^2t\right]-\frac{z^2}{2\sigma^2}t\right) F_0(\rd z)}
	{\displaystyle\int_{\mathbf{R}} \exp\left(\frac{z}{\sigma^2}\left[Y_t-Y_0+{1\over 2}\sigma^2t\right]-\frac{z^2}{2\sigma^2}t\right) F_0(\rd z)}.
\end{equation}
\end{proposition}
\proof
	By condition \eqref{assup1} and Girsanov's theorem, $W^\mathbb{Q}$ is a standard Brownian motion w.r.t. filtration $\{\mathcal{F}^S_t\}$ under the probability $\mathbb{Q}$. Moreover,  by \citet[Proposition 1]{Bis2019}, under the probability $\mathbb{Q}$, $ Z$ is independent of $W^\mathbb{Q}_t$ for all $t\in[0,T]$. Then by Bayes' rule, we have, $\forall\, t\in[0,T]$,
	\begin{align}%\label{equ:beta}
\mathbf{E}[f( Z)\mid\mathcal{F}^S_t]
			&=\frac{\mathbf{E}^\mathbb{Q}\left[\left.f( Z) \frac{\rd \mathbb{P}}{\rd \mathbb{Q}}\right|\mathcal{F}^S_t\right]}
			{\mathbf{E}^\mathbb{Q}\left[\left.\frac{\rd \mathbb{P}}{\rd \mathbb{Q}}\right|\mathcal{F}^S_t\right]}\nonumber
			\\
			&=\frac{\mathbf{E}^\mathbb{Q}\left[\left.f( Z) \exp\left(\frac{ Z-r}{\sigma}W^\mathbb{Q}_T-\frac{ ( Z-r)^2}{2\sigma^2}T\right)\right|\mathcal{F}^S_t\right]}
			{\mathbf{E}^\mathbb{Q}\left[\left.\exp\left(\frac{ Z-r}{\sigma}W^\mathbb{Q}_T-\frac{ ( Z-r)^2}{2\sigma^2}T\right)\right|\mathcal{F}^S_t\right]}\nonumber
			\\
			&=\frac{\mathbf{E}^\mathbb{Q}\left[\left.f( Z) \exp\left(\frac{ Z-r}{\sigma}W^\mathbb{Q}_t-\frac{ ( Z-r)^2}{2\sigma^2}t\right)\right|\mathcal{F}^S_t\right]}
			{\mathbf{E}^\mathbb{Q}\left[\left.\exp\left(\frac{ Z-r}{\sigma}W^\mathbb{Q}_t-\frac{ ( Z-r)^2}{2\sigma^2}t\right)\right|\mathcal{F}^S_t\right]}\nonumber
			\\
			&=\frac{\displaystyle\int_{\mathbf{R}} f(z) \exp\left(\frac{z}{\sigma^2}\left[Y_t-Y_0+{1\over 2}\sigma^2t\right]-\frac{z^2}{2\sigma^2}t\right) F_0(\rd z)}{\displaystyle\int_{\mathbf{R}} \exp\left(\frac{z}{\sigma^2}\left[Y_t-Y_0+{1\over 2}\sigma^2t\right]-\frac{z^2}{2\sigma^2}t\right) F_0(\rd z)}. \tag*{\qed}
	\end{align}

%\begin{corollary}\label{cor:posterior}
%{\color{red} If $Z\sim N\left(\beta_0,\sigma_0^2\right)$, then $Z|\mathcal{F}^S_t\sim N\left(\bbeta_t,\zeta(t)\right)$, where
%$\zeta(t)=(\sigma^{-2}_0+t\sigma^{-2})^{-1}$ and
%$\bbeta_t=\zeta(t)\left\{\sigma^{-2}\left[rt+\sigma W^\mathbb{Q}_t\right]+\sigma_0^{-2} \beta_{0}\right\}$.
%%\begin{eqnarray*}
%%	\bbeta_t=\zeta(t)\left\{\sigma^{-2}\left[rt+\sigma W^\mathbb{Q}_t\right]+\sigma_0^{-2} \beta_{0}\right\}.
%%\end{eqnarray*}
%}
%\end{corollary}

\section{Proof of Theorem \ref{thm:y}}\label{app:thm}
	First, a combination of condition (b), the Markov property of $(X^{z,{\tilde \pi}^*},Y^z)$ w.r.t. $\{\F^W_t\}$, \eqref{equ:hjbg2}, and the independence between $Z$ and $\{\F^W_t\}$ leads to, $\forall z\in\range(Z)$, ${\tilde \pi}\in\tilde\Pi$ and $(t,x,y)\in[0,T)\times\Rbf^2$,
\[
\begin{split}{\tilde g}^z(t,x,y)
	%			&=\left(\mathbf{E}\left[\left.{\tilde g}^z\left(T,X^{z,{\tilde \pi}^*}_T,Y^{z}_T\right)\,\right|\, \F^W_t\right]\right)_{X^{z,{\tilde \pi}^*}_t=x,Y^z_t=y}\\
	%			&=\left(\mathbf{E}\left[U(X^{z,{\tilde \pi}^*}_T)\mid \F^W_t\right]\right)_{X^{z,{\tilde \pi}^*}_t=x,Y^z_t=y}
	&=\mathbf{E}\left[\left.{\tilde g}^z\left(T,X^{z,{\tilde \pi}^*}_T,Y^{z}_T\right)\,\right|\, X^{z,{\tilde \pi}^*}_t=x,Y^z_t=y\right]
	\\&=\mathbf{E}\left[U(X^{z,{\tilde \pi}^*}_T)\mid X^{z,{\tilde \pi}^*}_t=x,Y^z_t=y\right]
	\\&=\mathbf{E}\left[U(X^{z,{\tilde \pi}^*}_T)\mid Z=z, X^{z,{\tilde \pi}^*}_t=x,Y^z_t=y\right]
	\\&=\mathbf{E}\left[U(X^{ Z,{\tilde \pi}^*}_T)\mid Z=z,X^{Z,{\tilde \pi}^*}_t=x,Y^Z_t=y\right]
	\\&=\mathbf{E}\left[U(X^{{\tilde \pi}^*}_T)\mid Z=z,X^{{\tilde \pi}^*}_t=x,Y_t=y\right].
	%	\\&=\mathbf{E}\left[U(X^{z,{\tilde \pi}^*}_t)\mid Z=z, \F^W_t\right]\\&=\mathbf{E}\left[U(X^{z,{\tilde \pi}^*}_t)\mid Z=z, \F^S_t\right]\\&=\mathbf{E}\left[U(X^{ Z,{\tilde \pi}^*}(T))\mid Z=z, \F^S_t\right].
\end{split}
\]

Now we show that ${\tilde \pi}^*$ is an equilibrium strategy.

By \eqref{eq:gzpi:z}, $\forall z\in\range(Z)$, ${\tilde \pi}\in\tilde\Pi$, $(t,x,y)\in[0,T)\times\Rbf^2$ and $h\in(0,T-t)$,
\begin{equation}\label{eq:g:pi:th}
	\begin{split}
		{\tilde g}^{z, {\tilde \pi}_{t,h}}(t,x,y)&=
		\mathbf{E}\left [U(X^{z, {\tilde \pi}_{t,h}}_T)\mid X^{z,{\tilde \pi}_{t,h}}_t=x, Y^z_t=y\right]\\
		&=\mathbf{E} \left[\mathbf{E}\left[U(X^{z, {\tilde \pi}_{t,h}}_T)\mid X^{z,{\tilde \pi}_{t,h}}_{t+h}, Y^z_{t+h}\right]\mid X^{z,{\tilde \pi}_{t,h}}_t=x, Y^z_t=y\right]
		%		\\
		%		&=\mathbf{E}\left [\left.\mathbf{E}\left[\left. U(X^{t+h,X^{z,{\tilde \pi}}_{t+h},{\tilde \pi}^*}_T)\,\right|\,
		%		\mathcal{F}_{t+h}^W\right]
		%		\,\right|\,\mathcal{F}_t^W\right]
		\\
		&=\mathbf{E} \left[
		{\tilde g}^z(t+h,X^{z,{\tilde \pi}_{t,h}}_{t+h},Y^{z}_{t+h})
		\mid X^{z,{\tilde \pi}_{t,h}}_t=x, Y^z_t=y\right]\\
		&=\mathbf{E} \left[
		{\tilde g}^z(t+h,X^{z,{\tilde \pi}}_{t+h},Y^{z}_{t+h})
		\mid X^{z,{\tilde \pi}}_t=x, Y^z_t=y\right],
	\end{split}
\end{equation}
where the second equality is from the Markov property of $(X^{z,{\tilde \pi}_{t,h}},Y^z)$ w.r.t. $\{\F^W_t\}$.

Let  ${\tilde \pi}\in\tilde\Pi$ and $(t,x,y)\in[0,T)\times\Rbf^2$.
By Assumption \ref{assu0}, It\^{o}'s formula and the Markov property of $(X^{z,{\tilde \pi}_{t,h}},Y^z)$ w.r.t. $\{\F^W_t\}$,
there exists $\tilde t\in(t,T)$ such that, $\forall z\in\range(Z)$ and $h\in(0,\tilde{t}-t)$,
\begin{align*}
	&{\tilde g}^{z, {\tilde \pi}_{t,h}}(t,x,y)-{\tilde g}^z(t,x,y)\\
	=&\mathbf{E}\left[\left.{\tilde g}^z\left(t+h,X^{z,{\tilde \pi}}_{t+h},Y^{z}_{t+h}\right)\,\right |\, X^{z,{\tilde \pi}_{t,h}}_t=x, Y^z_t=y\right]-{\tilde g}^z(t,x,y)\\
	%=&\mathbf{E}\left[\left.{\tilde g}^z\left(t+h,X^{z,{\tilde \pi}}_{t+h},Y^{z}_{t+h}\right)\,\right |\,\F^W_t, X^{z,{\tilde \pi}_{t,h}}_t=x, Y^z_t=y\right]-{\tilde g}^z(t,x,y)\\
	%\le& \mathbf{E}\left[\left.\int_t^{t+h}\mathcal{\tilde A}^{z,{\tilde \pi}} {\tilde g}^z(u,X^{z,{\tilde \pi}}_u,Y^{z}_u)\mathrm{d}u\,\right|\,\F^W_t, X^{z,{\tilde \pi}_{t,h}}_t=x, Y^z_t=y\right]\\
	\le & \mathbf{E}\left[\left.\int_t^{t+h}\mathcal{\tilde A}^{z,{\tilde \pi}} {\tilde g}^z(u,X^{z,{\tilde \pi}}_u,Y^{z}_u)\mathrm{d}u\,\right|\, X^{z,{\tilde \pi}_{t,h}}_t=x, Y^z_t=y\right].
\end{align*}
%and hence
%	\begin{equation*}
%		\begin{split}
%			{\tilde g}^{z, {\tilde \pi}_{t,h}}(t,x,y)-{\tilde g}^z(t,x,y)
%			%&=\mathbf{E}\left[\left.\int_t^{t+h}\mathcal{\tilde A}^{z,{\tilde \pi}} {\tilde g}^z(u,X^{z,{\tilde \pi}}_u,Y^{z}_u)\mathrm{d}u\,\right|\,\F^W_t\right]
%			\le\mathbf{E}\left[\left.\int_t^{t+h}\mathcal{\tilde A}^{z,{\tilde \pi}} {\tilde g}^z(u,X^{z,{\tilde \pi}}_u,Y^{z}_u)\mathrm{d}u\,\right|\,X^{z,{\tilde \pi}}_t=x,Y^{z}_t=y\right].
%		\end{split}
%	\end{equation*}

Let ${\tilde \pi}\in\tilde\Pi$ and $(t,x,y)\in[0,T)\times\Rbf^2$. By Assumption \ref{assu1}, we have, $\forall z\in\range(Z)$,
\begin{equation}\label{eq:lim:hth:gz}
	\mathop{\lim}\limits_{h\rightarrow 0^+}\mathbf{E}\left[\left.\frac{1}{h}\int_t^{t+h}\mathcal{\tilde A}^{z,{\tilde \pi}} {\tilde g}^z(u,X^{z,{\tilde \pi}}_u,Y^{z}_u)\mathrm{d}u\,\right|\,X^{z,{\tilde \pi}}_t=x,Y^{z}_t=y\right]
	=\mathcal{\tilde A}^{z,{\tilde \pi}}{\tilde g}^z(t,x,y).
\end{equation}
Moreover, $\forall z\in\range(Z)$ and $h\in(0,T-t)$,
\begin{align*}
	&\mathbf{E}\left[\left.\frac{1}{h}\int_t^{t+h}\mathcal{\tilde A}^{z,{\tilde \pi}} {\tilde g}^z(u,X^{z,{\tilde \pi}}_u,Y^{z}_u)\mathrm{d}u\,\right|\,X^{z,{\tilde \pi}}_t=x,Y^{z}_t=y\right]\\
	=&\mathbf{E}\left[\left.\frac{1}{h}\int_t^{t+h}\mathcal{\tilde A}^{Z,{\tilde \pi}} {\tilde g}^Z(u,X^{z,{\tilde \pi}}_u,Y^{z}_u)\mathrm{d}u\,\right|\,Z=z, X^{z,{\tilde \pi}}_t=x,Y^{z}_t=y\right]
\end{align*}
and hence
\begin{align*}
	&\left(\mathbf{E}\left[\left.\frac{1}{h}\int_t^{t+h}\mathcal{\tilde A}^{z,{\tilde \pi}} {\tilde g}^z(u,X^{z,{\tilde \pi}}_u,Y^{z}_u)\mathrm{d}u\,\right|\,X^{z,{\tilde \pi}}_t=x,Y^{z}_t=y\right]\right)_{z=Z}\\
	=&\mathbf{E}\left[\left.\frac{1}{h}\int_t^{t+h}\mathcal{\tilde A}^{Z,{\tilde \pi}} {\tilde g}^Z(u,X^{Z,{\tilde \pi}}_u,Y^{Z}_u)\mathrm{d}u\,\right|\,Z,X^{Z,{\tilde \pi}}_t=x,Y^{Z}_t=y\right].
\end{align*}
Therefore, by \eqref{eq:lim:hth:gz},
\[
\mathop{\lim}\limits_{h\rightarrow 0^+}\mathbf{E}\left[\left.\frac{1}{h}\int_t^{t+h}\mathcal{\tilde A}^{Z,{\tilde \pi}} {\tilde g}^Z(u,X^{Z,{\tilde \pi}}_u,Y^{Z}_u)\mathrm{d}u\,\right|\,Z, X^{Z,{\tilde \pi}}_t=x,Y^{Z}_t=y\right]
=\mathcal{\tilde A}^{Z,{\tilde \pi}}{\tilde g}^Z(t,x,y).
\]

\paragraph{Case 1: $\phi$ is concave.} In this case, $\forall h\in(0,\tilde t-t)$, 	\begin{equation*}
	\begin{split}
		&\frac{\phi\left({\tilde g}^{z, {\tilde \pi}_{t,h}}(t,x,y)\right)-\phi\left({\tilde g}^z(t,x,y)\right)}{h}
		\\ \leq&\frac{1}{h}\phi^\prime\left({\tilde g}^z(t,x,y)\right)[{\tilde g}^{z, {\tilde \pi}_{t,h}}(t,x,y)-{\tilde g}^z(t,x,y)]
		\\ \le &\phi^\prime\left({\tilde g}^z(t,x,y)\right)\mathbf{E}\left[\left.\frac{1}{h}\int_t^{t+h}\mathcal{\tilde A}^{z,{\tilde \pi}} {\tilde g}^z(u,X^{z,{\tilde \pi}}_u,Y^{z}_u)\mathrm{d}u \,\right|\,  X^{z,{\tilde \pi}}_t=x,Y^{z}_t=y\right].
	\end{split}
\end{equation*}	
Then %{\color{red} ($\tilde t$ should be uniform w.r.t. $z$ so that the above inequality holds uniformly w.r.t. $z$. Therefore, Assumption \ref{assu0} is updated)}
\begin{equation*}
	\begin{split}
		& \mathop{\limsup}\limits_{h\rightarrow 0^+} \frac{J(t,x,y, {\tilde \pi}_{t,h})-J(t,x,y,{\tilde \pi}^*)}{h}
		\\\leq&\mathop{\limsup}\limits_{h\rightarrow 0^+}\mathbf{E}\left[\left. \phi^\prime\left({\tilde g}^Z(t,x,y)\right)\mathbf{E}\left[\left.\frac{1}{h}\int_t^{t+h}\mathcal{\tilde A}^{Z,{\tilde \pi}} {\tilde g}^Z(u,X^{Z,{\tilde \pi}}_u,Y^{Z}_u)\mathrm{d}u\right| Z, \genfrac{}{}{0pt}{}{X^{Z,{\tilde \pi}}_t=x}{Y^{Z}_t=y}\right]\right|Y_t=y\right]
		%		\\=&-\mathop{\limsup}\limits_{h\rightarrow 0^+}\mathbf{E}\left[\left. \phi^\prime\left({\tilde g}^Z(t,x,y)\right)\mathbf{E}\left[\left.\frac{1}{h}\int_t^{t+h}\mathcal{\tilde A}^{Z,{\tilde \pi}} {\tilde g}^Z(u,X^{z,{\tilde \pi}}_u,Y^{z}_u)\mathrm{d}u\right|  Z,X^{z,{\tilde \pi}}_t=x,Y^{z}_t=y\right]\right|Y_t=y\right]
		%		\\=&-\mathbf{E}\left[\left.\mathop{\limsup}\limits_{h\rightarrow 0^+} \phi^\prime\left({\tilde g}^Z(t,x,y)\right)\mathbf{E}\left[\left.\frac{1}{h}\int_t^{t+h}\mathcal{\tilde A}^{Z,{\tilde \pi}} {\tilde g}^Z(u,X^{z,{\tilde \pi}}_u,Y^{z}_u)\mathrm{d}u\right|  Z,X^{z,{\tilde \pi}}_t=x,Y^{z}_t=y\right]\right|Y_t=y\right]
		\\=&\mathbf{E}\left[ \phi^\prime\left({\tilde g}^ Z(t,x,y)\right)
		\mathcal{\tilde A}^{ Z,{\tilde \pi}}{\tilde g}^ Z(t,x,y)\mid Y_t=y\right]
		\\\leq&0,
	\end{split}
\end{equation*}	
where the equality is from Assumption \ref{assu2} and the dominated convergence theorem.  Therefore, ${\tilde \pi}^*$ is an equilibrium strategy.

\paragraph{Case 2: $\phi$ is convex.} In this case, $\forall h\in(0,\tilde t-t)$,
\begin{equation*}
	\begin{split}
		&\frac{\phi\left({\tilde g}^{z, {\tilde \pi}_{t,h}}(t,x,y)\right)-\phi\left({\tilde g}^z(t,x,y)\right)}{h}
		\\ \leq&\frac{1}{h}\phi^\prime\left({\tilde g}^{z, {\tilde \pi}_{t,h}}(t,x,y)\right)[{\tilde g}^{z, {\tilde \pi}_{t,h}}(t,x,y)-{\tilde g}^z(t,x,y)]
		\\ \le&\phi^\prime\left({\tilde g}^{z, {\tilde \pi}_{t,h}}(t,x,y)\right)\mathbf{E}\left[\left.\frac{1}{h}\int_t^{t+h}\mathcal{\tilde A}^{z,{\tilde \pi}} {\tilde g}^z(u,X^{z,{\tilde \pi}}_u,Y^{z}_u)\mathrm{d}u\,\right|\,  X^{z,{\tilde \pi}}_t=x,Y^{z}_t=y\right].
	\end{split}
\end{equation*}	
Moreover, by Assumption \ref{assu3} and \eqref{eq:g:pi:th},
$\mathop{\lim}\limits_{h\rightarrow 0^+}{\tilde g}^{z, {\tilde \pi}_{t,h}}(t,x,y) ={\tilde g}^z(t,x,y)$ and hence, by the continuity of $\phi^\prime$,  $$\mathop{\lim}\limits_{h\rightarrow 0^+}\phi^\prime\left({\tilde g}^{z, {\tilde \pi}_{t,h}}(t,x,y)\right) =\phi^\prime\left({\tilde g}^z(t,x,y)\right).$$
%Similarly, we have
%\begin{equation*}
%	\begin{split}
%		& \mathop{\liminf}\limits_{h\rightarrow 0^+} \frac{J(t,x,y,{\tilde \pi}^*)-J(t,x,y, {\tilde \pi}_{t,h})}{h}
%		\\ \geq&\mathop{\liminf}\limits_{h\rightarrow 0^+}\mathbf{E}\left[ -\phi^\prime\left({\tilde g}^{Z, {\tilde \pi}_{t,h}}(t,x,y)\right)\mathbf{E}\left[\frac{1}{h}\int_t^{t+h}\mathcal{\tilde A}^{Z,{\tilde \pi}} {\tilde g}^Z(u,X^{z,{\tilde \pi}}_u,Y^{z}_u)\mathrm{d}u|  Z,\F_t^S\right]|\mathcal{F}_t^S\right]
%		\\ =&-\mathop{\limsup}\limits_{h\rightarrow 0^+}\mathbf{E}\left[ \phi^\prime\left({\tilde g}^{Z, {\tilde \pi}_{t,h}}(t,x,y)\right)\mathbf{E}\left[\frac{1}{h}\int_t^{t+h}\mathcal{\tilde A}^{Z,{\tilde \pi}} {\tilde g}^Z(u,X^{z,{\tilde \pi}}_u,Y^{z}_u)\mathrm{d}u|  Z,\F_t^S\right]|\mathcal{F}_t^S\right]
%		\\ =&-\mathbf{E}\left[ \mathop{\limsup}\limits_{h\rightarrow 0^+}\phi^\prime\left({\tilde g}^{Z, {\tilde \pi}_{t,h}}(t,x,y)\right)\mathbf{E}\left[\frac{1}{h}\int_t^{t+h}\mathcal{\tilde A}^{Z,{\tilde \pi}} {\tilde g}^Z(u,X^{z,{\tilde \pi}}_u,Y^{z}_u)\mathrm{d}u|  Z,\F_t^S\right]|\mathcal{F}_t^S\right]
%		\\=&-\mathbf{E}\left[ \phi^\prime\left({\tilde g}^ Z(t,x,y)\right)
%		\mathcal{\tilde A}^{ Z,{\tilde \pi}}{\tilde g}^ Z(t,x,y)|\mathcal{F}_t^S\right]
%		\\ \geq&0.
%	\end{split}
%\end{equation*}	
Then similarly to Case 1, we can show that
\begin{eqnarray*}
	\mathop {\limsup }\limits_{h\rightarrow 0^+} \frac{J(t,x,y, {\tilde \pi}_{t,h})-J(t,x,y,{\tilde \pi}^*)}{h}\leq 0,
\end{eqnarray*}
and hence, ${\tilde \pi}^*$ is an equilibrium strategy.
\qed

\section{Some Results on $\bbeta^z$}\label{app:bbeta}

%We present several useful results about process $\bbeta^z$ in this appendix.

From \eqref{equ:beta:mu}, we obtain
\begin{equation}\label{equ:betaclosed}
	\bbeta^z_s=\zeta(s)\left[\zeta(t)^{-1}\bbeta^z_t+ z\sigma^{-2}(s-t)+\sigma^{-1}(W_s-W_t)\right],\quad s\ge t.
\end{equation}
Then
 \begin{equation}\label{equ:betain}
\begin{split}
	{\bbeta^z_s}^2&=\zeta^{2}(s)\left[\zeta(t)^{-1}\bbeta^z_t+ z\sigma^{-2}(s-t)+\sigma^{-1}(W_s-W_t)\right]^2
	\\ &\leq {3\sigma_0^4} \left[{\zeta(t)}^{-2}{\bbeta^z_t}^2+ z^2\sigma^{-4}(s-t)^2+\sigma^{-2}(W_s-W_t)^2\right].
\end{split}
\end{equation}	

\begin{lemma}\label{lemma1}
	For any $C>0$ and $\epsilon>0$, there exists some $h>0$  such that, $\forall t\in[0,T)$, $\beta\in\Rbf$ and $z\in\range(Z)$,
	%\begin{equation}
	%	\mathbf{E}\left[\exp\left\{C{\bbeta^z_s}^2\right\}\big|\mathcal{F}_t^W\right]\leq C_1 \exp(\epsilon z^2),
	%\end{equation}
	%i.e.,
	\begin{equation}\label{inequ:lemma1}
		\left\{\mathbf{E}\left[\left.\exp\left\{C\int_t^{(t+h)\land T}{\bbeta^z_u}^2
		\rd u\right\}\,\right|\, \bbeta^{z}_t=\beta \right]\right\}\leq {4\over 3} e^{\epsilon (z^2+\beta^2)}
	\end{equation}
	and
	\begin{equation}\label{inequ:lemma2}
		\sup_{s\in[t,(t+h)\land T)}\left\{\mathbf{E}\left[\left.\exp\left\{C{\bbeta^z_s}^2\right\}\,\right|\, \bbeta^{z}_t=\beta\right]\right\}\leq \sqrt{2} e^{\epsilon z^2+3C\sigma_0^4{\zeta(T)}^{-2}\beta^2}.
	\end{equation}
%	where $C_{1,\epsilon}$ is some constant which is independent of $z$.	
\end{lemma}
\proof
First, we prove \eqref{inequ:lemma1}. By \eqref{equ:betain}, we have
	\begin{equation}\label{equ:lemma1}
	\begin{split}
		&\mathbf{E}\left[\left.\exp\left\{C\int_t^s{\bbeta^z_u}^2\rd u\right\}\,\right|\, \bbeta^z_t=\beta\right]\\
		\leq&\mathbf{E}\left[\exp\left\{3\sigma_0^4C\int_t^s\left[{\zeta(t)}^{-2}\beta^2+ z^2\sigma^{-4}(u-t)^2+\sigma^{-2}(W_u-W_t)^2\right]\rd u\right\}\right]
%		\\=&\exp\left\{3\sigma_0^2C\int_t^s\left[{\zeta(u)}^{-2}\beta^2+ z^2\sigma^{-4}(u-t)^2\right]\rd u\right\}\mathbf{E}\left[\exp\left\{3\sigma_0^2C \sigma^{-2}\int_t^s(W_u-W_t)^2\rd u\right\}\right]
		\\=&\exp\left\{3\sigma_0^4C{\zeta(t)}^{-2}(s-t)\beta^2\right\}
		\exp\left\{\sigma_0^4C\sigma^{-4} (s-t)^3z^2\right\}\\
		&\quad\times\mathbf{E}\left[\exp\left\{3\sigma_0^4C\sigma^{-2}\int_t^s(W_u-W_t)^2\rd u\right\}\right]
	\end{split}
\end{equation} Then, by Taylor's expansion and Jensen's inequality,
\[
\begin{split}
	&\mathbf{E}\left[\exp\left\{3\sigma_0^4C \sigma^{-2}\int_t^s(W_u-W_t)^2\rd u\right\}\right]\\
	=&
	1+\sum_{n=1}^\infty\frac{[3\sigma_0^4C \sigma^{-2}]^n}{n!}\mathbf{E}\left[
	\left(\int_t^s(W_u-W_t)^2\rd u\right)^n
	\right]\\
	%\le &
%	1+\sum_{n=1}^\infty\frac{[3\sigma_0^4C \sigma^{-2}]^n}{n!}\mathbf{E}\left[
%	\left(\left(\int_t^s(W_u-W_t)^{2n}\rd u\right)^{1\over n}\left(\int_t^s 1 \rd u\right)^{1-{1\over n}}\right)^n
%	\right]\\
	\le& 1+\sum_{n=1}^\infty\frac{[3\sigma_0^4C \sigma^{-2}]^n}{n!}\mathbf{E}\left[(s-t)^{n-1}\int_t^s(W_u-W_t)^{2n}\rd u\right]\\
	=&1+\sum_{n=1}^\infty\frac{[3\sigma_0^4C \sigma^{-2}]^n}{n!}(s-t)^{n-1}\int_t^s
	(u-t)^n {(2n)!\over 2^n n!}\rd u\\
	=&1+\sum_{n=1}^\infty\frac{[3\sigma_0^4C \sigma^{-2}]^n}{n!}(s-t)^{n-1}{(2n)!\over 2^n n!}{(s-t)^{n+1}\over n+1}\\
	\le &1+\sum_{n=1}^\infty[6\sigma_0^4C \sigma^{-2}(s-t)^2]^n\\
	=&{1\over 1-6\sigma_0^4C \sigma^{-2}(s-t)^2}
\end{split}
\]
if $6\sigma_0^4C \sigma^{-2}(s-t)^2<1$, i.e.,
$$s-t<\left({1\over 6}\sigma_0^{-4}\sigma^{2}C^{-1}\right)^{1\over 2}.$$
Moreover,  $\sigma_0^4C \sigma^{-4} (s-t)^3<\epsilon$ and
$3\sigma_0^4C{\zeta(T)}^{-2}(s-t)<\epsilon$
if
$$s-t<\min\left\{\left(\epsilon\sigma_0^{-4}\sigma^{4}C^{-1}\right)^{1\over 3},{1\over3}\epsilon\sigma_0^{-4}C^{-1}\zeta^2(T)\right\}.$$
Let
$$h={1\over2}\min\left\{\left({1\over 6}\sigma_0^{-4}\sigma^{2}C^{-1}\right)^{1\over 2},\left(\epsilon\sigma_0^{-4}\sigma^{4}C^{-1}\right)^{1\over 3},{1\over3}\epsilon\sigma_0^{-4}C^{-1}\zeta^2(T)
\right\}.$$
Then by \eqref{equ:lemma1},
	\begin{equation*}
	\begin{split}
		&\mathbf{E}\left[\left.\exp\left\{C\int_t^{(t+h)\land T}{\bbeta^z_u}^2\rd u\right\}\,\right|\, \bbeta^z_t=\beta\right]\\ \leq&\exp\left\{3\sigma_0^4C{\zeta(t)}^{-2}h\beta^2\right\} e^{\epsilon z^2} {1\over 1-6\sigma_0^4C \sigma^{-2}h^2}\\
	\leq&{1\over 1-6\sigma_0^4C \sigma^{-2}h^2}e^{\epsilon (z^2+\beta^2)}
		\leq{4\over 3}e^{\epsilon (z^2+\beta^2)}.
	\end{split}
\end{equation*}
Thus,  \eqref{inequ:lemma1} is proved.

Next, we show \eqref{inequ:lemma2}. By \eqref{equ:betain} and Lemma \ref{lma:exp:normal} below, we have
\begin{equation}\label{equ:cbeta}
\begin{split}
	&\mathbf{E}\left[\left.\exp\left\{C{\bbeta^z_s}^2\right\}\,\right|\,\bbeta^{z}_t=\beta\right]
	\\&\leq \mathbf{E}\left[\left.\exp\left\{3C\sigma_0^4[{\zeta(t)}^{-2}{\bbeta^z_t}^2+ z^2\sigma^{-4}(s-t)^2+\sigma^{-2}(W_s-W_t)^2]\right\}\,\right|\,\bbeta^{z}_t=\beta\right]
	\\&=\exp\left\{3C\sigma_0^4[{\zeta(t)}^{-2}\beta^2+ z^2\sigma^{-4}(s-t)^2]\right\}
	\mathbf{E}[e^{3C\sigma_0^4\sigma^{-2}(W_s-W_t)^2}]
	\\&=\exp\left\{3C\sigma_0^4{\zeta(t)}^{-2}\beta^2\right\}\exp\left\{3C\sigma_0^4 \sigma^{-4}(s-t)^2z^2\right\}(1-6C\sigma_0^4\sigma^{-2}(s-t))^{-0.5}
\end{split}
\end{equation}
if $6C\sigma_0^4\sigma^{-2}(s-t)<1$, i.e., $$s-t<{1\over 6}\sigma_0^{-4}\sigma^{2}C^{-1}.$$
Moreover,  $3C\sigma_0^4 \sigma^{-4}(s-t)^2<\epsilon$ if
$$s-t<\left({\epsilon\over 3}\sigma_0^{-4}\sigma^{4}C^{-1}\right)^{1\over 2}.$$ Let
$$h={1\over 2}\min\left\{{1\over 6}\sigma_0^{-4}\sigma^{2}C^{-1},\left({\epsilon\over 3}\sigma_0^{-4}\sigma^{4}C^{-1}\right)^{1\over 2}
\right\}.
$$
Then by \eqref{equ:cbeta}, for $s\in[t,(t+h)\land T)$,
\[
\begin{split}
	&\mathbf{E}\left[\left.\exp\left\{C{\bbeta^z_s}^2\right\}\,\right|\,\bbeta^{z}_t=\beta\right]
	\\&\leq(1-6C\sigma_0^4\sigma^{-2}h)^{-0.5}\exp\left\{3C\sigma_0^4{\zeta(t)}^{-2}\beta^2\right\}e^{\epsilon z^2}
	\\&\leq\sqrt{2}\exp\big\{3C\sigma_0^4{\zeta(T)}^{-2}\beta^2\big\}e^{\epsilon z^2}.
\end{split}
\]
Therefore, \eqref{inequ:lemma2} holds.\qed

\begin{lemma}\label{lma:exp:normal}
Assume that $\xi\sim N(\mu,\sigma^2)$.   If  $2a\sigma^2<1$, then $$\Ebf[e^{a\xi^2+b\xi}]={1 \over \sqrt{1-2a\sigma^2}}e^{{2a\mu^2+2\mu b+b^2\sigma^2\over 2(1-2a\sigma^2)}}
\text{ and }\ \Ebf[\xi e^{a\xi^2+b\xi}]={\mu+b\sigma^2 \over (1-2a\sigma^2)^{3\over 2}}e^{{2a\mu^2+2\mu b+b^2\sigma^2\over 2(1-2a\sigma^2)}}.$$
If  $2a\sigma^2\ge1$, then $\Ebf[e^{a\xi^2+b\xi}]=\infty$.
\end{lemma}

Now we present a result on an exponential martingale generated by $\bbeta^z$.
\begin{lemma}\label{lemma2}
	If $f_1(t)$ and $f_2(t)$ are bounded Borel functions on $[0,T]$, then,  $\forall z\in\range(Z)$,
	\[
	\left\{G_t=\exp\left(-\frac{1}{2}
	\int_0^t\left[f_1(s)\bbeta_s^z+f_2(s)\right]^2\rd s+\int_0^t\left[f_1(s)\bbeta_s^z+f_2(s)\right]\rd W_s\right)\right\}_{t\in[0,T]}
	\]
	is a martingale w.r.t. filtration $\{\F_t^W\}$.
\end{lemma}
\proof  Let $z\in\range(Z)$ be fixed.
There exists a constant $M>0$ such that $f_1^2(t)\le M$ for all $t\in[0,T]$.
Then
\[
\left[f_1(t)\bbeta_t^z+f_2(t)\right]^2\le 2f_1^2(t){\bbeta^z_t}^2+2f_2^2(t)
\leq2M{\bbeta^z_t}^2+2f_2^2(t),\quad t\in[0,T].
\]
Therefore,
\begin{align*}
\mathbf{E}\left[\exp\left(\frac{1}{2}
\int_t^s\left[f_1(u)\bbeta_u^z+f_2(u)\right]^2\rd u\right)\right]
\leq \exp\left(
\int_t^sf_2^2(u)\rd u\right)\mathbf{E}\left[\exp\left(M\int_t^s{\bbeta^z_u}^2\rd u\right)\right].
\end{align*}
Let $\epsilon={\sigma^2\over 3\sigma_0^4T}$. By Lemma \ref{lemma1}, there exists some $h>0$ such that,  for any $t\in[0,T)$, $\beta\in\Rbf$ and $z\in\range(Z)$,
\[
\mathbf{E}\left[\left.\exp\left(M
\int_t^{(t+h)\land T}{\bbeta^z_u}^2\rd u\right)\,\right|\, \bbeta^{z}_t=\beta\right]\leq {4\over3} e^{\epsilon z^2+\epsilon \beta^2}.
\]
By \eqref{equ:betaclosed},
$$\bbeta^z_t\sim N\left(\zeta(t)(\sigma_0^{-2}\beta_0+z\sigma^{-2}t),\zeta^2(t)\sigma^{-2}t\right).$$
Then $2\epsilon\var(\bbeta^z_t)<1$ and hence
\begin{align*}
\mathbf{E}\left[\exp\left(M
\int_t^{(t+h)\land T}{\bbeta^z_u}^2\rd u\right)\right]\leq{4\over3}e^{\epsilon z^2}\Ebf\left[  e^{\epsilon (\bbeta^z_t)^2}\right]<\infty.
\end{align*}
Therefore,
\[
\mathbf{E}\left[\exp\left(\frac{1}{2}
\int_t^{(t+h)\land T}\left[f_1(u)\bbeta_u^z+f_2(u)\right]^2\rd u\right)\right]
<\infty,
\]
i.e.,  Novikov's condition is satisfied.
Then, for any $t\in[0,T)$,
\[
\left\{\exp\left(-\frac{1}{2}
\int_t^s\left[f_1(u)\bbeta_u^z+f_2(u)\right]^2\rd s+\int_t^s\left[f_1(u)\bbeta_u^z+f_2(u)\right]\rd W_u\right)\right\}_{s\in[t,(t+h)\land T]}
\]
is a martingale w.r.t. filtration $\{\F_t^W\}$.

For a partition
\[
0=t_0<t_1<\dots<t_{n-1}<t_n=T
\]
of $[0,T]$ such that $\sup\limits_k\{t_{k+1}-t_k\}<h$. Denote
\[
H_{k,s}=\exp\left(-\frac{1}{2}
\int_{t_k}^s\left[f_1(u)\bbeta_u^z+f_2(u)\right]^2\rd u+\int_{t_k}^s\left[f_1(u)\bbeta_u^z+f_2(u)\right]\rd W_u\right),\  s\in[t_k,t_{k+1}].
\]
Then $\{H_{k,s}\}_{s\in[t_k,t_{k+1}]}$ is a martingale and $\mathbf{E}\left[H_{k,t_{k+1}}\mid \F_{t_k}^W\right]=1$. Therefore,
\begin{equation*}
\begin{split}
	&\mathbf{E}\left[G_T\right]=\mathbf{E}\left[\exp\left(-\frac{1}{2}
	\int_0^T\left[f_1(s)\bbeta_s^z+f_2(s)\right]^2\rd s+\int_0^T\left[f_1(s)\bbeta_s^z+f_2(s)\right]\rd W_s\right)\right]
\\&=\mathbf{E}\left[\prod_{k=0}^{n-1}H_{k,t_{k+1}}\right]
=\mathbf{E}\left[\prod_{k=0}^{n-2}H_{k,t_{k+1}}\mathbf{E}\left[H_{n-1,t_{n}}\mid \F_{t_{n-1}}^W\right]\right]
=\mathbf{E}\left[\prod_{k=0}^{n-2}H_{k,t_{k+1}}\right]=\dots=1.
\end{split}
\end{equation*}
Because $\{G_t\}_{t\in[0,T]}$ is a supermartingale with expectation 1, $\{G_t\}_{t\in[0,T]}$ is a martingale w.r.t. filtration $\{\F_t^W\}$.\qed

\section{Proof of Proposition \ref{prop:admissible}}\label{app:admissible}

For notational simplicity, let
$$\bpi^z_s=\pi(s,X^{z,\pi}_s,\bbeta^z_s),\quad s\in[0,T], z\in\range(Z).$$

Before proving the proposition, we present a lemma.
\begin{lemma}\label{lemma0}
Let $\pi$ satisfies the assumption of Proposition \ref{prop:admissible}. For every  %Let $(t,x,\beta)\in[0,T)\times\Rbf^2$.
$p\geq 2$ and $\epsilon>0$, there exist $h>0$ and $C>0$ such that, $\forall (t,x,\beta)\in[0,T)\times\Rbf^2$ and $z\in\range(Z)$,
	\begin{equation}\label{inequ:lemma0}
		\sup_{s\in[t,(t+h)\land T)}\mathbf{E}\left[\left.|X^{z,\pi}_s|^p\,\right|\,X^{z,\pi}_t=x, \bbeta^{z}_t=\beta\right]
		\leq C\left(|x|^p+ e^{\epsilon z^2+3\sigma_0^4{\zeta(T)}^{-2}\beta^2}\right).
	\end{equation}	
\end{lemma}

\proof From \eqref{equ:x:mu}, we obtain
\begin{equation}\label{equ:xintegral}
	X^{z,\pi}_s
	=e^{r(s-t)}X^{z,\pi}_t+\int_t^se^{r(s-u)}\bpi^z_u(z-r)\rd u+
	\int_t^se^{r(s-u)}\bpi^z_u\sigma\rd W_u.
\end{equation}
Let  $p\ge 2$ be fixed. Then by Jensen's inequality,
\begin{equation*}
	\begin{split}
		\left|{X^{z,\pi}_s\over 3}\right|^p
		%=&|e^{r(s-t)}X^{z,\pi}_t+\int_t^se^{r(s-u)}\bpi^z_u(z-r)\rd u+
		%\int_t^se^{r(s-u)}\bpi^z_u\sigma\rd W_u|
		 \le& {1\over 3}\left(e^{pr(s-t)}|X^{z,\pi}_t|^p+\left|\int_t^se^{r(s-u)}\bpi^z_u(z-r)\rd u\right|^p+\left|\int_t^se^{r(s-u)}\bpi^z_u\sigma\rd W_u\right|^p\right)\\
		\le&{1\over3}\left(e^{pr(s-t)}|X^{z,\pi}_t|^p+ (s-t)^{p-1}\int_t^se^{pr(s-u)}|\bpi^z_u(z-r)|^p\rd u+ e^{prs}\left|\int_t^se^{-ru}\bpi^z_u\sigma\rd W_u\right|^p\right)\\
		\le&{1\over3}e^{prT}\left(|X^{z,\pi}_t|^p+T^{p-1}|(z-r)|^p\int_t^s|\bpi^z_u|^p\rd u+ \sigma^p\left|\int_t^se^{-ru}\bpi^z_u\rd W_u\right|^p\right).
	\end{split}
\end{equation*}
By the Burkholder–Davis–Gundy and the Jensen inequalities, there exists $C_{1,p}>0$ such that, $\forall (t,x,\beta)\in[0,T)\times\Rbf^2$, $z\in\range(Z)$ and $s\in[t,T]$,
\begin{equation*}
	\begin{split}
		& \mathbf{E}\left[\left.\left|\int_t^se^{-ru}\bpi^z_u\rd W_u\right|^p\,\right|\,X^{z,\pi}_t=x,\bbeta^z_t=\beta\right]
		\\\le&  C_{1,p} \mathbf{E}\left[\left.\left(\int_t^s(e^{-ru}\bpi^z_u)^2\rd u\right)^{p\over 2}\,\right|\,X^{z,\pi}_t=x,\bbeta^z_t=\beta\right]
		\\\le & C_{1,p}(s-t)^{{p\over 2} -1} \mathbf{E}\left[\left.\int_t^s|\bpi^z_u|^p\rd u\,\right|\,X^{z,\pi}_t=x,\bbeta^z_t=\beta\right].
	\end{split}
\end{equation*}
By \eqref{inequ022}, there exist $C_{2,p}>0$ and $C_{3,p}>0$ such that, $\forall (t,x,\beta)\in[0,T)\times\Rbf^2$,
\begin{equation}\label{inequ:pi}
|\pi(t,x,\beta)|^p\leq (Ce^{|\beta|}+C|x|^{1\over 2})^p
\leq  C_{2,p}(e^{p|\beta|}+|x|^{p\over 2})
\leq C_{3,p}(e^{\beta^2}+|x|^{p\over 2}).
\end{equation}
Then there exists $C_{4,p}>0$ such that, $\forall (t,x,\beta)\in[0,T)\times\Rbf^2$, $z\in\range(Z)$ and $s\in[t,T]$,
\begin{equation*}
	\begin{split}
		&\mathbf{E}\left[\left.|X^{z,\pi}_s|^p\,\right|\,X^{z,\pi}_t=x,\bbeta^z_t=\beta\right]\\
		\le& C_{4,p}|x|^p+C_{4,p}(|z|^p+1)\left(
		\int_t^s\mathbf{E}\left[e^{{\bbeta^z_u}^2}\mid \bbeta^z_t=\beta\right]\rd u
		+\int_t^s\mathbf{E}\left[|X^{z,\pi}_u|^{p\over 2}\left|\genfrac{}{}{0pt}{}{X^{z,\pi}_t=x}{\bbeta^z_t=\beta}\right.\right]\rd u
		\right)\\
		\le& C_{4,p}|x|^p+C_{4,p}(|z|^p+1)\left(
		\int_t^s\mathbf{E}\left[e^{{\bbeta^z_u}^2}\mid \bbeta^z_t=\beta\right]\rd u
		+\int_t^s\left(\mathbf{E}\left[|X^{z,\pi}_u|^{p}\left|\genfrac{}{}{0pt}{}{X^{z,\pi}_t=x}{\bbeta^z_t=\beta}\right.\right]\right)^{1\over 2}\rd u
		\right)
	\end{split}
\end{equation*}
Now let $\epsilon>0$ be fixed. By Lemma \ref{lemma1}, there exists $h>0$ such that, $\forall (t,\beta)\in[0,T)\times\Rbf$ and $z\in\range(Z)$,
\[\sup_{s\in[t,(t+h)\land T)}\mathbf{E}\left[\left.\exp\left\{{\bbeta^z_s}^2\right\}\,\right|\, \bbeta^{z}_t=\beta\right]\leq \sqrt{2}e^{{\epsilon\over 2}z^2+3\sigma_0^4{\zeta(T)}^{-2}\beta^2}.\]
Then there exists $C_{5,p}>0$ such that, $\forall (t,x,\beta)\in[0,T)\times\Rbf^2$, $z\in\range(Z)$ and $s\in[t,(t+h)\land T)$,
\begin{equation*}
	\begin{split}
		\mathbf{E}\left[\left.|X^{z,\pi}_s|^p\,\right|\,X^{z,\pi}_t=x,\bbeta^z_t=\beta\right]
		\le& C_{5,p}|x|^p+C_{5,p}(|z|^p+1)
		e^{{\epsilon\over 2} z^2+3\sigma_0^4{\zeta(T)}^{-2}\beta^2}
		\\&+C_{5,p}(|z|^p+1)\int_t^s\left(\mathbf{E}\left[\left.|X^{z,\pi}_u|^{p}\,\right|\,X^{z,\pi}_t=x,\bbeta^z_t=\beta\right]\right)^{1\over 2}\rd u.
	\end{split}
\end{equation*}
By the generalized Gronwall inequality in \citet[Theorem 2]{willett1965discrete}, there exists $C_{1,p,\epsilon}>0$ such that, $\forall (t,x,\beta)\in[0,T)\times\Rbf^2$, $z\in\range(Z)$ and $s\in[t,(t+h)\land T)$,
 \begin{equation*}
	\begin{split}
		&\mathbf{E}\left[\left.|X^{z,\pi}_s|^p\,\right|\,X^{z,\pi}_t=x,\bbeta^z_t=\beta\right]
		\\
		\leq&
		\left(\left(C_{5,p}|x|^p+C_{5,p}(|z|^p+1)e^{{\epsilon\over 2}z^2+3\sigma_0^4{\zeta(T)}^{-2}\beta^2}\right)^{1\over 2}
		+{1\over 2}C_{5,p}(|z|^p+1) (s-t)\right)^2
		\\
		\leq& 2C_{5,p}|x|^p+2C_{5,p}(|z|^p+1)e^{{\epsilon\over 2} z^2+3\sigma_0^4{\zeta(T)}^{-2}\beta^2}+
		{1\over 2}C_{5,p}^2(|z|^p+1)^2 T^2
		\\
%		\leq& C_{6,p}|x|^p+C_{6,p}[|z|^{2p}+1]e^{{\epsilon\over 2} z^2+3\sigma_0^4{\zeta(T)}^{-2}\beta^2}
%		\\
%		\leq & C_{6,p}|x|^p+C_{6,p}[1+e^{2p|z|}]e^{{\epsilon\over 2} z^2+3\sigma_0^4{\zeta(T)}^{-2}\beta^2}
%		\\
%		\leq & C_{6,p}|x|^p+C_{1,p,\epsilon}[1+e^{{\epsilon\over 2} z^2}]e^{{\epsilon\over 2} z^2+3\sigma_0^4{\zeta(T)}^{-2}\beta^2}
%		\\
		\leq & 2C_{5,p}|x|^p+C_{1,p,\epsilon}e^{\epsilon z^2+3\sigma_0^4{\zeta(T)}^{-2}\beta^2}.
	\end{split}
\end{equation*}
Therefore,  \eqref{inequ:lemma0} holds.\qed

\paragraph{Proof of Proposition \ref{prop:admissible}.}

First, we show that, for every $(t,x,\beta)\in[0,T)\times\Rbf^2$, $\epsilon>0$ and $\delta>0$, there exist  $\tilde{t}\in(t,T)$ and $C>0$ such that \eqref{inequ01} holds for all $z\in\range(Z)$.

Let $(t,x,\beta)\in[0,T)\times\Rbf^2$, $\epsilon>0$ and $\delta>0$ be fixed.

By \eqref{equ:xintegral}, $\forall s\in[t,T]$ and $z\in\range(Z)$,
\begin{equation}
	\begin{split}
		&e^{-\delta ke^{r(T-s)} X^{z,\pi}_s}
		%=&e^{-\delta ke^{r(T-t)} X^{z,\pi}_t}\exp\left\{-\delta k\int_t^se^{r(s-u)}\bpi^z_u(z-r)\rd u-\delta k \int_t^se^{r(s-u)}\bpi^z_u\sigma\rd W_u\right\}
		\\=&e^{-\delta ke^{r(T-t)} X^{z,\pi}_t}\underbrace{\exp\left\{\int_t^s\left[-\delta k e^{r(T-u)}\bpi^z_u(z-r)+\delta^2 k^2e^{2r(T-u)}{\bpi^z_u}^2\sigma^2\right]\rd u\right\}}_{{G}_s}
		\\&\times\underbrace{\exp\left\{-\int_t^s\delta^2 k^2e^{2r(T-u)}{\bpi^z_u}^2\sigma^2\rd u-\int_t^s\delta ke^{r(T-u)}\bpi^z_u\sigma\rd W_u\right\}}_{{H}_s}.
	\end{split}
\end{equation}
Obviously, $\{{H}_s^2\}_{s\ge t}$ is a non-negative supermartingale w.r.t. $\{\F_s^W\}$. Hence,
$$\mathbf{E}[{H}_s^2\mid\F_t^W]\leq1,\ s\ge t.$$
Therefore, by the Cauchy–Schwarz inequality, $\forall s\in[t,T]$ and $z\in\range(Z)$,
\begin{equation}\label{inequ:inx}
	\begin{split}
		&\mathbf{E}
		\left[\left.e^{-\delta k e^{r(T-s)} X^{z,\pi}_s}\,\right|\,X^{z,\pi}_t=x,\bbeta^z_t=\beta\right]\\
		=&e^{-\delta ke^{r(T-t)} x}\mathbf{E}
		\left[\left.{G}_s{H}_s\,\right|\,X^{z,\pi}_t=x,\bbeta^z_t=\beta\right]\\
		\leq& e^{-\delta ke^{r(T-t)} x} \left(\mathbf{E}
		\left[\left.{G}_s^2\,\right|\,X^{z,\pi}_t=x,\bbeta^z_t=\beta\right]\right)^{1\over 2}\left(\mathbf{E}
		\left[\left.{H}_s^2\,\right|\,X^{z,\pi}_t=x,\bbeta^z_t=\beta\right]\right)^{1\over 2}
		\\
		\le& e^{-\delta ke^{r(T-t)} x} \left(\mathbf{E}
		\left[\left.{G}_s^2\,\right|\,X^{z,\pi}_t=x,\bbeta^z_t=\beta\right]\right)^{1\over 2}.
	\end{split}
\end{equation}
Because the integrand in ${G}_s$ is a quadratic function of $z$ and $\bpi^z_u$, there exist  $C_1>0$ and $C_2>0$ such that, $\forall s\in[t,T]$ and $z\in\range(Z)$,
\begin{equation}\label{inequ:G}
\begin{split}
	{G}_s^2&=\exp\left\{\int_t^s\left[-2\delta k  e^{r(T-u)}\bpi^z_u(z-r)+2\delta^2 k^2e^{2r(T-u)}{\bpi^z_u}^2\sigma^2\right]\rd u\right\}
	\\ &\leq C_1\exp\left\{\int_t^s[\epsilon z^2 +C_2{\bpi^z_u}^2]\rd u
	\right\}
	\\&= C_1 e^{\epsilon(s-t) z^2} \exp\left\{C_2\int_t^s{\bpi^z_u}^2\rd u
	\right\}.
%	\\&\leq C_{3,\epsilon} e^{\epsilon(s-t) z^2} \exp\left\{C_{4,\epsilon}\int_t^s{\bbeta^z_s}^2\rd u
%	\right\}.
\end{split}
\end{equation}
%By \eqref{inequ022}, $\forall (t,x,\beta)\in[0,T)\times\Rbf^2$,
%$$\pi(t,x,\beta)^2\leq 2C^2(1+\beta^2+\log(1+|x|)).$$
Then by \eqref{inequ022} and Jensen's inequality, there exist $C_3, C_4, C_5 >0$ such that,
$\forall s\in[t,T)$ and $z\in\range(Z)$,
\begin{equation}\label{inequ:G2}
	\begin{split}
		{G}_s^2&\leq C_3 e^{\epsilon(s-t) z^2} \exp\left\{C_4\int_t^s[{\bbeta^z_s}^2
		+\log(1+|X^{z,\pi}_u|)]\rd u
		\right\}
		\\&\le  C_3 \left[e^{2\epsilon(s-t) z^2} \exp\left\{2C_4\int_t^s {\bbeta^z_s}^2 \rd u
		\right\}+\exp\left\{2C_4\int_t^s\log(1+|X^{z,\pi}_u|)\rd u\right\}\right]
		\\&\le C_3 \left[e^{2\epsilon(s-t) z^2} \exp\left\{2C_4 \int_t^s {\bbeta^z_u}^2
		\rd u
		\right\}+{1\over s-t}\int_t^s[1+|X^{z,\pi}_u|]^{2C_4 (s-t)}\rd u\right]
		\\&\le C_5 \left[e^{2\epsilon(s-t) z^2} \exp\left\{2C_4 \int_t^s {\bbeta^z_u}^2
		\rd u
		\right\}+{1\over s-t}\int_t^s|X^{z,\pi}_u|^{2C_4 T}\rd u+1\right].
	\end{split}
\end{equation}
By Lemmas \ref{lemma1} and \ref{lemma0},  there exist $\tilde{t}\in\left(t,T\land(t+{1\over 4})\right)$, $C_6>0$ and $C_7>0$ such that, for all $z\in\range(Z)$,
\[
\begin{split}
 \sup_{s\in[t,\tilde{t}]} \mathbf{E}\left[\left.\exp\left\{2C_4 \int_t^s{\bbeta^z_s}^2\rd u
\right\}\,\right|\,\bbeta^z_t=\beta\right]\leq C_6 e^{{1\over 2}\epsilon z^2}
\end{split}
\]
and
\[
	\sup_{s\in[t,\tilde{t}]}\left\{\mathbf{E}\left[\left.|X^{z,\pi}_s|^{2C_4 T}\,\right|\, \bbeta^{z}_t=\beta\right]\right\}\leq C_7 e^{\epsilon z^2}.
\]
A combination of the last two inequalities with \eqref{inequ:inx} and  \eqref{inequ:G2} yields \eqref{inequ01}.

Now we show that, for every $(t,x,\beta)\in[0,T)\times\Rbf^2$, $\epsilon>0$ and $\rho>2$, there exist  $\tilde{t}\in(t,T)$ and $C>0$ such that \eqref{inequ02} holds for all $z\in\range(Z)$.

Let $(t,x,\beta)\in[0,T)\times\Rbf^2$, $\epsilon>0$ and $\rho>2$ be fixed.

By \eqref{inequ:pi}, there exist $C_{1, \rho}>0$ and $C_{2, \rho}>0$ such that, $\forall (t,x,\beta)\in[0,T)\times\Rbf^2$,
\[
|\pi(t,x,\beta)|^\rho\leq C_{1, \rho}\left(e^{\beta^2}
+|x|^{\rho\over 2}\right)\leq C_{2, \rho}\left(1+e^{\beta^2}
+|x|^{\rho}\right).
\]
Then
\[
|\pi(s,X^{z,\pi}_s,\bbeta^{z}_s)|^\rho\leq  C_{2, \rho}\left(1+\exp({\bbeta^z_s}^2)+|X^{z,\pi}_s|^\rho\right).
\]
A combination of  Lemmas \ref{lemma0}  and  \ref{lemma1} yields the existence of $\tilde{t}\in(t,T)$ and $C>0$ such that \eqref{inequ02} holds.
\qed

%{\color{red}
%\section{Proof of Lemma \ref{lma:m4m5}}\label{app:lma:m4m5}
%First, we verify $rm_2(t)+rm_3(t)+m_5(t)=0$. By \eqref{equ:ode}, we have
%\[
%\begin{split}
%	&\sigma^2[rm_2(t)+rm_3(t)+m_5(t)]'\\
%=&r(1+\zeta(t)m_2(t))a_1\left(t,m_2(t),m_3(t)\right) -r(1+\zeta(t)m_2(t))^2\\&+(1+\zeta(t)m_2(t))a_2\left(t,m_2(t),m_3(t)\right)m_5(t)\\
%=&(1+\zeta(t)m_2(t))a_2\left(t,m_2(t),m_3(t)\right)
%\left[r{a_1\left(t,m_2(t),m_3(t)\right) \over a_2\left(t,m_2(t),m_3(t)\right)}-r{(1+\zeta(t)m_2(t)) \over a_2\left(t,m_2(t),m_3(t)\right)}+m_5(t)\right]\\
%=&(1+\zeta(t)m_2(t))a_2\left(t,m_2(t),m_3(t)\right)[rm_2(t)+rm_3(t)+m_5(t)].
%\end{split}
%\]
%Because $rm_2(T)+rm_3(T)+m_5(T)=0$, we obtain $rm_2(t)+rm_3(t)+m_5(t)=0, t\in[0,T]$.
%Similarly,  by \eqref{equ:ode}, we have
%\[\sigma^2[rm_1(t)+rm_2(t)+m_4(t)]'=\zeta(t)[rm_1(t)+rm_2(t)+m_4(t)].\]
%As $rm_1(T)+rm_2(T)+m_4(T)=0$, we also have $rm_1(t)+rm_2(t)+m_4(t)=0, t\in[0,T]$.
%\qed
%}

\section{Proof of Lemma \ref{lma:ansatz}}\label{app:lma:ansatz}
%In the following, we show that $g^z$ given by \eqref{eq:ans:gz}--\eqref{eq:ans:fz} and $\pi^*$ given by \eqref{equ:u*} satisfy \eqref{equ:hjb:beta}--\eqref{equ:hjbg2:beta}.

First, we verify \eqref{equ:m4m5}. By \eqref{equ:ode}, we have
\[
\begin{split}
	&\sigma^2[rm_2(t)+rm_3(t)+m_5(t)]'\\
=&r(1+\zeta(t)m_2(t))a_1\left(t,m_2(t),m_3(t)\right) -r(1+\zeta(t)m_2(t))^2\\&+(1+\zeta(t)m_2(t))a_2\left(t,m_2(t),m_3(t)\right)m_5(t)\\
=&(1+\zeta(t)m_2(t))a_2\left(t,m_2(t),m_3(t)\right)
\left[r{a_1\left(t,m_2(t),m_3(t)\right) \over a_2\left(t,m_2(t),m_3(t)\right)}-r{(1+\zeta(t)m_2(t)) \over a_2\left(t,m_2(t),m_3(t)\right)}+m_5(t)\right]\\
=&(1+\zeta(t)m_2(t))a_2\left(t,m_2(t),m_3(t)\right)[rm_2(t)+rm_3(t)+m_5(t)].
\end{split}
\]
Because $rm_2(T)+rm_3(T)+m_5(T)=0$, we obtain $rm_2(t)+rm_3(t)+m_5(t)=0, t\in[0,T]$.
Similarly,  by \eqref{equ:ode}, we have
\[\sigma^2[rm_1(t)+rm_2(t)+m_4(t)]'=\zeta(t)[rm_1(t)+rm_2(t)+m_4(t)].\]
As $rm_1(T)+rm_2(T)+m_4(T)=0$, we also have $rm_1(t)+rm_2(t)+m_4(t)=0, t\in[0,T]$.

Second, by \eqref{eq:ans:fz} and the terminal condition of \eqref{equ:ode}, we have
$f^z(T,\beta)=0$. Then by \eqref{eq:ans:gz}, we have \eqref{equ:hjbg2:beta}.

Third,  we verify \eqref{equ:hjbg:beta}. Plugging \eqref{eq:ans:gz} into $\Ac^{z,\pi^*}g^z$ yields
\begin{equation}\label{equ:Acg}
	\begin{split}
		\Ac^{z,\pi^*}g^z(t,x,\beta)
		=&g^z(t,x,\beta)\left\{f^z_t-ke^{r(T-t)}\left[\pi^*( z-r)\right]+ {\zeta\over\sigma^2}( z-\beta)f^z_\beta\right.\\
		&\left.+\frac{1}{2}k^2e^{-2r(T-t)}\sigma^2\pi^{*2} +\frac{1}{2}
		{\zeta^2\over\sigma^2}\left[f^z_{\beta\beta}+
		(f^z_\beta)^2\right ]-ke^{-r(T-t)}\zeta\pi^* f^z_\beta\right\}
	\end{split}
\end{equation}
Substituting \eqref{eq:ans:fz} and \eqref{equ:u*} into \eqref{equ:Acg} and sorting by the orders
%(corresponding to the ODE \eqref{equ:ode}))
of $\beta^2$, $z\beta$, $z^2$, $\beta$, $z$ and $1$  yields
\begin{equation*}
\begin{split}
&\Ac^{z,\pi^*}g^z(t,x,\beta)\\
=&\sigma^{-2}g^z(t,x,\beta)\bigg\{\frac{1}{2}\left[\sigma^2m_1^\prime(t)-2 \zeta(t)m_1(t)+ a_1^2(t,m_2(t),m_3(t))\right]\beta^2\\
&+\left[\sigma^2m_2^\prime(t)-(1+\zeta(t)m_2(t))a_1(t,m_2(t),m_3(t)) -\zeta(t)m_2(t)\right]\beta z\\
&+\frac{1}{2}\left[\sigma^2m_3^\prime(t)+2 \zeta(t)m_2(t)+\zeta^{2}(t)m_2^2(t)\right]z^2\\
&+\left[\sigma^2m_4^\prime(t)-\zeta(t)m_4(t)+ a_1(t,m_2(t),m_3(t)) a_2(t,m_2(t),m_3(t))m_5(t)
		+r \zeta(t)m_1(t)\right]\beta\\
&+\left[\sigma^2m_5^\prime(t)-(1+\zeta(t)m_2(t))a_2(t,m_2(t),m_3(t))m_5(t)+r(1+\zeta(t)m_2(t))\right]z\\
		&+\left[\sigma^2m_6^\prime(t)+\frac{1}{2} a_2^2(t,m_2(t),m_3(t))m^2_5(t)-\frac{1}{2}r^2
		+\frac{1}{2}\zeta^{2}(t) m_1(t)+r\zeta(t)m_4(t)\right]
		\bigg\}.
\end{split}
\end{equation*}
Then by \eqref{equ:ode}, we know that \eqref{equ:hjbg:beta} holds.

Finally, we prove \eqref{equ:hjb:beta}. Actually, $\mathbf{E}\left[\phi^\prime\left(g^ Z(t,x,\beta)\right)\Ac^{ Z,\pi} g^ Z(t,x,\beta)\mid \bbeta_t=\beta\right]$
is a quadratic function of $\pi$, where the coefficient of $\pi^2$ in \eqref{equ:hjb:beta} is
\[
\frac{1}{2}\sigma^2k^2e^{2r(T-t)}\mathbf{E}\left[\phi^\prime(g^ Z(t,x,\beta))g^ Z(t,x,\beta) \mid \bbeta_t=\beta\right]<0,
\]
  {since $g^z_{xx}(t,x,\beta)=k^2e^{2r(T-t)}g^z(t,x,\beta)$ by \eqref{eq:ans:gz}.}  Then the supremum in \eqref{equ:hjb:beta} is attained at
\begin{equation}\label{ofc:beta}
	\begin{split}
	\hat{\pi}(t,x,\beta)=&
	-\frac{\mathbf{E}\left[\phi^\prime(g^ Z(t,x,\beta))g_{x}^ Z(t,x,\beta)( Z-r) \mid \bbeta_t=\beta\right]}
	{\mathbf{E}\left[\phi^\prime(g^ Z(t,x,\beta))g^ Z_{xx}(t,x,\beta)\sigma^2 \mid \bbeta_t=\beta\right]}\\
	&-\frac{\mathbf{E}\left[\phi^\prime(g^ Z(t,x,\beta))g^ Z_{x\beta}(t,x,\beta)\zeta(t) \mid \bbeta_t=\beta\right]}
	{\mathbf{E}\left[\phi^\prime(g^ Z(t,x,\beta))g^ Z_{xx}(t,x,\beta)\sigma^2 \mid \bbeta_t=\beta\right]}
	\\=&-\frac{\mathbf{E}\left[\phi^\prime(g^ Z(t,x,\beta))g_{x}^ Z(t,x,\beta) \mid \bbeta_t=\beta\right]}
	{\mathbf{E}\left[\phi^\prime(g^ Z(t,x,\beta))g^ Z_{xx}(t,x,\beta)\sigma^2 \mid \bbeta_t=\beta\right]}( \beta-r)\\
	&-\frac{\mathbf{E}\left[\phi^\prime(g^ Z(t,x,\beta))g_{x}^ Z(t,x,\beta)( Z-\beta) \mid \bbeta_t=\beta\right]}
	{\mathbf{E}\left[\phi^\prime(g^ Z(t,x,\beta))g^ Z_{xx}(t,x,\beta)\sigma^2 \mid \bbeta_t=\beta\right]}\\
	&-\frac{\mathbf{E}\left[\phi^\prime(g^ Z(t,x,\beta))g^ Z_{x\beta}(t,x,\beta)\zeta(t) \mid \bbeta_t=\beta\right]}
	{\mathbf{E}\left[\phi^\prime(g^ Z(t,x,\beta))g^ Z_{xx}(t,x,\beta)\sigma^2 \mid \bbeta_t=\beta\right]}.
	%\\
	%&=-\frac{\int_{\mathbf{R} }\phi^\prime(g^z)g_{x}^z(z-r) p\rd z}{\int_{\mathbf{R} }\phi^\prime(g^z)g^z_{xx}\sigma^2 p\rd z}
	%-\frac{\int_{\mathbf{R} }\phi^\prime(g^z)g_{x\beta}^z\zeta(t) p\rd z}{\int_{\mathbf{R} }\phi^\prime(g^z)g^z_{xx}\sigma^2 p\rd z},
	\end{split}
\end{equation}
Plugging \eqref{eq:ans:gz} into \eqref{ofc:beta} yields
\begin{equation}\label{equ:piAB}
	\hat{\pi}(t,x,\beta)=\frac{e^{-r(T-t)}}{k\sigma^2}\left[(\beta-r)+\underbrace{
		{\Ebf[(Z-\beta)e^{\alpha f^Z(t,\beta)} \mid \bbeta_t=\beta]\over \Ebf[e^{\alpha f^Z(t,\beta)}\mid\bbeta_t=\beta]}
	}_{\pi_A(t,x,\beta)}
	+\underbrace{
		{\zeta(t)\Ebf[e^{\alpha f^Z(t,\beta)}f^Z_\beta(t,\beta) \mid \bbeta_t=\beta]\over \Ebf[e^{\alpha f^Z(t,\beta)}\mid\bbeta_t=\beta]}
	}_{\pi_B(t,x,\beta)}\right].
\end{equation}
Obviously,
$$\pi_A(t,x,\beta) ={\Ebf[Ze^{\alpha f^Z(t,\beta)}\mid\bbeta_t=\beta]\over \Ebf[e^{\alpha f^Z(t,\beta)}\mid\bbeta_t=\beta]}
-\beta.$$
Moreover,  a substitution yields
$$\pi_B(t,x,\beta)={\zeta(t)m_2(t)\Ebf[Ze^{\alpha f^Z(t,\beta)}\mid \bbeta_t=\beta]\over \Ebf[e^{\alpha f^Z(t,\beta)}\mid\bbeta_t=\beta]}+\zeta(t)(m_1(t)\beta+m_4(t)).$$
By $Z|(\bbeta_t=\beta)\sim N(\beta,\zeta(t))$, $\zeta(t)^{-1}-\alpha m_3(t)>0$ and Lemma \ref{lma:exp:normal},
\begin{align*}
&{\Ebf[Ze^{\alpha f^Z(t,\beta)}\mid \bbeta_t=\beta]\over \Ebf[e^{\alpha f^Z(t,\beta)}\mid\bbeta_t=\beta]}\\
=&{\Ebf[Z\exp\left\{\frac{1}{2}\alpha m_3(t)Z^2+\alpha m_2(t)\beta Z+\alpha m_5(t)Z\right\}\mid \bbeta(t)=\beta]\over
\Ebf[\exp\left\{\frac{1}{2}\alpha m_3(t)Z^2+\alpha m_2(t)\beta Z+\alpha m_5(t)Z\right\}\mid \bbeta(t)=\beta]}\\
=&\frac{\zeta(t)^{-1}\beta+\alpha m_2(t)\beta+\alpha m_5(t)}{\zeta(t)^{-1}-\alpha m_3(t)}.
\end{align*}
{
Then
\[
\pi_A(t,x,\beta)=\frac{\alpha m_2(t)+\alpha m_3(t)}{\zeta(t)^{-1}-\alpha m_3(t)}\beta+
\frac{\alpha m_5(t)}{\zeta(t)^{-1}-\alpha m_3(t)}
\]
and
\[
\pi_B(t,x,\beta)=\left[\frac{ m_2(t)+\alpha \zeta(t) m_2^2(t)}{\zeta(t)^{-1}-\alpha m_3(t)}+\zeta(t)m_1(t)\right]\beta+\frac{\alpha\zeta(t)m_2(t) m_5(t)}{\zeta(t)^{-1}-\alpha m_3(t)}+\zeta(t)m_4(t).
\]
By \eqref{equ:m4m5}, we obtain
\[
\pi_A(t,x,\beta)=\frac{\alpha m_2(t)+\alpha m_3(t)}{\zeta(t)^{-1}-\alpha m_3(t)}(\beta-r)
\]
and
\[
\pi_B(t,x,\beta)=\left[\frac{ m_2(t)+\alpha \zeta(t) m_2^2(t)}{\zeta(t)^{-1}-\alpha m_3(t)}+\zeta(t)m_1(t)\right](\beta-r).
\]
}
Substituting it into \eqref{equ:piAB} and comparing $\hat{\pi}$ with \eqref{equ:u*} yields $\hat{\pi}=\pi^*$. Then we have
\begin{align*}
&\sup_{\pi\in \Pi}\mathbf{E}\left[\phi^\prime\left(g^ Z(t,x,\beta)\right)\Ac^{ Z,\pi} g^ Z(t,x,\beta)\mid \bbeta_t=\beta\right]\\
=&\mathbf{E}\left[\phi^\prime\left(g^ Z(t,x,\beta)\right)\Ac^{ Z,\hat{\pi}} g^ Z(t,x,\beta)\mid \bbeta_t=\beta\right]\\
=&\mathbf{E}\left[\phi^\prime\left(g^ Z(t,x,\beta)\right)\Ac^{ Z,\pi^*} g^ Z(t,x,\beta)\mid \bbeta_t=\beta\right]\\
=&0
\end{align*}
and \eqref{equ:hjb:beta} holds. \qed

\section{Proof of Proposition \ref{prop:m2m3}}\label{app:prop:m2m3}

Let $$\hat{m}_2(t)=\zeta(t)m_2(t).$$ Obviously, $\zeta^\prime(t)=-\sigma^{-2}\zeta^2(t)$. Then  \eqref{equ:ode:m2m3} is equivalent to
\begin{align}
&\sigma^2\hat{m}_2^\prime(t)=\zeta(t) [\hat{m}_2(t)+1] a_1\left(t,m_2(t),m_3(t)\right),\ \hat{m}_2(T)=0, \label{equ:ode:hm2}\\
&\sigma^2 m_3^\prime(t)=-2\hat{m}_2(t)-\hat{m}_2^2(t),\ m_3(T)=0.\label{equ:ode:hm3}
\end{align}

We first consider the case of $\alpha=0$. In this case, \eqref{equ:ode:hm2} reads
\begin{equation}\label{eq:hatm2:sigze}
\hat{m}_2^\prime(t)=\sigma^{-2}\zeta(t)[\hat{m}_2(t)+1]^2,\ \hat{m}_2(T)=0.
\end{equation}
We can see that\footnote{\label{footnote:hatm2} Suppose on the contrary that $\hat m_2(t)\le -1$ for some $t\in[0,T)$. Let
$t_0=\sup\{t\in[0,T)\mid \hat m_2(t)\le-1\}$. Then $t_0<T$, $\hat m_2(t_0)=-1$ and \eqref{eq:d:hatm2:0} holds on $(t_0,T)$.
Therefore,  \eqref{eq:hatm2:0} holds on $(t_0,T]$,
 which implies $-1=\hat m_2(t_0)=\frac{1}{1+\log(\sigma^2+\sigma_0^2T)-\log(\sigma^2+\sigma_0^2 t_0)}-1>-1$, a contradiction.
}
$\hat m_2(t)>-1$ for all $t\in[0,T)$ and hence
\begin{equation}\label{eq:d:hatm2:0}
-\rd {1\over \hat m_2(t)+1}={\zeta(t)\over\sigma^2}
\end{equation}
holds on $[0,T)$.
Therefore,
\begin{equation}\label{eq:hatm2:0}
\hat{m}_2(t)=\frac{1}{1+\log(\sigma^2+T\sigma_0^2)-\log(\sigma^2+t\sigma_0^2)}-1={1\over 1+\log{\zeta(t)\over\zeta(T)}}-1
\end{equation}
for all $t\in[0,T]$. Obviously,
$$-1<{1\over 1+\log(\sigma^2+T\sigma_0^2)-\log\sigma^2}-1\le\hat m_2(t)<0\quad\text{for all }t\in[0,T).$$
Then by \eqref{equ:ode:hm3}--\eqref{eq:hatm2:sigze},
\begin{align*}
m_3(t)&=\sigma^{-2}\int_t^T\left[2\hat{m}_2(s)+\hat{m}_2^2(s)\right]\rd s\\
&=\int_t^T\zeta(s)^{-1}\rd\hat m_2(s)-\sigma^{-2}(T-t)\\
&=\int_t^T\zeta(s)^{-1}\rd \frac{1}{1+\log{\zeta(s)\over\zeta(T)}}-{T-t\over\sigma^2}\\
&=\int_{\log{\zeta(t)\over\zeta(T)}}^0\zeta(T)^{-1} e^{-x}d{1\over 1+x}-{T-t\over\sigma^2}\\
&=\zeta(T)^{-1}\int^{\log{\zeta(t)\over\zeta(T)}}_0 e^{-x}(1+x)^{-2}dx-{T-t\over\sigma^2}.
%&\xlongequal{x=\log(\sigma^2+s\sigma_0^2)}
%\int_{\log(\sigma^2+t\sigma_0^2)}^{\log(\sigma^2+T\sigma_0^2)}
%{e^x\over \sigma^2\sigma_0^2}\rd \frac{1}{1+\log(\sigma^2+T\sigma_0^2)-x}-{T-t\over\sigma^2}
%\\&\xlongequal{y=x-1-\log(\sigma^2+T\sigma_0^2)}
%\int^{-1}_{-1+\log {\sigma^2+t\sigma_0^2 \over \sigma^2+T\sigma_0^2}}
%{\sigma^2+T\sigma_0^2 \over \sigma^2\sigma^2_0}e^{y+1}y^{-2}\rd y-{T-t\over\sigma^2}
%\\&\xlongequal{x=-y}
%-{\sigma^2+T\sigma_0^2 \over \sigma^2\sigma^2_0}\int_1^{1+\log {\sigma^2+T\sigma_0^2 \over \sigma^2+t\sigma_0^2}}e^{-x+1}x^{-2}\rd x-{T-t\over\sigma^2}.
\end{align*}
Thus, assertion (a) is proved.

Now we consider the case of $\alpha\ne0$. In this case,  let
 $$\hat{m}_3(t)=\zeta(t)^{-1}-\alpha m_3(t), \quad t\in[0,T].$$
Then  \eqref{equ:ode:hm2}--\eqref{equ:ode:hm3} reads
\begin{align}
	&\hat{m}_2^\prime(t)=\sigma^{-2}\frac{[1+\hat{m}_2(t)]^2[1+\!\alpha
		\hat{m}_2(t)]}{\hat{m}_3(t)},\ \hat m_2(T)=0,\label{equ:hm2n}  \\
	&\hat{m}_3^\prime(t)=\sigma^{-2}[1+2\alpha\hat{m}_2(t)
	+\alpha\hat{m}_2^2(t)], \ \hat m_3(T)=\zeta(T).\label{equ:hm4n}
\end{align}
%Then $\hat m_2^\prime(t)>0$ for all $t\in[0,T)$, by \eqref{equ:hm2n}.
Combining \eqref{equ:hm2n}--\eqref{equ:hm4n}, we have\footnote{Similarly to footnote \ref{footnote:hatm2}, we can show, {for $\alpha>\alpha^*$},
$${[1+\hat{m}_2(t)]^2[1+\!\alpha
		\hat{m}_2(t)]}>0,\quad t\in[0,T).$$}
\begin{equation}\label{m2m4n}
	\frac{1+2\alpha\hat{m}_2(t)
		+\alpha\hat{m}_2^2(t)}{[1+\hat{m}_2(t)]^2[1+\!\alpha
		\hat{m}_2(t)]}\hat{m}_2^\prime(t)=\frac{\hat{m}_3^\prime(t)}{\hat{m}_3(t)}.
\end{equation}

In the case of $\alpha=1$, \eqref{m2m4n} reads
\[
\frac{\hat{m}_2^\prime(t)}{1+\hat{m}_2(t)}=\frac{\hat{m}^\prime_3(t)}{\hat{m}_3(t)}.
\]
Integrating both sides yields
$$\hat{m}_3(t)=\hat{m}_3(T)[1+\hat{m}_2(t)].$$
Substituting it into \eqref{equ:hm2n} with $\alpha=1$, we have
\[
\hat{m}^\prime_2(t)=\sigma^{-2}\zeta(T)[1+\hat{m}_2(t)]^2,~\hat{m}_2(T)=0,
\]
which admits a unique solution
$$\hat{m}_2(t)=\frac{1}{1+\sigma^{-2}\zeta(T)(T-t)}-1=
{\sigma^2+T\sigma_0^2\over\sigma^2+T\sigma_0^2+\sigma_0^2(T-t)}
-1.$$
Meanwhile, we have
\begin{align*}
	\hat{m}_3(t)=\frac{\zeta(T)^{-1}}{1+\sigma^{-2}\zeta(T)(T-t)}.
	%={\sigma^{-2}\sigma_0^{-2}[\sigma^2+T\sigma_0^2]^2\over\sigma^2+T\sigma_0^2+\sigma_0^2(T-t)}.
	\end{align*}
Obviously,
$$-1<-{T\sigma_0^2\over\sigma^2+2T\sigma_0^2}\le\hat m_2(t)<0\quad\text{for all }\in[0,T).$$
%Therefore,
%\begin{equation*}
%	\begin{cases}
%		&m_2(t)=\frac{\zeta(t)}{1+\sigma^{-2}\zeta(T)(T-t)}-\zeta(t),\\
%		&m_3(t)=\zeta(t)-\frac{\zeta(t)}{1+\sigma^{-2}\zeta(T)(T-t)}.
%	\end{cases}
%\end{equation*}
Thus, assertion (b) is proved.

Now we consider the general case of $\alpha\notin\{0,1\}$.
Decomposing the left-hand side of \eqref{m2m4n}  leads to
\[
\left\{\frac{1}{[1+\hat{m}_2(t)]^2}
+\frac{\alpha}{\alpha-1}\frac{1}{[1+\hat{m}_2(t)]}
-\frac{\alpha}{\alpha-1}\frac{1}{[1+\!\alpha
	\hat{m}_2(t)]}\right\}\hat{m}^\prime_2(t)=\frac{\hat{m}^\prime_3(t)}{\hat{m}_3(t)}.
\]
Integrating the previous equation from $t$ to $T$ yields
\begin{equation}\label{m2m4}
	\hat{m}_3(t)=\hat{m}_3(T)e^{\frac{\hat{m}_2(t)}{1+\hat{m}_2(t)}}\left[1+\hat{m}_2(t)\right]^{\frac{\alpha}{\alpha-1}}\left[1+\!\alpha
	\hat{m}_2(t)\right]^{-\frac{1}{\alpha-1}}.
\end{equation}
%where $\hat{m}_3(T)=\zeta(t)$.
Then, substituting \eqref{m2m4} into \eqref{equ:hm2n}, we have
\begin{equation}\label{equ:hm2nn}
	\hat{m}^\prime_2(t)=\sigma^{-2}\zeta(T)e^{-\frac{\hat{m}_2(t)}{1+\hat{m}_2(t)}}\left[1+\hat{m}_2(t)\right]^{\frac{\alpha-2}{\alpha-1}}\left[1+\!\alpha
	\hat{m}_2(t)\right]^{\frac{\alpha}{\alpha-1}},~\hat{m}_2(T)=0.
\end{equation}
which is an autonomous ODE of the first order.
%To show the existence and uniqueness of $\hat{m}_2(t)$ given by \eqref{equ:hm2nn}, we first denote
Let
$$\bar{m}_2(t)=-\hat{m}_2(T-t),\quad t\in[0,T].$$
Then \eqref{equ:hm2nn} reads
\begin{equation}\label{equ:bm2}
	%\begin{split}
	\bar{m}_2^\prime(t)= \sigma^{-2}\zeta(T) \psi(\bar m_2(t)),
	%&=\sigma^{-2}\zeta(T)e^{-\frac{\bar{m}_2(t)}{1+\bar{m}_2(t)}}\left[1+\bar{m}_2(t)\right]^{\frac{\alpha-2}{\alpha-1}}\left[1+\alpha
	%\bar{m}_2(t)\right]^{\frac{\alpha}{\alpha-1}}
	%\\&
	%=\sigma^{-2}\zeta(T)e^{\frac{\bar{m}_2(t)}{1-\bar{m}_2(t)}}\left(1-\bar{m}_2(t)\right)^{\frac{\alpha-2}{\alpha-1}}\left(1-\alpha
	%\bar{m}_2(t)\right)^{\frac{\alpha}{\alpha-1}},
	\quad \bar{m}_2(0)=0, \, t\in[0,T).
	%\end{split}
\end{equation}
where function $\psi$ is given by
\begin{align*}
\psi(x)=e^{\frac{x}{1-x}}\left(1-x\right)^{\frac{\alpha-2}{\alpha-1}}\left(1-\alpha
x\right)^{\frac{\alpha}{\alpha-1}}, \quad x\in\left[0,\frac{1}{\alpha\vee 1}\right).
\end{align*}
Let\begin{align*}
\Psi(x)=\int^x_0\frac{1}{\psi(s)}\rd s, \quad x\in\left[0,\frac{1}{\alpha\vee 1}\right].
\end{align*}
Obviously, $\psi(x)$ is continuous and positive on $\left[0,\frac{1}{\alpha\vee 1}\right)$.
Then $\Psi$ is strictly increasing on $\left[0,\frac{1}{\alpha\vee 1}\right]$ and continuously differentiable on $\left(0,\frac{1}{\alpha\vee 1}\right)$.

Consider the case when $\alpha>1$. In this case,
\begin{align*}
\Psi\left({1\over\alpha}\right)%=&-\int_{-{1\over\alpha}}^0\frac{1}{\psi(s)}\rd s\\
&=
		\int^{\frac{1}{\alpha}}_0e^{\frac{-s}{1-s}}\left(1-s\right)^{-\frac{\alpha-2}{\alpha-1}}\left(1-\!\alpha
		s\right)^{-\frac{\alpha}{\alpha-1}}
		\rd s	 \\
		&\geq C\int^{\frac{1}{\alpha}}_0\left(1-\!\alpha
		s\right)^{-\frac{\alpha}{\alpha-1}}
		\rd s=+\infty,
\end{align*}
where $C>0$ is the minimum of function $e^{\frac{-s}{1-s}}\left(1-s\right)^{-\frac{\alpha-2}{\alpha-1}}$ on  $[0,\frac{1}{\alpha}]$.
Therefore, when $\alpha>1$, $$\bar m_2(t)=\Psi^{-1}(\sigma^{-2}\zeta(T)t), \quad t\in[0,T],$$
is the unique solution to \eqref{equ:bm2}.\footnote{The uniqueness is clear. Actually, suppose that $n(t)$ is another solution to \eqref{equ:bm2}. Then the chain rule implies that
${\rd \Psi(n(t))\over \rd t }=\Psi^\prime(n(t))n^\prime(t)={n^\prime(t)\over \psi(n(t))}=\sigma^{-2}\zeta(T)$,
which implies that $\Psi(n(t))=\sigma^{-2}\zeta(T)t+C$ for all $t\in[0,T]$ and some constant $C$. Then by $\Psi(0)=0$, we have $C=0$ and hence $n(t)=\Psi^{-1}(\sigma^{-2}\zeta(T)t)$.
}
Consequently,
$$\hat m_2(t)=-\Psi^{-1}\left((\sigma^{-2}\zeta(T)(T-t)\right),\quad t\in[0,T],$$
is the unique solution to \eqref{equ:hm2nn}. Moreover, it is easy to see that
$$0>\hat m_2(t)>-\Psi^{-1}(\sigma^{-2}\zeta(T)T)>-{1\over\alpha},\quad t\in[0,T).$$

Consider the case when $\alpha<1$.  Let
\[
\begin{split}
	T_\alpha\trieq\sigma^2{\zeta(T)}^{-1}\Psi(1)&=
	\int^{1}_0\frac{\sigma^2}{\zeta(T)}e^{\frac{-s}{1-s}}\left(1-s\right)^{-\frac{\alpha-2}{\alpha-1}}\left(1-\!\alpha
	s\right)^{-\frac{\alpha}{\alpha-1}}
	\rd s	 \\
	&=\left(T+\sigma_0^{-2}\sigma^2\right)\int^{1}_0e^{\frac{-s}{1-s}}\left(1-s\right)^{-2}\left(\frac{1-\!\alpha
		s}{1-s}\right)^{-\frac{\alpha}{\alpha-1}}
	\rd s\\
	&=\left(T+\sigma_0^{-2}\sigma^2\right)\int_0^\infty e^{-x}(1+(1-\alpha)x)^{\alpha\over 1-\alpha}\rd x.
	\end{split}
\]

In the case of $\alpha\in(0,1)$, we have
\[
	T_\alpha\geq \left(T+\sigma_0^{-2}\sigma^2\right)\int_0^\infty e^{-x}\rd x
	=T+\sigma_0^{-2}\sigma^2.
\]
Therefore,
$$\bar m_2(t)=\Psi^{-1}(\sigma^{-2}\zeta(T)t),\quad t\in[0,T],$$ is the unique solution to \eqref{equ:bm2}. Consequently,
$$\hat m_2(t)=-\Psi^{-1}\left((\sigma^{-2}\zeta(T)(T-t)\right),\quad t\in[0,T],$$
is the unique solution to \eqref{equ:hm2nn}. Moreover, it is easy to see that
$$0>\hat m_2(t)>-\Psi^{-1}(\sigma^{-2}\zeta(T)T)>-1,\quad t\in[0,T).$$

Now we consider the case when $\alpha<0$.  For every $\alpha\le0$, let
$$f_\alpha(x)=e^{-x}(1+(1-\alpha)x)^{\alpha\over 1-\alpha},\quad x\in(0,\infty).$$
Then
\[
	T_\alpha=\left(T+\sigma_0^{-2}\sigma^2\right)\int_0^\infty f_\alpha(x)\rd x.
\]
For every $x\in(0,\infty)$,  $\{f_\alpha(x)\}$ is continuous and strictly increasing w.r.t. $\alpha$ on $(-\infty,0]$.
Obviously,  $T_0=T+\sigma_0^{-2}\sigma^2$.
Then by the  monotone convergence theorem, $T_\alpha$ is continuous and strictly increasing w.r.t. $\alpha$ on $(-\infty,0]$. Moreover,
$$\lim_{\alpha\rightarrow -\infty}T_\alpha
=\left(T+\sigma_0^{-2}\sigma^2\right)\lim_{\alpha\rightarrow -\infty}\int_0^\infty f_\alpha(x)\rd x
=\left(T+\sigma_0^{-2}\sigma^2\right)\int_0^\infty \lim_{\alpha\rightarrow -\infty}f_\alpha(x)\rd x
=0.
$$
Then there exists a unique $\alpha^*\in(-\infty,0)$ such that $T_{\alpha^*}=T$, i.e., equation \eqref{equ:alpha*} has a unique solution $\alpha^*\in(-\infty,0)$.
Obviously,
\[
\begin{cases}
	T_\alpha>T \ \text{ if } 0>\alpha>\alpha^*,\\
	T_\alpha<T \ \text{ if }  \alpha<\alpha^*.
\end{cases}
\]
Therefore, if $\alpha\in(\alpha^*,0)$, then
$$\bar m_2(t)=\Psi^{-1}(\sigma^{-2}\zeta(T)t),\quad t\in[0,T],$$
is the unique solution to \eqref{equ:bm2}.
Consequently,
$$\hat m_2(t)=-\Psi^{-1}(\sigma^{-2}\zeta(T)(T-t)),\quad t\in[0,T],$$
is the unique solution to \eqref{equ:hm2nn}. Moreover, it is easy to see that
$$0>\hat m_2(t)>-\Psi^{-1}(\sigma^{-2}\zeta(T)T)>-1,\quad t\in[0,T).$$
Thus, assertion (c) is proved.

Finally, from the above discussion, we have that, for every $\alpha\in(\alpha^*,\infty)$,
$$\hat m_2(t)\in\left(-{1\over\alpha\vee 1},0\right)\quad \text{for all }t\in[0,T).$$
Then by \eqref{equ:ode:hm3}, $m_3^\prime(t)>0$ for all $t\in[0,T)$ and hence ${m}_3(t)<0$ for all $t\in[0,T)$.
Furthermore, by \eqref{m2m4}, we have $\hat{m}_3(t)>0$, i.e., $\zeta(t)^{-1}-\alpha m_3(t)>0$, for all $t\in[0,T]$.
Thus, the proposition is proved.\qed
%
%\
%
%Summarizing the above statements,  \eqref{equ:bm2} exists a unique
%solution $\bar{m}_2(t)=\Psi^{-1}(t)$ on  $[0,T_{\alpha})$.  We also have \[x_\alpha<\bar{m}_2(t)\leq0,~t\in[0,T_{\alpha}).\] Besides,
%\[
%\lim_{t\rightarrow T_\alpha} \bar{m}_2(t)=x_\alpha.
%\]
% By \eqref{m2m4}, \eqref{equ:hm4n} also exists a unique solution $\hat{m}_3(t)$ on  $[0,T_{\alpha})$ and
% \[
% \lim_{t\rightarrow T_\alpha} \hat{m}_3(t)=0.
% \]
% Moreover,
%  \[
% \lim_{t\rightarrow T_\alpha} a_1(t,m_2(t),m_3(t))=+\infty.
% \]
%
%Then \eqref{equ:bm2} exists a unique solution on  $[0,T]$ if and only if $T_\alpha> T$, i.e., $\alpha> \alpha^*$.  Because  $\hat{m}_2(t)=\bar{m}_2(T-t)$, if $\alpha>\alpha^*$, \eqref{equ:hm2nn} also exists a unique solution $\hat{m}_2(t)=\Psi^{-1}(T-t)$ on  $[0,T]$. Then, by \eqref{m2m4}, we can also obtain a closed form solution of $\hat{m}_3(t)$ on $[0,T]$, similar for $m_3(t)$.
%Finally, a substitution leads to the conclusion of the lemma.

\section{Proof of Theorem \ref{thm:result}}\label{app:thm:result}

%{\color{red} update the proof since $m_4$ and $m_5$ are not in the new formula of $\pi^*$.}

Obvious,  $\pi^*$ is admissible, by Proposition \ref{prop:admissible}.
 Condition (a) of Theorem \ref{thm} is guaranteed by Lemma \ref{lma:ansatz} and Proposition \ref{prop:m2m3}.
Then we prove that $\{g^z\}_{z\in\range(Z)}$ and $\pi^*$ satisfy conditions (b)--(c) of Theorem \ref{thm}.
The proof is divided into the following six steps.

\paragraph{Step 1.} Let $z\in\range(Z)$ be fixed. We show $\left\{g^z(t,X^{z,\pi^*}_t,\bbeta^{z}_t)\right\}_{t\in[0,T]}$ is a martingale w.r.t. filtration $\{\mathcal{F}^W_t\}$.
Actually,
\eqref{equ:u*} can be rewritten as
\begin{equation}\label{equ:pi2}
	\pi^*(t,x,\beta)={e^{-r(T-t)} \over k\sigma^2}A(t)\left(\beta-r\right),
\end{equation}
where
\[
%\begin{cases}
A(t)=a_1(t,m_2(t),m_3(t))+\zeta(t)m_1(t).%\\ %B(t)=a_2(t,m_2(t),m_3(t))m_5(t)+\zeta(t)m_4(t)-r.
%\end{cases}
\]
$A(t)$ is bounded and continuous on $[0,T]$.  Applying It\^{o}'s formula and using \eqref{equ:hjbg},
\begin{equation*}\label{equ:dg}
	\begin{split}
		&\rd g^z\left(t,X^{z,\pi^*}_t,\bbeta^{z}_t\right)\\
		=&g^z_x(t,X^{z,\pi^*}_t,\bbeta^z_t)
		\pi^*(t)\sigma \mathrm{d}W_t+g_\beta^z(t,X^{z,\pi^*}_t,\bbeta^z_t) \sigma^{-1}\zeta(t)\mathrm{d}W_t
		\\=&g^z(t,X^{z,\pi^*}_t,\bbeta^{z}_t)
		\left[-ke^{r(T-t)}\pi^*(t)\sigma+\sigma^{-1}\zeta(t)
		\left(m_1(t)\bbeta_t^z+m_2(t)z+m_4(t)\right)
		\right]\mathrm{d}W_t
		\\=&g^z(t,X^{z,\pi^*}_t,\bbeta^{z}_t)
		\sigma^{-1}\left[\left(-A(t)+\zeta(t)m_1(t)\right)\bbeta_t^z+\left(-rA(t)+\zeta(t)m_2(t)z+\zeta(t)m_4(t)\right)
		\right]\mathrm{d}W_t.
	\end{split}
\end{equation*}
Because $\left(-A(t)+\zeta(t)m_1(t)\right)$ and $\left(-rA(t)+\zeta(t)m_2(t)z+\zeta(t)m_4(t)\right)$ are bounded on $[0,T]$, using Lemma \ref{lemma2},
$\left\{g^z(t,X^{z,\pi^*}_t,\bbeta^{z}_t)\right\}_{t\in[0,T]}$ is a martingale w.r.t. filtration $\{\mathcal{F}^W_t\}$.

\paragraph{Step 2.}
For notational simplicity, let
$$\bpi^z_u= \pi(u,X^{z,\pi}_u,\bbeta^z_u)\quad\text{ and }\quad b(t)=ke^{r(T-t)}.$$
 Let  $(t,x,\beta)\in[0,T)\times\Rbf^2$ and $\pi\in\Pi$ be fixed.
We show that there exists some $\tilde t\in(t,T)$ such that, for all $ z\in\range(Z)$, under the conditional probability $\mathbb{P}[\,\cdot\mid X^{z,\pi}_t=x,\bbeta^z_t=\beta]$,
	\begin{equation*}
	\begin{split}
	\left\{\int_t^sg^z_x(u,X^{z,\pi}_u,\bbeta^z_u) \bpi^z_u \mathrm{d}W_u\right\}_{s\in[t,\tilde t]}
	\text{ and }  \left\{\int_t^sg^z_\beta(u,X^{z,\pi}_u,\bbeta^z_u)\zeta(u)\mathrm{d}W_u\right\}_{s\in[t,\tilde t]}
	\end{split}
\end{equation*} are martingales w.r.t. filtration $\{\F_s^W\}$.
%It suffices to show that, for every $t\in[0,T)$,  there exists some $h\in(0,T-h)$ such that
%	\begin{equation*}
%	\begin{split}
%	\left\{\int_0^sg^z_x(u,X^{z,\pi}_u,\bbeta^z_u) \bpi^z_u)\mathrm{d}W_u\right\}_{s\in[t,t+h)}\text{ and }  \left\{\int_0^sg^z_\beta(u,X^{z,\pi}_u,\bbeta^z_u)\zeta(u)\mathrm{d}W_u\right\}_{s\in[t,t+h)}
%	\end{split}
%\end{equation*} are martingales w.r.t. filtration $\{\F_s^W\}$.
To this end,  we show that there exists some $\tilde t\in(t,T)$ such that, for any $z\in\range(Z)$,
\begin{equation*}
	\begin{split}
	&\mathbf{E}\left[\left.\int_t^{\tilde t}|g^z_x(u,X^{z,\pi}_u,\bbeta^z_u) \bpi^z_u|^2\mathrm{d}u\,\right|\, X^{z,\pi}_t=x,\bbeta^z_t=\beta\right]<\infty\\
	&\text{and } \mathbf{E}\left[\left.\int_t^{\tilde t} |g^z_\beta(u,X^{z,\pi}_u,\bbeta^z_u)\zeta(u)|^2\mathrm{d}u\,\right|\, X^{z,\pi}_t=x,\bbeta^z_t=\beta\right]<\infty.
	\end{split}
\end{equation*}
%Then we only need to ensure that there exists small enough $h>0$ such that
%\[
%\begin{split}
%	\sup_{s\in[t,t+h]}\left\{\mathbf{E}\left[\left. |g^z_x(s,X^{z,\pi}_s,\bbeta^z_s)\pi(s) \sigma|^2\,\right|\, X^{z,\pi}_t=x, \bbeta^z_t=\beta\right]\right\}<\infty,\\
%	\sup_{s\in[t,t+h]}\left\{\mathbf{E}\left[\left. |g^z_\beta(s,X^{z,\pi}_s,\bbeta^z_s)\zeta(s)\sigma^{-1}|^2\,\right|\, X^{z,\pi}_t=x, \bbeta^z_t=\beta\right]\right\}<\infty.
%\end{split}
%\]
%Denote $$b(t)=ke^{r(T-t)}.$$
By \eqref{eq:ans:gz}--\eqref{eq:ans:fz} and $m_3(s)<0$,
for any $\epsilon>0$, there exist $C_{1,\epsilon}>0$ and $C_{2,\epsilon}>0$ such that for any $z\in\range(Z)$ and $s\in[t,T]$,
\begin{align*}
|g^z_x(s,X^{z,\pi}_s,\bbeta^z_s)\bpi^z_s|
&= b(s)|g^z(s,X^{z,\pi}_s,\bbeta^z_s)\bpi^z_s|
\\
&\leq C_{1,\epsilon}\exp\left(-b(s)X^{z,\pi}_s+C_{2,\epsilon}{\bbeta^z_s}^2+\epsilon z^2\right)|\bpi^z_s|
\end{align*}
and
\begin{align*}
|g^z_\beta(s,X^{z,\pi}_s,\bbeta^z_s)\zeta(s)|
&=\zeta(s)\left[m_1(s)\bbeta^z_s+m_2(s)z+m_4(s)\right]|g^z(s,X^{z,\pi}_s,\bbeta^z_s)|
\\&
\leq C_{1,\epsilon}\exp\left(-b(s)X^{z,\pi}_s+C_{2,\epsilon}{\bbeta^z_s}^2+\epsilon z^2\right).
\end{align*}
Therefore,  we only need to show that, for any constant $C>0$, % and $(t,x,\beta)\in[0,T)\times\Rbf^2$,
there exists some $\tilde t\in(t,T)$ such that
$$\sup_{s\in[t,\tilde t]}\mathbf{E}\left[\left.\exp\left(-2b(s)X^{z,\pi}_s+C{\bbeta^z_s}^2\right)|\bpi^z_s|^{\rho}\,\right|\, X^{z,\pi}_t=x,\bbeta^z_t=\beta\right]<\infty, \ \rho=0,2.$$
%By \eqref{inequ02}, for any $q>0$, there exist constants $C_{1,q}$, $C_{2,q}$ such that
%\begin{equation}\label{inequ:pi}
%|\pi(t,x,\beta)|^q\leq C^q(1+|\beta|^p)^q\leq C^qC_q(1+|\beta|^{pq})
%\leq C_{1,q}(1+e^{pq|\beta|})\leq 2C_{1,q}e^{pq|\beta|}
%\leq 2C_{1,q}e^{pq\beta^2+pq}
%= C_{2,q}e^{pq\beta^2},
%\end{equation}
%where  $C_q=2^{q-1}$ if $q>1$ and $C_q=1$ if $q\leq1$.
Actually, for any $\delta>2$, by H\"{o}lder's inequality, for $\rho=0,2$,
\begin{align}
&\mathbf{E}\left[\left.\exp\left(-2b(s)X^{z,\pi}_s+C{\bbeta^z_s}^2\right)|\bpi^z_s|^{\rho}\,\right|\,X^{z,\pi}_t=x, \bbeta^z_t=\beta\right]\nonumber\\
\leq&\left(\left.\mathbf{E}\left[\exp\left(-\delta b(s)X^{z,\pi}_s\right)\,\right|\,X^{z,\pi}_t=x, \bbeta^z_t=\beta\right]\right)^{\frac{2}{\delta}}\nonumber\\
&\times\left(\mathbf{E}
\left[\left.\exp\left({C\delta\over\delta-2}{\bbeta^z_s}^2\right)|\bpi^z_s|^{\frac{\rho\delta}{\delta-2}}\,\right|\,X^{z,\pi}_t=x, \bbeta^z_t=\beta\right]\right)^{1-\frac{2}{\delta}}\label{inequ}
\\ \leq&\left(\mathbf{E}\left[\left.\exp\left(-\delta b(s)X^{z,\pi}_s\right)\,\right|\,X^{z,\pi}_t=x, \bbeta^z_t=\beta\right]\right)^{\frac{2}{\delta}}
\nonumber\\
&\times\left(\mathbf{E}\left[\left.\exp\left({2C\delta\over\delta-2}{\bbeta^z_s}^2\right)\,\right|\, \bbeta^z_t=\beta\right]
+\mathbf{E}\left[\left.|\bpi^z_s|^{\frac{2\rho\delta}{\delta-2}}\,\right|\,X^{z,\pi}_t=x, \bbeta^z_t=\beta\right]\right)^{1-\frac{2}{\delta}}.\nonumber
%\\ \leq&\left(\sup_{s\in[t,t+h]}\mathbf{E}\left[\left.\exp\left(-\delta b(s)X^{z,\pi}_s\right)\,\right|\,X^{z,\pi}_t=x, \bbeta^z_t=\beta\right]\right)^{\frac{2}{\delta}}
%\\&\times\left(\sup_{s\in[t,t+h]}\mathbf{E}\left[\left.\exp\left(2C_\delta{\bbeta^z_s}^2\right)\,\right|\, \bbeta^z_t=\beta\right]+C_1\sup_{s\in[t,t+h]}\mathbf{E}\left[\left.\exp\left(\frac{2p\rho\delta}{\delta-2}{\bbeta^z_s}^2\right)\,\right|\,\bbeta^z_t=\beta\right]\right)^{1-\frac{2}{\delta}}.
\end{align}
Then by \eqref{inequ01}--\eqref{inequ02} and Lemma \ref{lemma1}, there exist $C_1>0$ and some $\tilde t\in(t,T)$ such that for any $z\in\range(Z)$
\[
\sup_{s\in[t,\tilde t]}\mathbf{E}\left[\left.\exp\left(-2b(s)X^{z,\pi}_s+C{\bbeta^z_s}^2\right)|\bpi^z_s|^{\rho}\,\right|\,X^{z,\pi}_t=x, \bbeta^z_t=\beta\right]\le C_1e^{\epsilon z^2}.
\]
%Therefore, Assumption \ref{assu0} is verified.

\paragraph{Step 3.} Let $\pi\in\Pi$  and $(t,x,\beta)\in[0,T)\times\Rbf^2$ be fixed. We show that there exists $\tilde{t}\in(t,T)$ such that, for all $z\in\range(Z)$,
	$$\left\{\frac{1}{h}\int_t^{t+h}	\mathcal{A}^{z,\pi} g^z(s,X^{z,\pi}_s,\bbeta^z_s)\mathrm{d}s\right\}_{0<h<\tilde{t}-t}$$ is uniformly integrable under the conditional probability $\mathbb{P}[\,\cdot\mid X^{z,\pi}_t=x,\bbeta^z_t=\beta]$.
%	As
%	\[
%	\begin{split}
%	\mathbf{E}&\left[\left.|\frac{1}{h}\int_t^{t+h}	\mathcal{A}^{z,\pi} g^z(s,X^{z,\pi}_s,\bbeta^z_s)\mathrm{d}s|^{2}\,\right|\,X^{z,\pi}_t=x, \bbeta^z_t=\beta\right]\\&\leq
%	\mathbf{E}\left[\left.\frac{1}{h}\int_t^{t+h}|	\mathcal{A}^{z,\pi} g^z(s,X^{z,\pi}_s,\bbeta^z_s)|^{2}\mathrm{d}s\,\right|\,X^{z,\pi}_t=x, \bbeta^z_t=\beta\right].
%	\end{split}
%	\]
  We only need to show {that there exists some $\tilde{t}\in(t,T)$, such that for any $z\in\range(Z)$,}
  \[\sup_{s\in[t,\tilde t]}\mathbf{E}\left[\left.|\mathcal{A}^{z,\pi} g^z(s,X^{z,\pi}_s,\bbeta^z_s)|^{2}\,\right|\,X^{z,\pi}_t=x, \bbeta^z_t=\beta\right]<
  \infty.\]
%for some $\tilde{t}\in(t,T)$.
Using the closed form \eqref{eq:ans:gz}--\eqref{eq:ans:fz} of $g^z$, we have
\[
\begin{split}
	&\mathcal{A}^{z,\pi} g^z(s,X^{z,\pi}_s,\bbeta^z_s)\\
	=&g^z(s,X^{z,\pi}_s,\bbeta^{z}_s)\times Q(s,\bpi^z_s,\bbeta^{z}_s,z),
	\\=&-\frac{1}{k}\exp\left[-b(s)X^{z,\pi}_s+f^z(s,\bbeta^z_s)\right]\times Q(s,\bpi^z_s,\bbeta^{z}_s,z),
\end{split}
\]
where
\[\begin{split}
	Q(t,\pi,\beta,z)=&\frac{1}{2}{m}_1^\prime(t)\beta^2+{m}_2^\prime(t)\beta z+\frac{1}{2}{m}_3^\prime(t)z^2+{m}_4^\prime(t)\beta
	+{m}_5^\prime(t)z+{m}_6^\prime(t)\\&-b(t)\pi(z-r)+\zeta(t)[{m}_1(t)\beta+{m}_2(t)z+{m}_4(t)][z\sigma^{-2}-\beta\sigma^{-2}-b(t)\pi]\\&+\frac{1}{2}b(t)^2\sigma^2\pi^2+\frac{1}{2}\sigma^{-2}\zeta^{2}(t)[{m}_1(t)+({m}_1(t)\beta+{m}_2(t)z+{m}_4(t))^2].
\end{split}
\]
%The coefficient of $z^2$ in $Q(t,\pi,\beta,z)$ is
%\[
%\frac{1}{2}{m}_3^\prime(t)+\sigma^{-2}\zeta(t){m}_2(t)+\frac{1}{2}\sigma^{-2}\zeta^{2}(t){m}_2^2(t)=0.
%\]
$Q(t,\pi,\beta,z)$ is a quadratic form of $\pi, \beta$ and $z$ with bounded coefficients. Then, for any $\epsilon>0$, there exists $C_{1,\epsilon}>0$  such that {$\forall t\in[0,T]$ and $z\in\range(Z)$,}
\[\begin{split}|Q(t,\pi,\beta,z)|\le C_{1,\epsilon}[\beta^2+\pi^2+\frac{\epsilon}{2}z^2+1]
\le C_{1,\epsilon}\left[\exp(\beta^2+\frac{\epsilon}{2}z^2)+\pi^2\right].% \ t\in[0,T].
\end{split}\]
%Meanwhile, $f^z(t,\beta)$ is also a quadratic form of  $\beta$ and $z$. The coefficient of  $z^2$ in $f^z(t,\beta)$ is $\frac{1}{2}{m}_3(t)<0$. Then, for any $\epsilon>0$, there exist constants $C_{3,\epsilon}>0$ and $C_{4,\epsilon}>0$ such that
%\[|f^z(t,\beta)|\le C_{3,\epsilon}+C_{4,\epsilon}\beta^2+{1\over 2}\epsilon z^2, \ t\in[0,T].\]
By \eqref{eq:ans:gz}--\eqref{eq:ans:fz} and $m_3(s)<0$,
for any $\epsilon>0$, there exist $C_{2,\epsilon}>0$ and $C_{3,\epsilon}>0$ such that  $\forall s\in[t,T]$ and $z\in\range(Z)$,
\begin{align*}
	|g^z(s,X^{z,\pi}_s,\bbeta^z_s)|
\leq C_{2,\epsilon}\exp\left(-b(s)X^{z,\pi}_s+C_{3,\epsilon}{\bbeta^z_s}^2+{1\over2}\epsilon z^2\right).
\end{align*}
Then, for any $\epsilon>0$, there exist $C_{4,\epsilon}>0$ and $C_{5,\epsilon}>0$ such that $\forall s\in[t,T]$ and $z\in\range(Z)$,
\begin{align*}
|\mathcal{A}^{z,\pi} g^z(s,X^{z,\pi}_s,\bbeta^z_s)|
=&g^z(s,X^{z,\pi}_s,\bbeta^{z}_s)\times |Q(s,\bpi^z_s,\bbeta^{z}_s,z)|\\
\le&C_{4,\epsilon}\exp\left[-b(s)X^{z,\pi}_s+C_{5,\epsilon}{\bbeta^z_s}^2+\epsilon z^2\right][\pi^2(s)+1].
\end{align*}
Similarly to the derivation of \eqref{inequ}, by \eqref{inequ01}--\eqref{inequ02} and Lemma \ref{lemma1}, for any $\epsilon>0$, there exist  $C>0$ and $\tilde t\in[t,T)$  such that $\forall s\in[t,T]$ and $z\in\range(Z)$,
\[\sup_{s\in[t,\tilde t]}\mathbf{E}\left[\left.|\mathcal{A}^{z,\pi} g^z(s,X^{z,\pi}_s,\bbeta^z_s)|^{2}\,\right|\,X^{z,\pi}_t=x, \bbeta^z_t=\beta\right]\le
C e^{\epsilon z^2}<\infty.\]
%and Assumption \ref{assu1} is proved.
Moreover, we  have, $\forall h\in(0,\tilde t-t]$ {and $z\in\range(Z)$},
	\[
\begin{split}
&\left(\mathbf{E}\left[\left.\frac{1}{h}\int_t^{t+h}	\mathcal{A}^{z,\pi} g^z(u,X^{z,\pi}_u,\bbeta^{z}_u)\mathrm{d}u\,\right|\,X^{z,\pi}_t=x, \bbeta^z_t=\beta\right]\right)^2
\\&\leq \mathbf{E}\left[\left.\left|\frac{1}{h}\int_t^{t+h}	\mathcal{A}^{z,\pi} g^z(u,X^{z,\pi}_u,\bbeta^{z}_u)\mathrm{d}u\right|^{2}\,\right|\,X^{z,\pi}_t=x, \bbeta^z_t=\beta\right]
\\&\leq\mathbf{E}\left[\left.\frac{1}{h}\int_t^{t+h}|\mathcal{A}^{z,\pi} g^z(u,X^{z,\pi}_u,\bbeta^{z}_u)|^{2}\mathrm{d}u\,\right|\,X^{z,\pi}_t=x, \bbeta^z_t=\beta\right]
\\&\le C e^{\epsilon z^2}<\infty.
\end{split}
\]

\paragraph{Step 4.} Let $\pi\in\Pi$ and $(t,x,\beta)\in[0,T)\times\Rbf^2$ be fixed.  Let $\alpha\ge1$ ($\phi$ is concave). We show that there exists $\tilde{t}\in(t,T)$ such that,
	 under the conditional probability $\mathbb{P}[\,\cdot\mid \bbeta^z_t=\beta]$,	the family
	\begin{equation}\label{equ:assu2:beta:Z}
		\left\{\phi^\prime\left(g^Z(t,x,\beta)\right)\mathbf{E}\left[\left.\frac{1}{h}\int_t^{t+h}\mathcal{A}^{Z,\pi} g^Z(u,X^{Z,\pi}_u,\bbeta^{Z}_u)\mathrm{d}u\,\right|Z,X^{Z,\pi}_t=x,\bbeta^Z_t=\beta\right]\right\}_{0<h\le\tilde{t}-t}
	\end{equation}
is uniformly integrable.
By \eqref{eq:ans:gz}--\eqref{eq:ans:fz}, $\phi'(x)=(-x)^{\alpha-1}$ and $m_3(s)<0$,
for any $\epsilon>0$, there exists $C_{1,\epsilon}>0$ such that
 $$|\phi^\prime\left(g^z(t,x,\beta)\right)|\le C_{1,\epsilon} e^{\epsilon z^2}, \ z\in\range(Z).$$
From Step 3, there exist some $\tilde{t}\in(t,T)$ and $C_{2,\epsilon}>0$ such that,  $\forall h\in(0,\tilde{t}-t]$ {and $z\in\range(Z)$},
\[
\begin{split}
&\left(\phi^\prime\left(g^z(t,x,\beta)\right)\mathbf{E}\left[\left.\frac{1}{h}\int_t^{t+h}\mathcal{A}^{z,\pi} g^z(u,X^{z,\pi}_u,\bbeta^{z}_u)\mathrm{d}u\,\right|\,X^{z,\pi}_t=x, \bbeta^z_t=\beta\right]\right)^2
\\&=\left[\phi^\prime\left(g^z(t,x,\beta)\right)\right]^2
\left(\mathbf{E}\left[\left.\frac{1}{h}\int_t^{t+h}
\mathcal{A}^{z,\pi} g^z(u,X^{z,\pi}_u,\bbeta^{z}_u)\mathrm{d}u\,\right|\, X^{z,\pi}_t=x, \bbeta^z_t=\beta\right]\right)^2
\\&\leq C_{2,\epsilon} e^{3\epsilon z^2}.
\end{split}
\]
We can choose a sufficiently small  $\epsilon$ such that $6\epsilon\zeta(t)<1$. Then {by Lemma \ref{lma:exp:normal},}
\[\begin{split}
&\sup_{h\in(0,\tilde{t}-t)}\!\!\mathbf{E}\!\left[\left.\!\left(\phi^\prime\left(g^Z(t,x,\beta)\right)\mathbf{E}\left[\frac{1}{h}\int_t^{t+h}\mathcal{A}^{Z,\pi} g^Z(u,X^{Z,\pi}_u,\bbeta^Z_u)\mathrm{d}u\left|Z,\genfrac{}{}{0pt}{}{X^{Z,\pi}_t=x}{\bbeta^Z_t=\beta}\right.\right]\right)^2\,\right|\bbeta_t=\beta\right]
\\&\leq C_{2,\epsilon} \mathbf{E}\left[e^{3\epsilon Z^2}\mid\bbeta_t=\beta\right]
\\&= C_{2,\epsilon} {1 \over \sqrt{1-6\epsilon \zeta(t)}}\exp\left(3\epsilon \beta^2\over 1-6\epsilon \zeta(t)\right)
<\infty.
\end{split}
\]
Therefore, family \eqref{equ:assu2:beta:Z} is uniformly integrable under the conditional probability $\mathbb{P}[\,\cdot\mid \bbeta^z_t=\beta]$.

\paragraph{Step 5.} Let $\pi\in\Pi$ and $(t,x,\beta)\in[0,T)\times\Rbf^2$ be fixed.  Let $\alpha<1$ ($\phi$ is convex). We show that there exists $\tilde{t}\in(t,T)$ such that, under the conditional probability $\mathbb{P}[\,\cdot\mid \bbeta^z_t=\beta]$,	the family
$$\left\{\phi^\prime\left(g^{Z,\pi_{t,h}}(t,\!x,\!\beta)\right)\mathbf{E}\left[\left.\frac{1}{h}\int_t^{t+h}\mathcal{A}^{Z,\pi} g^Z(u,X^{Z,\pi}_u,\bbeta^{Z}_u)\mathrm{d}u\,\right|Z,X^{Z,\pi}_t=x,\bbeta^Z_t=\beta\right]\right\}_{0<h\le\tilde{t}-t}
$$
is uniformly integrable.	
It suffices to show that there exists $\tilde{t}\in(t,T)$ such that
\[
\sup_{h\in(0,\tilde{t}-t)}\!\mathbf{E}\!\left[\!\left(\!\phi^\prime\!\left(g^{Z,\pi_{t,h}}(t,\!x,\!\beta)\right)\mathbf{E}\left[\left.\frac{1}{h}\int_t^{t+h}\!\mathcal{A}^{Z,\pi} g^Z(u,X^{Z,\pi}_u,\bbeta^Z_u)\mathrm{d}u\left|Z,\genfrac{}{}{0pt}{}{X^{Z,\pi}_t=x}{\bbeta^Z_t=\beta}\right.\right]\right)^2\,\right|\bbeta_t=\beta\right]
\]
if finite. We consider the optimization problem
\begin{equation}\label{opt1}
\sup_\pi\mathbf{E}\left[U(X^{z, \pi}_T)\mid\mathcal{F}_t^W\right],
\end{equation}
where $\pi$ is adapted to $\{\mathcal{F}_t^W\}$. Problem \eqref{opt1} is a standard expected utility maximization problem. The optimal strategy of problem \eqref{opt1} is $\bar{\pi}(t)=\frac{z-r}{b(t)\sigma^2}$. The corresponding value function is given by
$$\bar{g}^z(t,x)=-\frac{1}{k}\exp\left[-b(t)x+0.5\sigma^{-2}(t-T)(z-r)^2\right].$$
Obviously
$$g^{z,\pi_{t,h}}(t,x,\beta)\leq \bar{g}^z(t,x).$$
By the assumption of the theorem,  $\alpha>1-0.5\sigma_0^{-2}\sigma^2T^{-1}$. Then $$2(1-\alpha)\sigma^{-2}(T-t)\zeta(t)<1, \ t\in[0,T].$$
Then we have, for any $\epsilon>0$ with
$$2\epsilon\zeta(t)+2(1-\alpha)\sigma^{-2}(T-t)\zeta(t)<1,\quad t\in[0,T],$$
 there exist $\tilde{t}\in(t,T)$ and constant $C_\epsilon>0$ such that,  $\forall h\in(0,\tilde{t}-t]$, {by Lemma \ref{lma:exp:normal},}
\[
\begin{split}
&\mathbf{E}\left[\left(\phi^\prime\left(g^{Z,\pi_{t,h}}(t,x,\beta)\right)\mathbf{E}\left[\left.\frac{1}{h}\int_t^{t+h}\mathcal{A}^{Z,\pi} g^Z(u,X^{Z,\pi}_u,\bbeta^Z_u)\mathrm{d}u\left|Z,\genfrac{}{}{0pt}{}{X^{Z,\pi}_t=x}{\bbeta^Z_t=\beta}\right.\right]\right)^2\,\right|\,\bbeta^Z_t=\beta\right]
\\&\leq C_\epsilon\mathbf{E}\left[\left.\exp\left(\epsilon Z^2 +(\alpha-1)\sigma^{-2}(t-T)(Z-r)^2\right)\,\right|\, \bbeta^Z_t=\beta\right]<\infty.
\end{split}
\]
%which is finite if $2\epsilon\zeta(t)+2(1-\alpha)\sigma^{-2}(T-t)\zeta(t)<1$. Because $\epsilon$ can be arbitrage chosen, we only require that
%$$2(1-\alpha)\sigma^{-2}(T-t)\zeta(t)<1, \ t\in[0,T],$$
%which means
%\[
%1-\alpha<\min_{t\in(0,T)}{\sigma^2+\sigma_0^2 t\over2\sigma_0^2(T-t)}=0.5\sigma_0^{-2}\sigma^2T^{-1},
%\]
%i.e.,

\paragraph{Step 6.}  Let $\pi\in\Pi$, $(t,x,\beta)\in[0,T)\times\Rbf^2$ and $z\in\range(Z)$  be fixed. We show
\begin{equation}\label{eq:lim:gz}
\lim_{h\to 0^+}\mathbf{E} \left[
g^z(t+h,X^{z,\pi}_{t+h},\bbeta^{z}_{t+h})
\mid X^{z,\pi}_t=x, \bbeta^z_t=\beta\right]=g^z(t,x,\beta).
\end{equation}
Actually, by Step 2,  there exists some $\tilde{t}\in(t,T)$ such that, under the conditional probability $\mathbb{P}[\,\cdot\mid X^{z,\pi}_t=x,\bbeta^z_t=\beta]$,
\begin{equation*}
	\begin{split}
		&\left\{\int_t^sg^z_x(u,X^{z,\pi}_u,\bbeta^z_u) \bpi^z_u \sigma\mathrm{d}W_u\right\}_{s\in[t,\tilde{t})}  \\& \text{ and } \left\{\int_t^sg^z_\beta(u,X^{z,\pi}_u,\bbeta^z_u)\zeta(u)\sigma^{-1}\mathrm{d}W_u\right\}_{s\in[t,\tilde{t})}
	\end{split}
\end{equation*} are martingales w.r.t. filtration $\{\F_s^W\}$. Then, by the Markov property, $\forall h\in(0,\tilde{t}-t)$,
	\begin{equation}\label{gequ}
	\begin{split}
	&\mathbf{E} \left[
g^z(t+h,X^{z,\pi}_{t+h},\bbeta^{z}_{t+h})
\mid X^{z,\pi}_t=x, \bbeta^z_t=\beta\right]-g^z(t,x,\beta)\\
	%=&\mathbf{E}\left[\left.g^z\left(t+h,X^{z,\pi}_{t+h},\bbeta^{z}_{t+h}\right)\,\right |\,\F^W_t\right]-g^z(t,X^{z,\pi}_t,\bbeta^z_t)\\
	=&\mathbf{E}\left[\left.\int_t^{t+h}\mathcal{A}^{z,\pi} g^z(u,X^{z,\pi}_u,\bbeta^{z}_u)\mathrm{d}u\,\right|\,X^{z,\pi}_t=x, \bbeta^z_t=\beta\right].
	\end{split}
\end{equation}
By Step 3,
\[
\mathop{\lim}\limits_{h\rightarrow 0^+}\mathbf{E}\left[\left.\frac{1}{h}\int_t^{t+h}\mathcal{A}^{z,\pi} g^z(u,X^{z,\pi}_u,\bbeta^{z}_u)\mathrm{d}u\,\right|\,X^{z,\pi}_t=x,\bbeta^{z}_t=\beta\right]
=\mathcal{A}^{z,\pi}g^z(t,x,\beta).
\]
Then
\[\mathop{\lim}\limits_{h\rightarrow 0^+}\mathbf{E}\left[\left.\int_t^{t+h}\mathcal{A}^{z,\pi} g^z(u,X^{z,\pi}_u,\bbeta^{z}_u)\mathrm{d}u\,\right|\,X^{z,\pi}_t=x,\bbeta^{z}_t=\beta\right]=0,\]
which combined with \eqref{gequ} leads to \eqref{eq:lim:gz}.
%$$\lim_{h\to 0^+}\mathbf{E} \left[
%g^z(t+h,X^{z,\pi}_{t+h},\bbeta^{z}_{t+h})
%\mid X^{z,\pi}_t=x, \bbeta^z_t=\beta\right]=g^z(t,x,\beta)$$
%and Assumption \ref{assu3} is valid.
\qed
\end{appendices}

\paragraph{Acknowledgements.} The work is supported by the National Key R\&D Program of China
({2020YFA0712700}) and
 the National Natural Science Foundation of China (11901574, 11871036, 12271290, {12071146}).
The authors thank the members of the group of Mathematical Finance and Actuarial Science at the Department of Mathematical Sciences, Tsinghua University for their helpful feedbacks and conversations.

\bibliographystyle{abbrvnat}
\bibliography{reference}
\end{document}